%% file: kopf.tex
\documentclass[a4paper]{amsart}
\usepackage{amscd}
\usepackage{amssymb}
\usepackage{color}
\usepackage[T1]{fontenc} 
\usepackage[utf8]{inputenc}
\usepackage{hyperref}
\usepackage[arrow,curve,matrix]{xy}
\usepackage{comment}
\usepackage{enumerate}
\usepackage{caption}
\usepackage{extarrows}

\usepackage{graphicx}
\usepackage{epstopdf}

\usepackage{mathpazo}
\usepackage{mathrsfs}
\usepackage{MnSymbol}{

\usepackage{graphicx}
\usepackage{color}

\usepackage{tikz}
\usetikzlibrary{matrix,arrows,decorations.pathmorphing,decorations.markings,calc}
\usepackage{stackrel}

\usepackage{MnSymbol}
\usepackage{enumitem}

\unitlength 1cm

\setcounter{tocdepth}{1}

\hypersetup{colorlinks}

\sloppy
\allowdisplaybreaks

\DeclareFontFamily{OMS}{rsfs}{\skewchar\font'60}
\DeclareFontShape{OMS}{rsfs}{m}{n}{<-5>rsfs5 <5-7>rsfs7 <7->rsfs10 }{}
\DeclareSymbolFont{rsfs}{OMS}{rsfs}{m}{n}
\DeclareSymbolFontAlphabet{\scr}{rsfs}

\newcommand{\sF}{\scr{F}}
\newcommand{\sG}{\scr{G}}

\newcommand{\sI}{\scr{I}}

\newcommand{\sM}{\scr{M}}

\newcommand{\sO}{\scr{O}}

\newcommand{\bC}{\mathbb{C}}

\newcommand{\bP}{\mathbb{P}}
\newcommand{\bQ}{\mathbb{Q}}

\newcommand{\bS}{\mathbb{S}}

\newcommand{\bZ}{\mathbb{Z}}

\newcommand{\h}{h}
\newcommand{\LB}{\mathcal{L}}
\newcommand{\IH}{IH}

\newcommand{\Homom}{\text{Hom}}
\newcommand{\sHom}{\scr{H}\negmedspace om}
\newcommand{\sExt}{\scr{E}\negmedspace xt}
\newcommand{\sm}{\text{sm}}
\newcommand{\reso}[1]{\tilde{#1}}
\newcommand{\red}{\text{red}}
\newcommand{\reg}{\text{reg}}
\newcommand{\sing}{\text{sing}}
\newcommand{\tor}{\text{tor}}
\newcommand{\LieDer}{\mathcal{L}}
\newcommand{\refl}{\text{refl}}
\newcommand{\an}{\text{an}}
\newcommand{\pr}{\text{pr}}
\newcommand{\id}{\text{id}}
\newcommand{\res}{\text{res}}
\newcommand{\singu}{\text{sing}}

\newcommand{\sg}[2]{\ensuremath{{#1}_{#2}}}

\DeclareMathOperator{\Hom}{Hom}

\theoremstyle{plain}    \newtheorem{thm}{Theorem}[section]
\newtheorem{defn}[thm]{Definition}
 
\numberwithin{equation}{thm}
\numberwithin{figure}{section}
\theoremstyle{plain}    
\newtheorem{cor}[thm]{Corollary}
\newtheorem{lem}[thm]{Lemma}
\newtheorem{fact}[thm]{Fact}

\newtheorem{plainclaim}[thm]{Claim}
\newtheorem{plainass}[thm]{Assumption}

\theoremstyle{plain}    
\newtheorem{prop}[thm]{Proposition}
\theoremstyle{remark}
\newtheorem{rem}[thm]{Remark}
\newtheorem{claim}[equation]{Claim} %%Delete [...] to re-start numbering
\newtheorem{notation}[thm]{Notation}
\newtheorem{example}[thm]{Example}

\newtheorem{const}[thm]{Construction}

\newtheorem{add-ass}[equation]{Additional Assumption}
\newtheorem{obs}[equation]{Observation}

\definecolor{tomato}{RGB}{180,62,39}
\definecolor{forrest}{RGB}{81,133,49}
\definecolor{lighttomato}{RGB}{253,65,65}
\definecolor{lightforrest}{RGB}{145,237,87}
\definecolor{mygreen}{RGB}{40,104,69}
\definecolor{mygreen2}{RGB}{3,149,39}
\definecolor{darkolivegreen}{RGB}{102,118,75}
\definecolor{cranegreen}{RGB}{102,118,75}
\definecolor{mydarkblue}{RGB}{10,92,153}
\definecolor{myblue}{RGB}{57,222,186}
\definecolor{pinkish}{RGB}{213,83,222}
\definecolor{colD}{RGB}{213,83,222}
\definecolor{defb}{RGB}{213,83,222}
\definecolor{goldenrod}{RGB}{225,115,69}
\definecolor{mauve}{RGB}{224, 176, 255}
\definecolor{fuchsia}{RGB}{255, 0, 255}
\definecolor{lavender}{RGB}{230, 230, 250}
\definecolor{gold}{RGB}{255, 215, 0}
\definecolor{orange}{RGB}{255, 127, 0}
\definecolor{maroon}{RGB}{123, 17, 19}
\definecolor{brightmaroon}{RGB}{195, 33, 72}
\definecolor{richmaroon}{RGB}{176, 48, 96}
\definecolor{green}{RGB}{3,149,39}

\date{\text{January 29, 2014}}

\author{Clemens J\"order}
\address{Clemens J\"order, Mathematisches Institut, Albert-Ludwigs-Universit\"at
  Freiburg, Eckerstraße 1, 79104 Freiburg im Breisgau, Germany}
\email{\href{mailto:c.joerder@web.de}{c.joerder@web.de}}

\thanks{The author was supported in part by the DFG-Forschergruppe 790
  ``Classification of Algebraic Surfaces and Compact Complex Manifolds''.}


\title{On the Poincar\'e Lemma for reflexive differential forms}

\begin{document}

\begin{abstract}
Let $X$ be a normal complex space and let $\Omega^{[i]}_{X,p}:=\bigl(\Omega^i_X\bigr)^{**}_p$ be the stalk of the sheaf of reflexive differential forms at $p\in X$. First, we show that the de Rham complex of reflexive differential forms $\,\cdots\xrightarrow{\text{d}}\Omega^{[i]}_{X,p}\xrightarrow{\text{d}}\Omega^{[i+1]}_{X,p}\xrightarrow{\text{d}}\cdots\,$ is exact in degree $i=1$ under suitable topological conditions, but that exactness in general depends on the complex structure. Second, we show exactness in high degrees for holomorphically contractible $X$ under mild assumptions on the nature of singularities of $X$, e.g. klt singularities.

Subsequently, the exactness of the de Rham complex of reflexive differential forms is related to the Lipman-Zariski conjecture and the failure of vanishing theorems of Kodaira-Akizuki-Nakano type on singular spaces.
\end{abstract}

\maketitle
\tableofcontents
\input{intro}
\input{diff-forms}

\input{topology}

\input{contractions}

\input{degeneration}

\input{lipman-zariski}
\input{kan}

\bibliography{bibliography/general}{}
\bibliographystyle{bibliography/skalpha}

\end{document}

%% file: intro.tex
\section{Introduction}\label{sec-intro}
The classical Poincar\'e Lemma states that the de Rham complex of sheaves of holomorphic differential forms on a complex manifold $M$ of dimension $n$ is a resolution
\[0\to\bC_M\to \sO_M\xrightarrow{\text{d}}\Omega^1_M\xrightarrow{\text{d}}\cdots\xrightarrow{\text{d}}\Omega^n_M\to 0 \]
of the sheaf $\bC_M$ of locally constant complex-valued functions. Via the Fr\"olicher spectral sequence it relates the complex singular cohomology of $M$ with the cohomology groups of the coherent sheaves $\Omega^i_M$. In this way the Poincar\'e Lemma can be regarded as a cornerstone of the Hodge theory of projective complex manifolds.

In the presence of singularities, the above picture breaks down completely. Indeed, none of the coherent sheaves of differential forms known in the literature satisfies the Poincar\'e Lemma. In this paper, we discuss the case of \emph{reflexive differential forms} on a normal complex space $X$, i.e., the sheaves $\Omega^{[i]}_X=\bigl(\Omega^i_X\bigr)^{**}$. More precisely, for any $p\in X$, we ask: What is the meaning of the cohomology groups of the complex
\[(\star)\quad\quad\quad 0\to\bC\to\sO_{X,p}\xrightarrow{\text{d}}\Omega^{[1]}_{X,p}\xrightarrow{\text{d}}\cdots\xrightarrow{\text{d}}\Omega^{[i]}_{X,p}\xrightarrow{\text{d}}\cdots\xrightarrow{\text{d}}\Omega^{[n]}_{X,p}\to 0? \]
We will see that exactness of $(\star)$ is closely related to various notions and results concerning the complex space $X$: the local topology, holomorphic contractibility, vanishing theorems and the Lipman-Zariski conjecture.

\subsection*{Low degrees and the topology of $X$.} The exactness of $(\star)$ depends a priori on the complex structure overlying the topological space $X$. In this spirit the complex structure is taken into account by any result so far obtained in the literature. Indeed, exactness in degree $i$ has been proven for
\begin{itemize}
 \item \emph{isolated rational singularities} if $i=1,2$, \cite[Prop.~2.5]{CF02}, 
 \item locally algebraic \emph{klt base spaces} if $i=1$, \cite[Thm.~5.4]{GKP12},
 \item \emph{toroidal singularities} and arbitrary $1\leq i\leq n$, \cite[Prop.~3.14]{Dan78}, and
 \item \emph{isolated complete intersection singularities} if $i\leq n-2$, \cite[Sect.~4]{G75}.
\end{itemize}
Although our first main Theorem~\ref{thm-top-poincare} does not require a deep proof, it clarifies the situation in degree $i=1$ by giving a sufficient, purely topological criterion for exactness, which covers all results mentioned above. Far better, its formulation only involves the \emph{first rational local intersection cohomology} $\IH^1_{\text{loc}}(p\in X,\bQ)$, see Definition~\ref{defn-local-intersection-coho}.

\begin{thm}[Topological Poincar\'e Lemma in degree one, Section~\ref{ssec-top-poincare}]\label{thm-top-poincare}
Let $p\in X$ be a normal, locally algebraic complex space singularity. Then
\[ \IH^1_{\text{loc}}(p\in X,\bQ)=0\quad\implies\quad (\star)\, \text{ is exact in degree } i=1.\]
\end{thm}

Does this topological approach admit a generalization to higher degrees? The following proposition gives a twofold negative answer if $n=2$: First, the vanishing of rational intersection cohomology is no longer a sufficient criterion. Second, exactness of $(\star)$ \emph{does} effectively depend on the complex structure. 

\begin{prop}[Dependancy on the complex structure, Section~\ref{ssec-dependancy-hol-structure}]\label{intro-ex-top-counterexample}
There exist two minimally elliptic normal surface singularities $p_1\in X_1$ and $p_2\in X_2$ such that
\begin{enumerate}
 \item\label{it-top-counterexample-homeo} $X^\text{top}_1\cong X^\text{top}_2$, i.e., the underlying topological spaces are homeomorphic,
 \item\label{it-top-counterexample-int-coho} $\IH^k_{\text{loc}}(p_i\in X_i,\bQ)=0$ for $k>0$ and $i=1,2$, and
 \item\label{it-top-counterexample-poincare} the complex $(\star)$ is exact for $X=X_1$, but not for $X=X_2$.
\end{enumerate}
\end{prop}

Fortunately, the proof of Proposition~\ref{intro-ex-top-counterexample} contains an elucidating \emph{geometric} explanation for the phenomenon: the complex space germ $p_1\in X_1$ is quasihomogeneous while $p_2\in X_2$ is not. The notion of quasihomogeneity is recalled in Definition~\ref{defn-quasihom-sing}. This observation leads us to our second topic.\newline

\subsection*{High degrees and holomorphic contractibility.} By a result of Gilmartin~\cite{Gil64} any complex space is locally topologically contractible. In contrast, the existence of a \emph{holomorphic} contraction map as in Definition~\ref{defn-contraction} below is a strong condition on the complex structure of $X$. For holomorphically contractible complex spaces, the Poincar\'e Lemma has been settled for
\begin{itemize}
 \item \emph{K\"ahler differential forms},~\cite{Rei67}, and
 \item \emph{K\"ahler differential form modulo torsion},~\cite{Ferr70}.
\end{itemize}
In the setup of reflexive differential forms, holomorphic contractibility does \emph{not} imply exactness of $(\star)$ as has been observed in~\cite[Rem.~5.4.2]{GKP12}.

However, we establish at least partial results in this direction. Our approach consists of two steps: First, we turn away from reflexive differential forms and examine another important class of differential forms, namely the sheaves $\Omega^i_\h|_X$ of \emph{$\h$-differential forms}. The letter $\h$ refers to the $\h$-Grothendieck topology on the category of schemes introduced by Voevodsky in~\cite[Def. 3.1.1]{Voe96} in his study of the homology of schemes, see Section~\ref{ssec-h-diff-forms} for details.

The class of $\h$-differential forms is the closest to reflexive differential forms that still admits a pull-back map by holomorphic maps. It is exactly this technical advantage that enables us to establish an analog of Reiffen's and Ferrari's results. The easiest version can be formulated as follows.

\begin{thm}[Poincar\'e Lemma for $\h$-differential forms on holomorphically contractible spaces, Section~\ref{ssec-poincare-contractible}]\label{thm-intro-contractions-h-forms}
Let $p\in X$ be a point on a reduced complex space of dimension $n$. If the space germ $\sg{X}{p}$ is holomorphically contractible to a subspace $Y\subset X$ of dimension $m$, then the sequence
\[\Omega^{m}_\h|_{X,p}\xrightarrow{\text{d}}\cdots\xrightarrow{\text{d}}\Omega^{n}_\h|_{X,p}\to 0\]
of stalks of sheaves of $\h$-differential forms is exact.
\end{thm}

Second, we show that for some types of singularities, the sheaves of $\h$-differential forms and reflexive differential forms agree. This yields the following version of Theorem~\ref{thm-intro-contractions-h-forms}. Locally algebraic klt base spaces are introduced in Definition~\ref{defn-klt-base-space}.

\begin{cor}[Poincar\'e Lemma on holomorphically contractible spaces, Section~\ref{ssec-poincare-contractible}]\label{cor-poincare-klt-isol-rat}
Suppose that, in addition to the assumptions in Theorem~\ref{thm-intro-contractions-h-forms}, $X$ is a locally algebraic klt base space, or that $p\in X$ is an isolated rational singularity.

Then the reflexive de Rham complex $(\star)$ is exact in degrees $i>m$.
\end{cor}

Theorem~\ref{thm-top-poincare} and Corollary~\ref{cor-poincare-klt-isol-rat} are far from covering all reasons for the exactness of the de Rham complex of reflexive differential forms. However, for Gorenstein normal surface singularities, there is a complete characterization in terms of holomorphic contractibility and topological properties:

\begin{prop}[Complete characterization for Gorenstein surfaces, Section~\ref{ssec-gorenstein-surface}]\label{intro-prop-gorenstein-surface}
Let $p\in X$ be a Gorenstein normal surface singularity . Then
\[{\textstyle (\star)\text{ is exact}\Leftrightarrow p\in X\text{ is quasihomogeneous and } \IH^1_{\text{loc}}(p\in X,\bQ)=0.} \]
\end{prop}

Quasihomogeneous complex space germs $\sg{X}{p}$ are holomorphically contractible to $\{p\}\subset X$ by Example~\ref{ex-action-contraction}.

\subsection*{Relation to other topics of interest} In the following we exhibit three examples illustrating the close relation between the exactness of $(\star)$ and many other questions of interest.

\subsubsection*{Global topology of $X$.} In this section let us \emph{assume} that the Poincar\'e Lemma holds for reflexive differential forms on $X$, i.e., the complex $(\star)$ is exact for any $p\in X$. Recall that the assumption holds if $X$ is a surface with rational singularities by \cite[Prop.~2.5]{CF02}.

The Frölicher spectral sequence known from the theory of complex manifolds has an analog
\[(\star\star)\quad\quad\quad E^{i,j}_1=H^j(X,\Omega^{[i]}_X)\implies H^{i+j}(X,\bC)\]
relating the cohomology of the sheaves of reflexive differential forms and the global topology of $X$. Since degeneration of $(\star\star)$ at $E_1$ has been proved in~\cite[Thm.~12.5]{Dan78} for normal toric projective varieties, it seems natural to hope for degeneration in the new cases established e.g. in Corollary~\ref{cor-poincare-klt-isol-rat}. The following proposition dashes these hopes starting from dimension three, even under very strong assumptions on the local nature of the singularities of $X$.

\begin{prop}[Degeneration of $(\star\star)$, Section~\ref{sec-degeneration}]\label{intro-ex-degeneration} The spectral sequence satisfies the following:
\begin{enumerate}[label={(\alph*)}] 
 \item\label{it-intro-ex-degeneration-surface} If $X$ is a projective surface with rational singularities, then $(\star\star)$ degenerates at $E_1$.
 \item\label{it-intro-ex-degeneration-threefold} There exists a projective three-dimensional complex space with only one quasihomogeneous terminal hypersurface singularity such that $(\star\star)$ does not degenerate at $E_1$.
\end{enumerate}
\end{prop}

\subsubsection*{On the Lipman-Zariski conjecture.} Let $V$ be an algebraic variety over a field of characteristic zero and suppose that the tangent sheaf $T_V$ is locally free in a neighborhood of some point $p\in V$. \emph{Is $p\in V$ a smooth point?} Motivated by his advisor Zariski, Lipman was the first to approach this question in~\cite{Lip65}. Since then a positive answer has been found in numerous cases. In our context the most noteworthy is the case of a quasihomogeneous singularity $p\in V$ settled by Hochster in~\cite{Hoch77}.

The proof of the following result illustrates how exactness properties of the complex $(\star)$ can be used to expand Hochster's result to slightly more general local $\bC^*$-actions.

\begin{cor}[Lipman-Zariski conjecture on contractible spaces, Section~\ref{sec-lz}]\label{cor-lz}
Let $X$ be a normal complex space and $p\in X$. Suppose that the space germ $\sg{X}{p}$ admits a holomorphic $\bC^*$-action with only non-negative weights such that the fixed point locus $X^{\bC^*}$ is a curve not contained in $X_\sing$.

Then the Lipman-Zariski conjecture holds at $p\in X$.
\end{cor}

The definition of a $\bC^*$-action on a space germ and its weights is given in Section~\ref{ssec-local-c-star-action}. Notice that in Corollary~\ref{cor-lz} the complex space $X$ is not assumed to be locally algebraic, whereas quasihomogeneous singularities as considered by Hochster are automatically algebraic, see Fact~\ref{fact-quasihom-algebraizity}.

\subsubsection*{On Kodaira-Akizuki-Nakano type vanishing.} The classical Kodaira-Akizuki-Nakano vanishing theorem states that if $\LB$ is an ample line bundle on a projective complex manifold $X$, then $H^j(X,\Omega^i_X\otimes\LB^{-1})=0$ for $i+j<\dim(X)$. Seeking to generalize these results, we ask whether the groups
\[H^j(X,\Omega^{[i]}_X\otimes\LB^{-1}) \stackbin[]{?}{=} 0\]
vanish if $i+j<\dim(X)$ and $X$ is a normal projective complex space. Vanishing has been proven in~\cite[Prop.~4.3]{GKP12} for $j\leq 1$ if $X$ has mild singularities. However, the same authors construct in~\cite[Prop.~4.8]{GKP12} a projective complex space $X$ with only one four-dimensional isolated rational singularity and $H^2(X,\Omega^{[1]}_X\otimes\LB^{-1})\neq 0$.

The following result is the easiest version of Theorem~\ref{thm-kan-poincare} relating these considerations to our topic.

\begin{thm}[Section~\ref{sec-kan}]\label{thm-poincare-vs-kan}
Let $X$ be a projective complex space of dimension $\geq 4$ with only one isolated rational singularity $p\in X$, and let $\LB$ be an ample line bundle on $X$. Then
\[{\textstyle \dim_\bC\,H^2(X,\Omega^{[1]}_X\otimes\LB^{-1})\,\geq \, \dim_\bQ\text{WDiv}_\bQ(\sg{X}{p})/\hspace{-0.2em}\sim_\bQ\, + \, \, \dim_\bC\,h^3(\star),} \]
where $\text{WDiv}_\bQ(\sg{X}{p})/\hspace{-0.4em}\sim_\bQ$ denotes the group of local analytic Weil divisors with rational coefficients on arbitrarily small neighborhoods of $p\in X$ modulo $\bQ$-linear equivalence.
\end{thm}

The counterexample given in~\cite{GKP12} only explores the first contribution on the right hand side. In fact, the singularity $p\in X$ in \emph{loc. cit.} is rational and quasihomogeneous so that the complex $(\star)$ is exact by Corollary~\ref{cor-poincare-klt-isol-rat}.

\subsection{Outline of the paper.} The paper is structured as follows: The current Section~\ref{sec-intro} contains the introduction. 

In Section~\ref{sec-diff-forms} we set up notation and recall classical notions of differential forms on reduced complex spaces. Then we include a construction of $\h$-differential forms on singular complex spaces and prove elementary properties that are used in the sequel.

The core results of the paper are contained in Sections~\ref{sec-top} and~\ref{sec-contr}. Section~\ref{sec-top} is concerned with results based on topological properties of $X$ as well as a closer look at the surface case. Section~\ref{sec-contr} contains a unified proof of Reiffen's and Ferrari's results and Theorem~\ref{thm-intro-contractions-h-forms}.

Section~\ref{sec-degeneration} contains a discussion of the degeneration properties of the reflexive analog of the Fr\"olicher spectral sequence. The relation to the Lipman-Zariski conjecture is exposed in Section~\ref{sec-lz}. Finally, Section~\ref{sec-kan} contains the proofs of the results on Kodaira-Akizuki-Nakano type cohomology groups mentioned in the introduction.

\subsection{Acknowledgments.} The results of this paper constitute the author's PhD thesis. He would like to thank his advisor Stefan Kebekus and his co-advisor Daniel Greb for stimulating discussions leading to the questions treated in this paper and for fruitful advice during the research. He would also like to thank Annette Huber-Klawitter, Patrick Graf, Tian Shun-Feng, Alex K\"uronya, Thomas Peternell, Hubert Flenner and Wolfgang Soergel for interesting discussions.

%% file: diff-forms.tex
\section{Complex spaces and differential forms}\label{sec-diff-forms}

\subsection{Notation}The base field is the field of complex numbers $\bC$.

\subsubsection{Schemes and complex spaces}
We will switch frequently between the algebraic and the analytic setting. A scheme is a scheme of finite type over $\text{Spec}(\bC)$ and is usually denoted by $V$ or $W$. A variety is a separated irreducible reduced scheme. Complex spaces are usually denoted by $X,Y,Z$. A complex manifold $M$ is a smooth complex space. The dimension of a complex space or a scheme is considered as a function with values in $\bZ_{\geq 0}$. Likewise the codimension of a complex subspace or a subscheme is a function defined on the subspace.

Let $p\in X$ be a point on a complex space. Then the complex space germ is denoted by $\sg{X}{p}$.

For any algebraic morphism $f:V\to W$ between schemes the corresponding holomorphic map between associated complex spaces is denoted by $f^\an:V^\an\to W^\an$. 

\begin{defn}\label{defn-locally-algebraic}
We say that a complex space $X$ is \emph{locally algebraic} if there exists a covering $X=\bigcup_{i\in I}X_i$ by open subsets $X_i\subset X$ and schemes $V_i$ together with an open embedding $X_i\subset V^\an_i$ for any $i\in I$.
\end{defn}

\subsubsection{Sheaves}
For any morphism $\phi:\sF\to \sG$ between coherent sheaves of $\sO_V$-modules on a scheme $V$, the associated morphism between analytically coherent sheaves of $\sO_{V^\an}$-modules is denoted by $\phi^\an:\sF^\an\to\sG^\an$.

If $\sF$ is a coherent sheaf on a reduced complex space or a scheme, then we denote by $\sF_\tor\subset\sF$ the torsion subsheaf. Recall from~\cite[Anhang~§4.4]{GR71} that $\sF_\tor=\ker(\sF\to\sF^{**})$ is a coherent sheaf, whose local sections are exactly the local sections of $\sF$ with nowhere dense support. The quotient will be denoted by $\sF/\tor:=\sF/\sF_\tor$. Recall that $(\sF_\tor)^\an=(\sF^\an)_\tor=:\sF^\an_\tor$ for a coherent sheaf $\sF$ on a reduced scheme.

\subsubsection{Singularities, resolutions of singularities}
The reduced complex space associated with $X$ is denoted by $X_\red$. The smooth and the singular locus of a reduced complex space $X$ are denoted by $X_{\text{sm}}\subset X$ and $X_\singu\subset X$, respectively. 

A resolution of singularities of a reduced complex space $X$ is a proper surjective morphism $\pi:\tilde{X}\to X$ such that $\tilde{X}$ is smooth and there exists a nowhere dense analytic subset $A\subset X$ such that $\pi^{-1}(A)\subset\tilde{X}$ is nowhere dense and  $\pi:\pi^{-1}(X\backslash A)\xrightarrow{\sim}X\backslash A$ is an isomorphism. The morphism $\pi$ is called a small resolution if $A$ can be chosen such that $\pi^{-1}(A)\subset \tilde{X}$ has codimension $\geq 2$. 
The morphism $\pi$ is called a strong resolution if we can choose $A=X_\singu$ and the reduced preimage $\pi^{-1}(A)_\red\subset \tilde{X}$ is a divisor with simple normal crossings. By a functorial resolution we mean a strong resolution $\pi:\tilde{X}\to X$ such that for any open set $U\subset X$ and any vector field $V\in T_X(U)$ there exists a vector field $\tilde{V}\in T_{\tilde{X}}(\tilde{U})$ on the preimage $\tilde{U}=\pi^{-1}(U)$ that is $\pi$-related to $V$. Recall from~\cite[Thm.~3.36]{Koll07} that functorial resolutions exist for any reduced complex space.

With 'complex space' replaced by 'scheme' the previous definitions apply verbatim to the algebraic setting.

We will need several times the following fact. It is a corollary of a result by Lojasiewicz~\cite[Thm.~2,\,3]{Loj64}, as explained in the proof of~\cite[Lem.~14.4]{GKP12}.

\begin{fact}[Topology of resolutions]\label{fact-top-reso}
Let $\pi:\tilde{X}\to X$ be a resolution of a reduced complex space and let $F:=\pi^{-1}(\{p\})_\red$ be the reduced fiber over some point $p\in X$. Then there exist arbitrarily small contractible neighborhoods $U\subset X$ of $p$ such that the inclusion $F\to \pi^{-1}(U)$ is a homotopy-equivalence.
\end{fact}

For the definition of pairs with Kawamata log terminal singularities and rational singularities we refer to \cite{KM98}. We will also need the following definition.

\begin{defn}[Klt base spaces, see {\cite[Def.~5.1]{Keb12}}]\label{defn-klt-base-space}
Let $V$ be a normal variety. We call $V$ a \emph{klt base space} if there exists a $\bQ$-divisor $D$ such that the pair $(X,D)$ has Kawamata log terminal singularities.

A complex space $X$ is called a \emph{locally algebraic klt complex base space} if there exists a cover $X=\bigcup_{i\in I}X_i$ by open subsets and, for any $i\in I$, an algebraic klt base space $V_i$ together with an open embedding $X_i\subset V^\an_i$.
\end{defn}

\begin{example}
Let $V$ be a normal toric variety. Then $V^\an$ is a locally algebraic klt base space by \cite[Ex.~11.4.26]{CLS11}.
\end{example}

We will not need the definition of Du Bois singularities but the following fact.

\begin{fact}[Isolated Du Bois singularities, {\cite[Sect.~0]{Kov99}}]\label{fact-isolated-du-bois}
Let $p\in X$ be a normal isolated singularity of a complex space $X$ together with a strong resolution $\pi:\tilde{X}\to X$, $E=\pi^{-1}(\{p\})_\red$. Then the following are equivalent.
\begin{enumerate}
\item The singularity $p\in X$ is Du Bois.
\item For any $i>0$, the map $R^i\pi_*\sO_{\tilde{X}}\xrightarrow{\sim}H^i(E,\sO_E)$ is an isomorphism.
\end{enumerate}
\end{fact}

\subsection{Holomorphic $\bC^*$-actions on space germs}\label{ssec-local-c-star-action}

\begin{defn}[Holomorphic $\bC^*$-action on a space germ]\label{defn-loc-hol-c-star-action}
Let $p\in X$ be a point on a complex space. A \emph{holomorphic $\bC^*$-action on the space germ $\sg{X}{p}$} consists of a holomorphic map $\bC^*\times X\supset U\to X,(t,x)\mapsto t\cdot x$ defined on an open neighborhood $U$ of $\bC^*\times\{p\}$ such that $t\cdot p=p$ for all $t\in\bC^*$, $1\cdot x=x$ and $t\cdot(s\cdot x)=(st)\cdot x$ for all $x\in X$, $s,t\in\bC^*$, whenever this makes sense.
\end{defn}

\begin{example}\label{ex-linear-c-star-action}
Let $\bC^*$ act on $\bC^r$ by $t\cdot (x_1,\cdots,x_r)=(t^{z_1}x_1,\cdots, t^{z_r}x_r)$, where $z_1,\cdots,z_r\in\bZ$ are integers. Let $0\in X\subset\bC^r$ be a locally closed analytic subset such that the vector field $z_1\cdot x_1\cdot \frac{\partial}{\partial x_1}+\cdots+z_n\cdot x_n\cdot \frac{\partial}{\partial x_n}$ is tangent to $X$. Then the space germ $\sg{X}{0}$ inherits a holomorphic $\bC^*$-action.
\end{example}

Let $p\in X$ be as in Definition~\ref{defn-loc-hol-c-star-action}. By compactness of $\bS^1$ there exists an open neighborhood $U\subset X$ of $p$ on which $\bS^1$ acts by biholomorphic automorphisms. In particular, the cotangent space is a direct sum  of eigenspaces $(m_p/m_p^2)_z=\{v\in m_p/m_p^2:\, t\cdot v = t^zv\,\forall t\in \bS^1\}$. The \emph{weights} of the action at $p\in X$ are the integers $z\in\bZ$ such that $(m_p/m_p^2)_z\neq \{0\}$.

\begin{lem}[Linearization of $\bC^*$-actions]\label{lem-c-star-linearization}
Any holomorphic $\bC^*$-action on a space germ is isomorphic to a holomorphic $\bC^*$-action as in Example~\ref{ex-linear-c-star-action}.
\end{lem}

\begin{proof}
Let $\iota:\sg{X}{p}\to T_pX$ be an arbitrary local embedding of the space germ $\sg{X}{p}$ into the tangent space at $p$ such that the induced map $d_p\iota:T_pX\to T_0T_pX$ is the canonical map. Observe that $\bS^1\subset\bC^*$ acts on the tangent space so that we can define
\[{\textstyle \overline{\iota}:\sg{X}{p}\to T_pX,\quad \overline{\iota}(x)=\frac{1}{2\pi}\int_{s=0}^{2\pi}e^{-s\cdot i}\cdot\iota(e^{s\cdot i}\cdot x)\text{d}s.}\]
It is easy to check that $\overline{\iota}(g\cdot x)=g\cdot \overline{\iota}(x)$ for all $(g,x)$ in a neighbourhood of $\bS^1\times\{p\}$ in $\bS^1\times X$. Since $\overline{\iota}$ is holomorphic, this even holds for all $(g,x)$ in a neighbourhood of $\bC^*\times\{p\}$ in $\bC^*\times X$. This shows that $\overline{\iota}$ is the desired equivariant local embedding, since $d_p\overline{\iota}:T_pX\to T_0T_pX$ is again the canonical map.
\end{proof}

\begin{defn}[Quasihomogeneous singularities]\label{defn-quasihom-sing}
A point $p\in X$ on a complex space is said to be a \emph{quasihomogeneous} singularity, if the space germ $\sg{X}{p}$ admits a holomorphic $\bC^*$-action with only positive weights.
\end{defn}

\begin{fact}[Algebraizity of quasihomogeneous singularities,~{\cite[Sect.~9.B]{Loo84}}]\label{fact-quasihom-algebraizity}
Let $p\in X$ be a quasihomogeneous singularity on a complex space. Then there exists an affine complex scheme $V$ together with an algebraic action of $\bC^*$ with fixed point $p\in V$ and only positive weights on $T_pX$, together with an equivariant isomorphism $\sg{X}{p}\cong\sg{V}{p}^\an$ of complex space germs.
\end{fact}

\subsection{Classical differential forms}\label{ssec-class-diff-forms}
For a reduced complex space $X$ we consider the following sheaves of differential forms. The notations apply verbatim to the case of a reduced scheme $V$. 
\begin{itemize}
\item $\Omega^i_X$ - the sheaf of K\"ahler differential forms of degree $i\geq 0$ on $X$.
\item $\Omega^i_X/\tor$ - the sheaf of K\"ahler differential forms modulo torsion of degree $i\geq 0$ on $X$.
\item If $X$ is normal in addition, we consider the sheaf $\Omega^{[i]}_X$ of reflexive differential forms. It satisfies $\Omega^{[i]}_X=j_*\Omega^i_{X_\sm}=(\Omega^i_X)^{**}$, where $j:X_\sm\subset X$ is the inclusion of the smooth locus.
\end{itemize}

The sheaves $\Omega^i_X$ and $\Omega^i_X/\tor$ are constructed as quotients of $\Omega^i_{\bC^m}$ for some local embedding $X\subset\bC^m$. From this point of view they seem rather inconvenient for purposes of birational geometry. In fact, given a resolution of singularities $\tilde{X}\to X$, it seems difficult to determine which differential forms on $\tilde{X}$ are the pull-back of a K\"ahler differential form on $X$.

On the other hand, the sheaves $\Omega^{[i]}_X$ play an essential role in the classification of singularities arising in the minimal model program in birational geometry, see~\cite{KM98} for a thorough discussion. For log canonical singularities, they admit an easy description in terms of differential forms on a resolution, see~\cite[Thm.~2.12]{GKP12}.

\subsubsection*{Pull-back properties of classical differential forms}\label{ssec-pull-back-properties-classical}
K\"ahler differential forms come up with a pull-back map associated with any holomorphic map. This property still holds for the sheaves of K\"ahler differential forms modulo torsion by a result of Ferrari. The proof of the algebraic version is due to Kebekus.

\begin{fact}[{\cite[Prop.~1.1]{Ferr70},~\cite[Sect.~2.2]{Keb12}}]
Let $f:X\to Y$ be a holomorphic map between reduced complex spaces (or a morphism between reduced schemes). Then the pull-back by $f^*$ maps torsion K\"ahler differential forms on $Y$ to torsion K\"ahler differential forms on $X$ and thus induces a pull-back map $f^*:\Omega^i_Y/\tor\to f_*(\Omega^i_X/\tor)$ that fits into a commutative diagram
\[\begin{array}{l}\xymatrix{\Omega^i_Y/\tor \ar[r]^{f^*} & f_*\Omega^i_X/\tor \\ \Omega^i_Y \ar[r]^{f^*} \ar[u]^{\text{quotient}} & f_*\Omega^i_X \ar[u]_{\text{quotient}}. }\end{array}\]
\end{fact}

The following fact in the algebraic setup is the main result of~\cite{Keb12} and states that for morphisms between klt base spaces there exists a meaningful pull-back map for reflexive differential forms. The result has been proven independently in~\cite[Thm.~2]{HJ13}.

\begin{fact}[{\cite{Keb12},~\cite[Thm.~2]{HJ13}}]\label{fact-kebekus-pullback}
There exists a naturally defined transitive pull-back map \[f^*:\Omega^{[i]}_W(W)\to\Omega^{[i]}_V(V)\] for any morphism $f:V\to W$ between algebraic klt base spaces such that the canonical maps
\[\Omega^i(V)\to\Omega^i_V/\tor(V)\to\Omega^{[i]}_V(V),\quad\quad V\text{ klt base space}\]
define transformations of contravariant functors from the category of klt base spaces to the category of differential graded commutative $\bC$-algebras.
\end{fact}

It seems feasible but exhausting to rewrite the whole paper \cite{Keb12} in the analytic setup which would yield an analytic analog of Fact~\ref{fact-kebekus-pullback} for holomorphic morphisms between locally algebraic klt complex base spaces. We will include a short proof of the analytic version in Remark~\ref{rem-kebekus-pull-back}.

\subsection{$\h$-differential forms}\label{ssec-h-diff-forms} The upshot of the preceding Section~\ref{ssec-class-diff-forms} is that if we pick up one of the classical sheaves of differential forms, then \emph{either} it does not admit a pull-back map by any holomorphic map between singular complex spaces \emph{or} it is difficult to handle in terms of a resolution of singularities, i.e., from the viewpoint of birational geometry. The sheaves $\Omega^i_\h|_X$ of \emph{$\h$-differential forms} try to solve this problem. 
\begin{center}
{ \renewcommand{\arraystretch}{1.5}
\begin{tabular}{|rl||c|c|}\hline
  & & $\exists$ pull-back & birational geometry \\ \hline\hline
  $\Omega^\bullet_X$ & K\"ahler diff. forms & $\bigoplus$ & $\bigominus$ \\ \hline
  $\Omega^\bullet_X/\tor$ &K\"ahler mod torsion & $\bigoplus$  & $\bigominus$ \\ \hline
  $\Omega^\bullet_\h|_X$ & $\h$-diff. forms & $\bigoplus$ & $\bigoplus$ \\ \hline
  $\Omega^{[\bullet]}_X$ & reflexive diff. forms & $\bigominus$  & $\bigoplus$\\ \hline
\end{tabular}}\end{center}
In fact, $\h$-differential forms turn out to admit a pull-back map by definition and Lemma~\ref{lem-easy-descr} provides a satisfying description in terms of resolutions of singularities.

\subsubsection{Definition of $\h$-differential forms.}
The letter $\h$ refers to the $\h$-Grothendieck topology on the category of schemes of finite type over a field of characteristic zero introduced by Voevodsky in~\cite[Def.~3.1.2]{Voe96}. In the algebraic setting, $\h$-differential forms are constructed as the sheafification of K\"ahler differential forms with respect to the $\h$-topology. That said, establishing elementary properties requires considerable technical efforts, see~\cite{HJ13}.

In contrast, in the analytic setting we are really interested in, many technical obstacles disappear. This renders possible a less involved approach pursued in the sequel.

\begin{defn}[$\h$-differential forms]\label{defn-h-diff}
Let $X$ be a reduced complex space and $i\geq 0$. An \emph{$\h$-differential form $\alpha$} of degree $i$ consists of the following data: 

For any  holomorphic map $f:M\to X$ from a complex manifold $M$ to $X$ we are given a holomorphic differential form $\alpha_f\in\Omega^i_M(M)$ of degree $i$ on $M$. These differential forms are required to satisfy $g^*\alpha_f=\alpha_{f\circ g}$ whenever $g:M'\to M$ is a holomorphic map from another complex manifold $M'$ to $M$.

We equip the \emph{set}
\[\Omega^i_\h(X) \cong \left\{ \alpha=\bigl(\alpha_f\in\Omega^i_M(M)\bigr)_f\Big|\quad \begin{array}{l}\xymatrix{M' \ar[rd]^{f\circ g} \ar[d]_g & \\ M \ar[r]_f & X }\end{array}\Rightarrow \phi^*\alpha_{f}=\alpha_{f\circ g}\right\}\]
of $\h$-differential forms of degree $i$ on $X$ with the obvious structure of an $\sO_X(X)$-module. We further define

\begin{itemize}
\item the \emph{wedge product} of $\h$-differential forms
\[ \wedge:\Omega^i_\h(X)\times\Omega^j_\h(X)\to\Omega^{i+j}_\h(X),\quad (\alpha\wedge\beta)_f=\alpha_f\wedge\beta_f;\]
\item the \emph{exterior derivative} of $\h$-differential forms
\[ d:\Omega^i_\h(X)\to\Omega^{i+1}_\h(X),\quad (d\alpha)_f=d\alpha_f; \]
\item the \emph{pull-back of $\h$-differential forms} by a holomorphic map  $\phi:X\to Y$
\[\phi^*:\Omega^i_\h(Y)\to\Omega^i_\h(X),\quad \phi^*(\alpha)_f=\alpha_{\phi\circ f};\]
\item the \emph{sheaf $\Omega^i_\h|_X$ of $\h$-differential forms}
\[\Omega^i_\h|_X(U):=\Omega^i_\h(U),\quad U\subset X\text{ open.}\]
\end{itemize}
\end{defn}

Let us justify the implicit claim that $\Omega^i_{\h}(X)$ is a \emph{set}: Fix some proper surjective map $\pi:X'\to X$ from a complex manifold $X'$ to $X$, e.g. a resolution of singularities. Given any holomorphic map $f:M\to X$ from a complex manifold $M$ to $X$, an arbitrary resolution $M'\to (M\times_XX')_\red$ yields a commutative square
\[\begin{array}{l}\xymatrix{M' \ar[d]^{p}_{\text{proper,}\atop \text{surj.}} \ar[r] & X' \ar[d]^\pi \\ M \ar[r]_f & X. }\end{array}\]
Since $p$ is proper and surjective, the pull-back of differential forms $p^*$ is injective by Lemma \ref{lem-prop-surj}(\ref{it-prop-surj-section}). In particular, a family $\alpha=(\alpha_f)_f$ as in Definition \ref{defn-h-diff} is determined by its value on $\pi$.

\subsubsection{Properties of $\h$-differential forms.} The relevant properties of the sheaves of $\h$-differential forms are summarized in the following proposition, which will be proved in Section~\ref{sssec-proof-properties-h} on page~\pageref{sssec-proof-properties-h}.

\begin{prop}[Properties of $\h$-differential forms]\label{prop-h-properties} Let $X$ be a reduced complex space and $i\geq 0$. Then
 \begin{enumerate}
 \item\label{it-prop-coh} $\Omega^i_\h|_X$ is a torsion-free coherent sheaf of $\sO_X$-modules;
 \item\label{it-prop-high-deg} $\Omega^i_\h|_X=0$ whenever $i>\dim(X)$;
 \item\label{it-prop-normal} The natural map $\sO_X\to\Omega^0_\h|_X$ is an isomorphism if $X$ is normal;
 \item\label{it-prop-smooth} The natural map $\Omega^i_X\to\Omega^i_\h|_X$ is an isomorphism on the smooth locus $X_{sm}$;
\end{enumerate}
Moreover, for any vector field  $V\in T_X(X)$ on $X$: 
\begin{enumerate}[resume]
 \item\label{it-prop-contraction} The contraction of K\"ahler differential forms along $V$ can be extended uniquely to the sheaf of $\h$-differential forms, i.e., there exists a unique commutative diagram
\[\begin{array}{l}\xymatrix{\Omega^i_X \ar[d]_{\iota_V}\ar[r] & \Omega^i_\h|_X \ar[d]^{\iota_V} \\ \Omega^{i-1}_X \ar[r] & \Omega^{i-1}_\h|_X }\end{array}\]
of morphisms of $\sO_X$-modules. 
\item\label{it-prop-lie-der} The Lie derivative of K\"ahler differential forms along $V$ can be extended uniquely to the sheaf of $\h$-differential forms, i.e., there exists a unique commutative diagram
\[\begin{array}{l}\xymatrix{\Omega^i_X \ar[d]_{\LieDer_V}\ar[r] & \Omega^i_\h|_X \ar[d]^{\LieDer_V} \\ \Omega^{i}_X \ar[r] & \Omega^{i}_\h|_X }\end{array}\]
of morphisms of sheaves of complex vector spaces. 
\end{enumerate}
\end{prop}

\subsubsection{Preparation for the proof of Proposition~\ref{prop-h-properties}} For lack of an adequate reference we include a proof of the following fact, which is well-known to experts.

\begin{lem}\label{lem-prop-surj}
Let $f:Y\to X$ be a proper, surjective, holomorphic map between complex manifolds.
\begin{enumerate}
 \item\label{it-prop-surj-section} There exists a closed analytic subset $A\subset X$ of codimension $\geq 1$ such that $$f:f^{-1}(X\backslash A)\to X\backslash A$$ admits local sections.
 \item\label{it-prop-surj-finite} There exists a closed analytic subset $B\subset X$ of codimension $\geq 2$ such that any $p\in X\backslash B$ admits a neighborhood in product form $B_1(0)\times V\subset X$, where $B_1(0):=\{t\in\bC:\,|t|<1\}$,  together with a commutative diagram
$$\begin{array}{l}\xymatrix{B_1(0)\times V \ar[d]_{\phi} \ar[r] & Y\ar[d] \\ B_1(0)\times V \ar[r]^-{\subset} & X, }\end{array}$$
where $\phi(t,v)=(t^k,v)$ for some $k>0$.
\end{enumerate}
\end{lem}

\begin{proof}
We equip the closed analytic subset
\[D:=\{y\in Y:\, d_yf:T_yY\to T_{f(y)}X\text{ is not surjective}\}\subset Y \]
with its reduced structure. Observe that by definition its image
\[A:=f(D)\subset Y\]
is an analytic subset of codimension $\geq 1$. The Inverse Mapping Theorem establishes Item~(\ref{it-prop-surj-section}).

Let $Y^0\subset Y$ be the union of the connected components of $Y$ contained in $f^{-1}(A)$. The image $f(Y^0)\subset X$ is nowhere dense so that $Y\backslash Y^0\to X$ is still surjective and proper. In particular, after replacing $Y$ by $Y\backslash Y^0$, we may assume the following.

\begin{add-ass}\label{add-ass-D-nowhere-dense}
The set $D\subset Y$ is nowhere dense.
\end{add-ass}

Using Additional Assumption~\ref{add-ass-D-nowhere-dense} and the Principalization Theorem~\cite[Thm.~3.26]{Koll07} we may even obtain by further blowing up $Y$ the following assumption.

\begin{add-ass}\label{add-ass-D-snc}
The reduced inverse image $f^{-1}(A)_\red\subset Y$ is an snc divisor with components $E_j$, $j\in J$. There exists a subset $I\subset J$ such that $D=\bigcup_{i\in I}E_i$. Moreover there exist positive integers $k_j>0$ such that 
\[{\textstyle \text{img}(f^*\sI_A\to\sO_Y)=\sO_Y(-\sum_{j\in J}k_j\cdot E_j)\subset\sO_Y.}\]
\end{add-ass}

We write $E_{j_1,j_2}:=E_{j_1}\cap E_{j_2}$ and define analytic subsets
\[
\begin{array}{lll}
S_j & := & \{s\in E_j:\, \text{rk}\bigl(d_sf:T_sE_j\to T_{f(s)}X\bigr)\leq\dim_{f(s)}X-2\}\subset E_j \\
S_{j_1,j_2} & := & \{s\in E_{j_1,j_2}:\, \text{rk}\bigl(d_sf:T_s(E_{j_1,j_2})\to T_{f(s)}X\bigr)\leq\dim_{f(s)}X-2\}\subset E_{j_1,j_2} \\
B & := & A_\singu\cup\bigcup_jf(S_j)\cup\bigcup_{j_1,j_2}f(S_{j_1,j_2})\subset X
\end{array}
\]
Observe that $B\subset X$ is a closed analytic subset of codimension $\geq 2$. 

\begin{claim}\label{claim-surj-fiber}
For any $p\in A\backslash B$ the fiber $F:=f^{-1}(\{p\})_\red$ has non-empty intersection with $f^{-1}(A)_{\red,\reg}$.
\end{claim}

\begin{proof}[Proof of the claim]
Let $y\in F_\reg\cap E_{j_1,j_2}$ be a point. Since $p\not\in f(S_{j_1,j_2})\cup f(S_{j_1})$, the maps $E_{j_1,j_2}\to A$ and $E_{j_1}\to A$ are both submersive at $y$. In particular, 
\[dim_y(F\cap E_{j_1})=dim(E_{j_1})-dim_x(A)>dim(E_{j_2,j_2})-dim_x(A)=dim_y(F\cap E_{j_1,j_2}),\]
which shows the claim.
\end{proof}

Let us now prove Item~(\ref{it-prop-surj-finite}). We may assume that $p\in A\backslash B$. By Claim~\ref{claim-surj-fiber} there exists a point $y\in f^{-1}(\{p\})\cap f^{-1}(A)_{\red,\reg}$. If $y\not\in D$, then $f$ is submersive at $y$ and the claim is obviously true for $k=1$ Otherwise $y\in E_i\cap D_\reg$ for some $i\in I$ and
\begin{enumerate}[label={(\alph*)}] 
 \item the map $E_i\to A$ is submersive at $y$ since $p\not\in f(S_i)$, and
 \item if, locally around $p\in X$, the subset $A$ is defined by an equation $t$, then $t\circ f=\epsilon\cdot s^{k_i}$ in a neighborhood of $y$, where $\epsilon$ is a unit and $s$ is a local equation describing $E_i\subset Y$, by Additional Assumption~\ref{add-ass-D-snc}.
\end{enumerate}
A short calculation using local coordinates and (a), (b) establishes the existence of a diagram as claimed in Lemma~\ref{lem-prop-surj}(\ref{it-prop-surj-finite}), with $k=k_i$.
\end{proof}

The algebraic analog of the following lemma can be found in~\cite[Prop.~4.2]{L09}, see also~\cite[Thm.~3.6]{HJ13}.

\begin{lem}\label{lem-descent-mfd}
Let $M,M',M''$ be complex manifolds, $M,M'$ connected, together with proper and surjective holomorphic maps $f:M'\to M$ and $g:M''\to M'\times_MM'$, and let $i\geq 0$. Then the pull-back by $f$ induces a bijection
\[\Omega^i_M(M)\,\cong\,\{\alpha\in\Omega^i_{M'}(M'):\, (pr_1\circ g)^*\alpha=(pr_2\circ g)^*\alpha\in\Omega^i_{M''}(M'')\}.\]
In other words, the pull-back maps fit into an equalizer diagram
\[
\begin{array}{l}\xymatrix@C=50pt@R=50pt{
 \Omega^i_M(M) \ar[r]^{f^*} & \Omega^i_{M'}(M')\ar@<-0.4ex>[r]^(.43){\parbox[b][12pt][t]{20pt}{\scriptsize $\scriptsize{(\pr_1\circ g)^*}$}} \ar@<0.4ex>[r]_(.43){\parbox[t][10pt][b]{20pt}{\scriptsize $\scriptsize{(\pr_2\circ g)^*}$}} & \Omega^i_{M''}(M''). 
}\end{array}
\]
\end{lem}

\begin{proof}
The pull-back map is injective by Lemma \ref{lem-prop-surj}(\ref{it-prop-surj-section}). To show surjectivity, let $\alpha$ be an element in the set on the right hand side. 

\begin{claim}\label{claim-local-section}
Let $\alpha$ be as above. Then for any two local sections $s,t:M\supset U\rightrightarrows M'$ of $f$ the pull-backs $s^*(\alpha)=t^*(\alpha)\in\Omega^i_M(U)$ of differential forms coincide.
\end{claim}

\begin{proof}[Proof of the claim]
The reduced preimage $S''=g^{-1}(s(U)\times_U t(U))_{\red}\subset M''$ is a locally closed analytic subset. A resolution $\reso{S}\to S''$ gives rise to a commutative diagram
\[
\begin{array}{l}\xymatrix@C=50pt@R=50pt{
\reso{S} \ar[d]^{q}_{\text{proper}\atop\text{surjective}} \ar[r] & M'' \ar[d]_g \ar@<-0.4ex>[rd]^-{\parbox[b][8pt][t]{18pt}{\scriptsize $\scriptsize{\pr_1\circ g}$}} \ar@<0.4ex>[rd]_-{\parbox[t][8pt][b]{20pt}{\scriptsize $\scriptsize{\pr_2\circ g}$}} & \\
U  \ar[r]_-{(s,t)} & M'\times_MM'\ar@<-0.4ex>[r]^-{\parbox[b][8pt][t]{8pt}{\scriptsize $\pr_1$}} \ar@<0.4ex>[r]_-{\parbox[t][8pt][b]{8pt}{\scriptsize $\pr_2$}} & M'. 
}\end{array}\]
The defining property of $\alpha$ implies that $q^*(s^*\alpha-t^*\alpha)=0$ and Lemma \ref{lem-prop-surj}(\ref{it-prop-surj-section}) applied to $q$ establishes Claim~\ref{claim-local-section}.
\end{proof}
 
Let $M_1:=M\backslash A$ where $A\subset M$ is as in Lemma~\ref{lem-prop-surj}(\ref{it-prop-surj-section}), with $Y\to X$ replaced by $f:M'\to M$. Claim \ref{claim-local-section} yields an element $\beta_1\in\Omega^i_M(M_1)$ such that $\beta|_U=s^*(\alpha)$ for any local section $s:M_1\supset U\to M'$ of $f$.

\begin{claim}\label{claim-pullback-agrees}
$\alpha|_{f^{-1}(M_1)}=f|_{f^{-1}(M_1)}^*(\beta_1)\in\Omega^i_M(M_1)$.
\end{claim}

\begin{proof}[Proof of the claim]
The set $f^{-1}(M_1)$ is connected since $M'$ is so by assumption and $M_1$ is the complement of an analytic subset of codimension $\geq 1$ by Lemma~\ref{lem-prop-surj}(\ref{it-prop-surj-section}). In particular, by the identity theorem, it suffices to verify the claimed equality on some non-empty open subset of $f^{-1}(M_1)$. There exist certainly non-empty open subsets $U\subset M$ and $U\times V\subset M'$ such that $f$ restricted to $U\times V$ is given by the first projection $U\times V\to U$. A short calculation in local coordinates using $(U\times V)\times_U(U\times V)\cong U\times V\times V$ finishes the proof of Claim~\ref{claim-pullback-agrees}.
\end{proof}

\begin{claim}\label{claim-extension}\label{claim-extension}
The differential form $\beta_1$ can be extended to $M$, i.e., there exists a differential form $\beta\in\Omega^i_M(M)$ such that $\beta|_{M_1}=\beta_1$.
\end{claim}

\begin{proof}[Proof of the claim]
Let $B_1(0)\times V\subset M$, $\phi$ and $k>0$ as in Lemma \ref{lem-prop-surj}(\ref{it-prop-surj-finite}), with $Y\to X$ replaced by $f:M'\to M$. Let $B^*= B_1(0)\backslash\{0\}$. Observe that $B^*\times V\subset M_1$. Claim~\ref{claim-pullback-agrees} implies that the pull-back $\phi|_{B^*\times V}^*(\beta_1)\in\Gamma(B^*\times V)$ coincides with the pull-back of $\alpha$ by $B^*\times V\to M'$. In particular, $\phi|_{B^*\times V}^*(\beta_1)$ can be extended to a differential form on $B_1(0)\times V$. A short calculation in local coordinates shows that this already implies that $\beta_1$ extends to $B_1(0)\times V\subset M$.

By what has been said so far and Lemma~\ref{lem-prop-surj}(\ref{it-prop-surj-finite}), the differential form $\beta_1$ can be extended to a differential form defined on the complement of an analytic subset of codimension $\geq 2$. This already implies that it can be extended to a differential form $\beta$ on $M$. The proof of Claim~\ref{claim-extension} is complete.
\end{proof}

Claims~\ref{claim-pullback-agrees} and~\ref{claim-extension} together finish the proof of Lemma~\ref{lem-descent-mfd}, since $f^{-1}(M_1)\subset M'$ is dense by connectedness of $M'$.
\end{proof}

\begin{lem}\label{lem-easy-descr}
Let $X$ be a reduced complex space, $i\geq 0$, and let $X'$, $X''$ be complex manifolds together with proper surjective holomorphic maps $\pi:X'\to X$ and $\phi:X''\to X'\times_XX'$. Then evaluation on $\pi$ yields a bijection
\[\Omega^i_\h(X)\,\cong\,\{\alpha\in\Omega^i_{X'}(X'):\, (pr_1\circ \phi)^*\alpha=(pr_2\circ \phi)^*\alpha\in\Omega^i_{X''}(X'')\}.\] 
\end{lem}

\begin{rem}\label{rem-easy-descr}
In the situation of Lemma~\ref{lem-easy-descr} let us write $q:=\pi\circ\pr_1\circ\phi:X''\to X$. Then Lemma~\ref{lem-easy-descr} establishes an isomorphism
\[\Omega^i_\h|_X\cong \text{ker}\bigl(\lambda:\,\pi_*\Omega^i_{X'}\to q_*\Omega^i_{X''}\bigr) \]
of sheaves on $X$, where $\lambda(\alpha)=\phi^*(pr_1^*\alpha-pr_2^*\alpha)$.
\end{rem}

\begin{proof}
The evaluation map is injective by what has been said following Definition \ref{defn-h-diff}. To see surjectivity let $\alpha$ be an element of the right hand side set. 

\begin{claim}\label{claim-easy-descr-1}
Let $M,M'$ be complex manifolds together with holomorphic maps $f:M\to X$, $f':M'\to X'$ and $p:M'\to M$ such that $p$ is proper and surjective and the diagram 
\[ {\textstyle \begin{array}{l}\xymatrix{X' \ar[r]^\pi & X \\ M' \ar[u]^{f'} \ar[r]_p & M\ar[u]_f }\end{array}}\]
is commutative. Then there exists a unique differential form $\alpha_{f,f',p}\in\Gamma(M,\Omega^i_M)$ such that $p^*(\alpha_{f,f',p})=f'^*\alpha\in\Gamma(M',\Omega^i_{M'})$.
\end{claim}

\begin{proof}[Proof of Claim~\ref{claim-easy-descr-1}]
Uniqueness holds by Lemma~\ref{lem-prop-surj}(\ref{it-prop-surj-section}) since $p$ is proper and surjective. To prove existence we may assume that $M$ and $M'$ are connected. Let $M''$ be a resolution of $X''\times_{X'\times_X X'}(M'\times_MM')$. This gives rise to a commutative diagram
\[\begin{array}{l}
\xymatrix@1@C=60pt@R=40pt{
X'' \ar@<-0.4ex>[r]^-{\parbox[b][10pt][t]{16pt}{\scriptsize $\scriptsize{\pr_1\circ\phi}$}} \ar@<0.4ex>[r]_-{\parbox[t][8pt][b]{16pt}{\scriptsize $\scriptsize{\pr_2\circ\phi}$}} & X'  \ar[r]^\pi & X \\ 
M'' \ar[u]^{f''} \ar@<-0.4ex>[r]^-{\parbox[b][10pt][t]{16pt}{\scriptsize $\scriptsize{\pr_1\circ\varphi}$}} \ar@<0.4ex>[r]_-{\parbox[t][8pt][b]{16pt}{\scriptsize $\scriptsize{\pr_1\circ\varphi}$}} &  M' \ar[u]^{f'} \ar[r]_p & M,\ar[u]_f 
}\end{array}\]
where $\varphi$ is the induced map $M''\to M'\times_MM'$. The commutativity implies that
\[(\pr_1\circ\varphi)^*(f'^*\alpha)=f''^*(\pr_1\circ\phi)^*\alpha=f''^*(\pr_2\circ\phi)^*\alpha=(\pr_2\circ\varphi)^*(f'^*\alpha) .\]
In this situation, Lemma \ref{lem-descent-mfd} asserts the existence of a differential form $\alpha_{f,f',p}\in\Omega^i_M(M)$ satisfying $p^*(\alpha_{f,f',p})=f'^*\alpha$.
\end{proof}

\begin{claim}\label{claim-easy-descr-2}
Let $M$ be a complex manifold and $f:M\to X$ be a holomorphic map. Then there exists a differential form $\alpha_f\in\Gamma(M,\Omega^i_M)$ such that $\alpha_f=\alpha_{f,f',p}$ for any triple $(f,f',p)$ as in Claim~\ref{claim-easy-descr-1}.
\end{claim}

\begin{proof}[Proof of Claim~\ref{claim-easy-descr-2}]
We abbreviate 
\[Y:=(M\times_XX')_\red.\]
Let $(f,f_1:M_1\to M,p_1:M_1\to X')$ be the triple obtained from a resolution $M_1\to Y$. We claim that $\alpha_{f,f',p'}=\alpha_{f,f_1,p_1}$ for any triple $(f,f',p')$ as in Claim~\ref{claim-easy-descr-1}.

To see this, let $M_1'\to (M_1\times_YM')_\red$ be a resolution and observe that the holomorphic map $M_1'\to M'$ is proper and surjective. The maps so far constructed fit into the following commutative diagram
\[\begin{array}{l}
\xymatrix@1@R=35pt@C=35pt{
  & X'  \ar[r]^\pi & X \\ 
 M_1 \ar@/^0.7pc/[ru]^{f_1} \ar[r] \ar@/_-0.7pc/[rr]^(.35){p_1}|(.49)\hole|(.505)\hole|(.52)\hole|(.58)\hole|(.595)\hole &  Y \ar[u] \ar[r]|(.15)\hole|(.18)\hole & M.\ar[u]_f \\
 M'_1 \ar[u] \ar[r]^{\text{proper}}_{\text{surjective}} & M' \ar[u] \ar@/_0.7pc/[ru]_{p'} \ar@/_0.7pc/[uu]_(.25){f'} &
}\end{array}\]
Commutativity shows that the pull-backs of $\alpha_{f,f_1,p_1}$ and $\alpha_{f,f',p'}$ to $M_1'$ agree. This implies that  $\alpha_{f,f_1,p_1}=\alpha_{f,f',p'}$, since both $M_1'\to M$ is proper and surjective, see also Lemma~\ref{lem-prop-surj}.
\end{proof}

The proof is finished if we show that the differential forms $\alpha_f$ defined in Claim~\ref{claim-easy-descr-2} yield a $\h$-differential form lifting $\alpha$. First, by applying Claim~\ref{claim-easy-descr-2} to the triple $(\pi,\id_\pi,\id_\pi)$, we see that $\alpha_\pi=\alpha$. Second, assume that we have maps $M_1\xrightarrow{g} M_2\xrightarrow{f} X$. By taking resolutions of reduced fiber products we obtain a commutative diagram
\[\begin{array}{l}
\xymatrix@1@C=40pt{
 M'_2 \ar[r] \ar[d]_{\text{proper}\atop\text{surjective}} & M'_1  \ar[r] \ar[d]_{\text{proper}\atop\text{surjective}}& X' \ar[d] \\ 
 M_2 \ar[r]_{g} & M_1  \ar[r]_{f} & X,
}\end{array}\]
which implies that $g^*\alpha_f=\alpha_{f\circ g}$.
\end{proof}

\subsubsection{Proof of Proposition~\ref{prop-h-properties}}\label{sssec-proof-properties-h}

Let $\pi:X'\to X$ and $\phi:X''\to (X'\times_XX')_\red$ be functorial resolutions. Then the assumptions in Lemma~\ref{lem-easy-descr} are satisfied.

Item~(\ref{it-prop-coh}) follows immediately from Remark~\ref{rem-easy-descr}, since $\Omega^i_\h|_X$ is the kernel of a morphism between torsion-free coherent sheaves on $X$.

Item~(\ref{it-prop-smooth}) holds true since over the smooth locus $X_{sm}$ we have $X\cong X'\cong X''$ and $\lambda=0$. Item~(\ref{it-prop-high-deg}) follows from Items~(\ref{it-prop-coh}) and~(\ref{it-prop-smooth}). Item~(\ref{it-prop-normal}) is a consequence of $j_*\sO_{X_{sm}}=\sO_X\to\Omega^0_\h|_X\subset j_*\Omega^0_{X_{sm}}=j_*\sO_{X_{sm}}$ where $j:X_{sm}\subset X$ is the inclusion of the smooth locus.

In order to see Item~(\ref{it-prop-contraction}), observe that there exists a vector field $V'\in T_{X'}(X')$ that is $\pi$-related to $V$. This implies that the vector field $pr_1^*(V')+pr_2^*(V')\in T_{X'\times X'}(X'\times X')$ preserves the subset $(X'\times_XX')_{\red}\subset X'\times X'$ and thus restricts to a vector field on $(X'\times_XX')_\red$, to which a unique vector field $V''\in T_{X''}(X'')$ is $\phi$-related. Existence in Item~(\ref{it-prop-contraction}) is a consequence of the commutative diagram
\[\begin{array}{l}\xymatrix{
\Omega^i_X\ar[r]\ar[d]_{\iota_V} & \Omega^i_\h|_X \ar[r] & \pi_*\Omega^i_{X'} \ar[r]^\lambda\ar[d]_{\iota_{V'}} & q_*\Omega^i_{X''}\ar[d]_{\iota_{V''}} \\
\Omega^{i-1}_X\ar[r] & \Omega^{i-1}_\h|_X \ar[r] & \pi_*\Omega^{i-1}_{X'} \ar[r]^\lambda & q_*\Omega^{i-1}_{X''}. 
}\end{array}\]
Uniqueness in Item~(\ref{it-prop-contraction}) follows from Items~(\ref{it-prop-coh}) and~(\ref{it-prop-smooth}). Existence in Item~(\ref{it-prop-lie-der}) is due to the formula $\LieDer_V=d\circ\iota_V+\iota_V\circ d$ and uniqueness follows again from Items~(\ref{it-prop-coh}) and~(\ref{it-prop-smooth}).\qed

\subsection{$\h$-differential forms in special cases}\label{ssec-h-mild-sing}

\begin{prop}[$\h$-differential forms on products]\label{prop-h-product}
Let $X$ be a reduced complex space and let $M$ be a complex manifold. Let further $\pr_X:X\times M\to X$ and $\pr_M:X\times M\to M$ denote the projection maps, respectively. Then the wedge product of $\h$-differential forms induces an isomorphism
\[ \bigoplus_{l=0}^i \pr_X^*(\Omega^l_\h|_X)\otimes {\pr}^*_M(\Omega^{i-l}_M)\cong \Omega^i_\h|_{X\times M}. \]
\end{prop}

\begin{proof}
The statement is known in the case when $X$ is smooth. Let us choose resolutions $\pi:X'\to X$ and $\phi:X''\to (X'\times_XX')_\red$ and write $q:X''\to X$. Using Remark~\ref{rem-easy-descr} we calculate
\begin{align*}
\Omega^i_\h|_{X\times M} &\,\cong \, \text{ker}\bigl((\pi\times\id_M)_*\Omega^i_{X'\times M}\to (q\times\id_M)_*\Omega^i_{X''\times M}\bigr)   \\
& \,\cong \,\bigoplus_{l=0}^i\text{ker}\Bigl((\pi\times\id_M)_*\bigl(\pr_{X'}^*\Omega^l_{X'}\otimes\pr_M^*\Omega^{i-l}_M\bigr)\to (q\times\id_M)_*\bigl(\pr_{X''}^*\Omega^l_{X''}\otimes\pr_M^*\Omega^{i-l}_M\bigr)\Bigr)   \\
&   \,\cong \,\bigoplus_{l=0}^i\pr_X^*\text{ker}\Bigl(\pi_*\pr_{X'}^*\Omega^l_{X'}\to q_*\pr_{X''}^*\Omega^l_{X''}\Bigr)\otimes\pr_M^*\Omega^{i-l}_M   \\
& \,\cong\, \bigoplus_{l=0}^i\pr_X^*(\Omega^l_\h|_X)\otimes\pr_M^*(\Omega^{i-l}_M),
\end{align*}
which finishes the proof.
\end{proof}

\begin{prop}\label{prop-h-isolated}
Let $X$ be a reduced complex space with only one isolated singularity $x\in X$ together with a strong resolution $\pi:\tilde{X}\to X$, $E=\pi^{-1}(\{x\})_\red$. 
\begin{enumerate}
\item\label{it-h-isolated-1} For any $i>0$ the pull-back by $\pi$ induces an isomorphism
\[\Omega^i_\h|_X\xrightarrow{\sim} \pi_*\sI_E\cdot\Omega^i_{\tilde{X}}(\log\, E). \]
\item\label{it-h-isolated-2} If in addition the complex space germ $\sg{X}{p}$ is irreducible, then
\[\Omega^0_\h|_X\xrightarrow{\sim}\pi_*\sO_{\tilde{X}}.\]
\end{enumerate}
\end{prop}

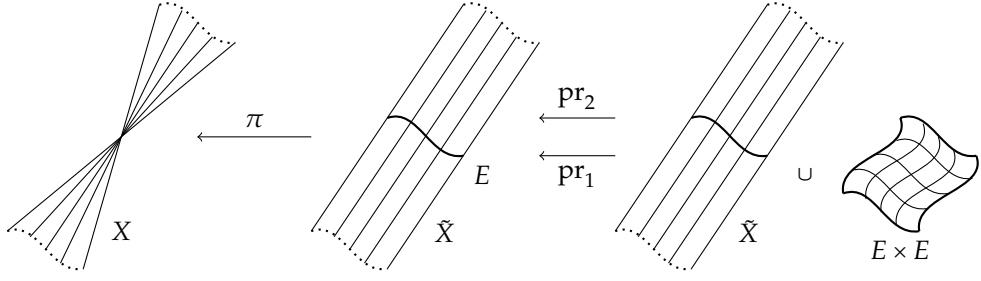
\begin{figure}
\begin{tikzpicture}
  [scale=.5,auto=left]

\tikzset{curve/.style={circle,draw = black,fill=blue!20, minimum size=0.5cm, inner sep = 0cm}}
\tikzset{contrcurve/.style={very thick,circle,draw = black,fill=blue!20, minimum size=0.5cm, inner sep = 0cm}}

  \node (L1) at (4,2) {$X$};
  \node (L2) at (12.5,2) {$\tilde{X}$};
  \node (L3) at (13.5,3.5) {$E$};
  \node (L4) at (20.5,2) {$\tilde{X}$};
  \node (L5) at (24.5,1.5) {$E\times E$};
  \node (L6) at (22,3.5) {$\cup$};

  \draw[thick,dotted] (1,2) to [out=20,in=190] (3,1);
  \draw[thick,dotted] (5,8) to [out=20,in=190] (7,7);
  \draw[-] (1,2) -- (7,7);
  \draw[-] (3,1) -- (5,8);
  \draw[-] (2,1.5) -- (6,7.5);
  \draw[-] (1.6,1.85) -- (6.4,7.15);
  \draw[-] (2.4,1.15) -- (5.6,7.85);

  \draw[thick,dotted] (9,2) to [out=20,in=190] (11,1);
  \draw[thick,dotted] (13,8) to [out=20,in=190] (15,7);
  \draw[thick] (11,5) to [out=20,in=190] (13,4);
  \draw[-] (9,2) -- (13,8);
  \draw[-] (11,1) -- (15,7);
  \draw[-] (10,1.5) -- (14,7.5);
  \draw[-] (9.6,1.85) -- (13.6,7.85);
  \draw[-] (10.4,1.15) -- (14.4,7.15);
  
  \draw[->] (9,4.5) -- node[above] {$\pi$}(6,4.5);

  \draw[thick,dotted] (9+8,2) to [out=20,in=190] (11+8,1);
  \draw[thick,dotted] (13+8,8) to [out=20,in=190] (15+8,7);
  \draw[thick] (11+8,5) to [out=20,in=190] (13+8,4);
  \draw[-] (9+8,2) -- (13+8,8);
  \draw[-] (11+8,1) -- (15+8,7);
  \draw[-] (10+8,1.5) -- (14+8,7.5);
  \draw[-] (9.6+8,1.85) -- (13.6+8,7.85);
  \draw[-] (10.4+8,1.15) -- (14.4+8,7.15);

  \draw[->] (17,5) -- node[above] {$\text{pr}_2$} (15,5);
  \draw[->] (17,4) -- node[below] {$\text{pr}_1$} (15,4);
 
  \draw[thick] (23,3) to [out=20,in=190] (25,2);
  \draw[thick] (24.5,5) to [out=20,in=190] (26.5,4);
  \draw (23.75,4) to [out=20,in=190] (25.75,3);
  \draw (23.2,3.6) to [out=20,in=190] (25.2,2.6);
  \draw (24.4,4.4) to [out=20,in=190] (26.4,3.4);

  \draw[thick] (23,3) to [out=100,in=290] (24.5,5);
  \draw[thick] (25,2) to [out=100,in=290] (26.5,4);
  \draw (24,2.5) to [out=100,in=290] (25.5,4.5);
  \draw (23.6,2.85) to [out=100,in=290] (25.1,4.85);
  \draw (24.4,2.15) to [out=100,in=290] (25.9,4.15);

\end{tikzpicture}
\setlength{\captionmargin}{7pt}
\caption{Proof of Proposition~\ref{prop-h-isolated}, if $X$ is a cone.}
\end{figure}

\begin{proof}
Let $E_j$ be the irreducible components of $E$. Then the maps $\pi:\tilde{X}\to X$ and $\bigsqcup_{j,k}E_j\times E_k\sqcup \tilde{X}\xrightarrow{(\text{incl},\,\text{diag})} \tilde{X}\times_X\tilde{X}$ satisfy the assumptions of Lemma~\ref{lem-easy-descr} so that for any $i\geq 0$ there is an exact sequence
\[0\to\Omega^i_\h(X)\xrightarrow{\pi^*}\Omega^i_{\tilde{X}}(\tilde{X})\xrightarrow{(0,\,\text{pr}^*_1-\text{pr}^*_2)}\Omega^i_{\tilde{X}}(\tilde{X})\oplus\bigoplus_{j,k}\Omega^i_{E_j\times E_k}(E_j\times E_k).\]
Item (\ref{it-h-isolated-1}) follows since for $i>0$ this implies that the pull-back map by $\pi$ yields a bijective map
\[{\textstyle \pi^*:\Omega^i_\h(X)\xrightarrow{\sim}\text{ker}\Bigl(\Omega^i_{\tilde{X}}(\tilde{X})\to\bigoplus_j\Omega^i_{E_j}(E_j)\Bigr)\cong H^0(\tilde{X},\sI_E\cdot\Omega^i_{\tilde{X}}(log\, E)).} \]
In the case $i=0$ we see that
\[\pi^*:\Omega^0_\h|_X(X)\xrightarrow{\sim}\left\{f\in\sO_{\tilde{X}}(\tilde{X}):\quad \forall j,k.\, f(E_j)=f(E_k)\subset\bC\right\}. \]
Moreover, if $\sg{X}{p}$ is irreducible, then the condition on the right hand side is automatically satisfied so that Item (\ref{it-h-isolated-2}) holds.
\end{proof}

\begin{lem}\label{lem-h-an-alg}
Let $V$ be a reduced separated scheme of finite type over $\bC$. Then there exists a natural isomorphism
\[\Omega^i_\h|_{V^\an}\cong\bigl(\Omega^i_\h|_{V}\bigr)^\an\]
of coherent sheaves on $V^\an$ for any $i\geq 0$, where $\Omega^i_\h|_V$ is the sheaf of algebraic $\h$-differential forms on $V$ introduced in~\cite{HJ13}.
\end{lem}

\begin{proof}
Let $\pi:V'\to V$ and $\phi:V''\to (V'\times_V V')_{red}$ be resolutions of singularities and write $q:V''\to V$. By \cite[Rem.~3.7]{HJ13} the pull-back of algebraic $\h$-differential forms induces an isomorphism
\[\Omega^i_\h|_V\cong\text{ker}\bigl(\lambda:\,\pi_*\Omega^i_{V'}\to q_*\Omega^i_{V''}\bigr)\]
where $\lambda(\alpha)=\phi^*(pr_1^*\alpha-pr_2^*\alpha)$. Analytifying this isomorphism yields
\begin{align*}
\bigl(\Omega^i_\h|_V\bigr)^{\an} &\,\cong \, \Bigl(\text{ker}\bigl(\lambda:\,\pi_*\Omega^i_{V'}\to q_*\Omega^i_{V''}\bigr)\Bigr)^\an  &&  \\
& \,\cong \,\text{ker}\bigl(\lambda^\an:\,(\pi_*\Omega^i_{V'})^\an\to (q_*\Omega^i_{V''})^\an\bigr)   &&\text{by exactness of }(\cdot)^\an\\
& \,\cong \,\text{ker}\bigl(\pi^\an_*\Omega^i_{V'^\an}\to q^\an_*\Omega^i_{V''^\an}\bigr)   &&\text{by~\cite[Ch.~XII, Thm.~4.2]{SGA1}}\\
& \,\cong\, \Omega^i_\h|_{V^\an} &&\text{by Remark }\ref{rem-easy-descr}
\end{align*}
and thus finishes the proof.
\end{proof}

\begin{prop}\label{prop-h-refl-klt}
Let $X$ be a locally algebraic klt base space. Then the sheaves of $\h$-differential forms and reflexive differential forms agree, i.e., for any $i\geq 0$,
\[\Omega^i_\h|_X=\Omega^{[i]}_X.\]
\end{prop}

\begin{rem}\label{rem-kebekus-pull-back}
Proposition~\ref{prop-h-refl-klt} implies that there exists a pull-back map $f^*:\Omega^{[i]}_Y\to f_*\Omega^{[i]}_X$ associated with any holomorphic map $f:X\to Y$ between locally algebraic klt base spaces.
\end{rem}

\begin{proof}[Proof of Proposition~\ref{prop-h-refl-klt}]
The question is local on $X$ and we thus may assume that $X$ is the complex space associated with an algebraic klt base space $V$. Then
\[\Omega^i_\h|_X\overset{\ref{lem-h-an-alg}}{\cong}\bigl(\Omega^i_\h|_V\bigr)^\an\overset{}{\cong}\bigl(\Omega^{[i]}_V\bigr)^\an \cong \Omega^{[i]}_X\]
where the middle isomorphism is \cite[Prop.~5.2]{HJ13} and the last isomorphism is shown in the proof of~\cite[Lem.~2.16]{GKP12}.
\end{proof}

\begin{prop}\label{prop-h-refl-isolated-rational}
Let $X$ be a normal complex space with isolated rational singularities. Then the sheaves of $\h$-differential forms and reflexive differential forms agree, i.e., for any $i\geq 0$,
\[\Omega^i_\h|_X=\Omega^{[i]}_X.\]
\end{prop}

\begin{proof}
Observe that the case $i=0$ is settled by Proposition~\ref{prop-h-properties}(\ref{it-prop-normal}). From now on, let $i>0$. Let $\pi:\tilde{X}\to X$ be a strong resolution with exceptional divisor $E$. We first prove that
\begin{equation}\label{eqn-refl-equals-f-k} \Omega^{[i]}_X\cong \pi_*\Omega^i_{\tilde{X}} \end{equation}
for $1\leq i\leq n$ by the following case-by-case analysis.
\[
\begin{array}{ll}
  1\leq i\leq n-2: & \text{by normality of } X  \text{ and~\cite[Thm.~1.3]{SvS85},}\\
  i=n: & \text{by the rationality assumption and~\cite[Thm.~5.10]{KM98},}\\
  i=n-1: & \text{by Case } i=n \text{ and~\cite[Cor.~1.4]{SvS85}}.
\end{array}
\]
This proves Equation~(\ref{eqn-refl-equals-f-k}). Recall from~\cite[Prop.~3.9]{Keb12} that there is a short exact sequence $0\to\sI_E\cdot\Omega^i_{\tilde{X}}(log\, E)\to \Omega^i_{\tilde{X}}\to \Omega^i_E/\tor\to 0$ of sheaves. Pushing forward yields an exact sequence
\[0\to\pi_*\bigl(\sI_E\cdot\Omega^i_{\tilde{X}}(log\, E)\bigr)\xrightarrow{\alpha} \pi_*\Omega^i_{\tilde{X}}\to H^0(E,\Omega^i_E/\tor). \]
The group on the right hand side vanishes by~\cite[Lem.~1.2]{Nam01}. In particular, the inclusion $\alpha$ is bijective and the claim follows from Equation~(\ref{eqn-refl-equals-f-k}) and Proposition~\ref{prop-h-isolated}.
\end{proof}

%% file: topology.tex
\section{Poincar\'e Lemma and the topology of $X$}\label{sec-top}

Our results concerning the topology of $X$ will be formulated in terms of intersection cohomology. An appropriate introduction can be found in~\cite{Bor84}. Recall from~\cite[Part~IV.]{Bor84} that any complex space $X$ is a pseudomanifold. By the \emph{$k$-th rational intersection cohomology $IH^k(X,\bQ)$ of $X$} we mean the intersection cohomology of $X$ with coefficients in the constant system $\bQ_{X_\sm}$ with respect to the lower middle perversity as defined in~\cite[Sect.~I.3.1,~Not.~V.2.6]{Bor84}.

\begin{defn}[Local intersection cohomology]\label{defn-local-intersection-coho}
Let $p\in X$ be a point on a complex space. The \emph{$k$-th rational local intersection cohomology group at $p\in X$} is the direct limit
\[{\textstyle \IH^k_\text{loc}(p\in X,\bQ):=\varinjlim_{U} \IH^k(U,\bQ),} \]
where $U$ runs over all open neighborhoods of $p$ in $X$.
\end{defn}

In other words, the rational local intersection cohomology group at $p\in X$ is isomorphic to the rational intersection cohomology group of the open cone over the link of the singularity $p\in X$.

\subsection{Topological Poincar\'e Lemma in degree one}\label{ssec-top-poincare} The proof of Theorem~\ref{thm-top-poincare} is based on the following lemma.

\begin{lem}\label{lem-local-int-coho-criterion}
Let $p\in X$ be a normal point on a complex space associated with a complex variety. Let $U\subset X$ be any neighborhood of $p$ that is homeomorphic to the open cone over the link of $p\in X$. Then
\[\IH^1_{\text{loc}}(p\in X,\bQ)\cong H^1(U_\reg,\bQ).\]
\end{lem}

\begin{proof}
Durfee has proved in~\cite[Lem.~1]{D95} that if $Y$ is the complex space associated with a normal complex variety, then $\IH^1(Y,\bQ)\cong H^1(Y_\reg,\bQ)$. A closer look at his proof reveals that this statement also applies to the open set $U$ given in Lemma~\ref{lem-local-int-coho-criterion}. This establishes the claim since $\IH^1_{\text{loc}}(p\in X,\bQ)\cong \IH^1(U,\bQ).$
\end{proof}

\begin{proof}[Proof of Theorem~\ref{thm-top-poincare}]
Let $\lambda\in\Gamma(V,\Omega^{[1]}_X)$ be a closed reflexive differential form of degree one defined on an open neighborhood $V\subset X$ of $p$. We need to show that there exists an open neighborhood $U\subset V$ of $p$ together with a holomorphic function $f\in\Gamma(U,\sO_U)$ such that $\lambda|_U=\text{d}f$. We claim that any $U\subset V$ of the form specified in Lemma~\ref{lem-local-int-coho-criterion} satisfies this requirement. In fact, let $x_0\in U_\reg$ be an arbitrary point. For $x\in U_\reg$ we define
\[f(x):=\int_\gamma\lambda,\]
where $\gamma$ is a continuous path from $x_0$ to $x$ contained in $U_\reg$. To see that this definition does not depend on the choice of $\gamma$, recall that the value $\int_\delta\lambda$ for a closed path $\delta$ in $U_\reg$ only depends on its integral cohomology class $H^1(U_\reg,\bZ)$. Since $H^1(U_\reg)$ is a torsion module by Lemma~\ref{lem-local-int-coho-criterion} and the universal coefficient theorem for cohomology, the integral $\int_\delta\lambda=0$ vanishes for any closed path so that $f$ is well-defined. Since $X$ is normal by assumption, $f$ extends to a holomorphic function on $U$ satisfying $df=\lambda$.
\end{proof}

As mentioned in the introduction, the following proposition exhibits many results in the literature as special cases of Theorem~\ref{thm-top-poincare}.

\begin{prop}\label{prop-van-loc-int-coho}
Let $X$ be a normal complex space and $p\in X$. Suppose that
\begin{enumerate}
 \item\label{it-van-loc-int-coho-2} the point $p\in X$ is an isolated complete intersection singularity of dimension $\geq 3$, or
 \item\label{it-van-loc-int-coho-1} $X$ is locally algebraic and has $1$-rational singularities, i.e., $R^1\pi_*\sO_{\tilde{X}}=0$ for any resolution $\pi:\tilde{X}\to X$.
\end{enumerate}
Then the first local intersection cohomology vanishes, i.e., $\IH^1_\text{loc}(p\in X,\bQ)=0$.
\end{prop}

\begin{proof}
To prove the claim in Case~\ref{it-van-loc-int-coho-2} choose a neighborhood $U\subset X$ as in Lemma~\ref{lem-local-int-coho-criterion}. Hamm showed in~\cite[Kor.~1.3]{H71} that $U_\reg$ is $(n-2)$-connected. Since $n\geq 3$, the claim follows from Lemma~\ref{lem-local-int-coho-criterion}.

In order to prove the claim in Case~\ref{it-van-loc-int-coho-1}, it certainly suffices to show that $\IH^1(U,\bQ)=0$ for $U$ running over a basis of the system of open neighborhoods of $p$ in $X$. Thus, using the local algebraicity assumption, it suffices to prove the following claim, which we will do in the sequel.

\begin{claim}
If $U\subset X$ is a contractible neighborhood of $p$ that admits an open embedding $U\subset V^\an$ into the complex space associated with a complex variety $V$, then $\IH^1(U,\bQ)=0$.
\end{claim}

Let $\phi:\tilde{V}\to V$ be an algebraic resolution and let $\pi:\tilde{U}\to U$ be the restriction of $\phi^\an$ to $\phi^{\an,-1}(U)$. Pushing forward the exponential sequence on $\tilde{U}$ gives rise to an exact sequence
\[{\textstyle 0\to\bZ_U\to\sO_U\xrightarrow[\text{surj.}]{\text{exp}} \sO_U^*\to R^1\pi_*\bZ_{\tilde{U}}\to \overbrace{R^1\pi_*\sO_{\tilde{U}}}^{=0\text{, by ass.}}\to\cdots} \]
so that $R^1\pi_*\bZ_{\tilde{U}}=0$. The five-term exact sequence of the Leray spectral sequence for the singular cohomology on $\tilde{U}$ is given by
\[0\to \underbrace{H^1(U,\bZ_U)}_{=0\text{, contractibility}} \to H^1(\tilde{U},\bZ)\to H^0(U,\underbrace{R^1\pi_*\bZ_{\tilde{U}}}_{=0}) \to\cdots\,.\]
Together with the universal coefficient theorem for cohomology this shows that 
\[H^1(\tilde{U},\bQ)=0.\]
Finally, since $\pi$ is the restriction of an algebraic resolution, the decomposition theorem~\cite[Thm.~6.2.5]{BBD81} implies that $\IH^1(U,\bQ)$ is a direct summand of $H^1(\tilde{U},\bQ)$, as outlined in the proof of~\cite[Cor.~5.4.11]{Dim04}. This shows that $\IH^1(U,\bQ)=0$.
\end{proof}

\subsection{Gorenstein normal surface singularities}\label{ssec-gorenstein-surface}
The proof of Proposition~\ref{intro-prop-gorenstein-surface} of the introduction can be found at the end of this section. We first include two preparatory lemmas. 

\begin{lem}\label{lem-int-coho-surfaces}
Let $p\in X$ be a normal surface singularity and let $\pi:\tilde{X}\to X$ be a strong resolution with reduced exceptional divisor $E=\sum_iE_i$ over $p$. Let $b$ be the number of loops in the dual graph of $E$ and let $g_i$ be the genus of $E_i$. Then
\[\IH^i_{\text{loc}}(p\in X,\bQ)=\begin{cases} \bQ & \text{if } i=0  \\ \bQ^{b+2\cdot\sum_ig_i} & \text{if } i=1 \\ 0 & \text{else}\end{cases} \]
\end{lem}

\begin{proof}
Let $L$ be the link of the singularity $p\in X$. It follows immediately from~\cite[II.3.1,~V.2.9]{Bor84} that $\IH^i_{\text{loc}}(p\in X,\bQ)=0$ if $i\geq 2$ and $\IH^i_{\text{loc}}(p\in X,\bQ)=H_{3-i}(L,\bQ)$ is the usual homology group of $L$ with rational coefficients if $i=0,1$.

We claim that the link $L$ is connected. Indeed, let us suppose to the contrary that $L=L_0\cup L_1$ is a disjoint union of two non-empty open subsets $L_0,L_1\subset L$. Let $p\in U\subset X$ further be an open neighbourhood that admits a homeomorphic map $U\cong L\times [0,1)/L\times\{0\}$ to the open cone over the link. Then we can define a holomorphic function $f:U\backslash\{p\}\to\bC$ that constantly equals zero on the preimage of $L_0\times (0,1)$ in $U$ and constantly equals one on the preimage of $L_1\times (0,1)$ in $U$. By normality this holomorphic function can be extended to a holomorphic function $U\to \bC$ defined on $U$. This is absurd since there exists not even continuous extension. In particular, $L$ is connected.

Since the three-dimensional compact manifold $L$ is orientable, we see that $H_3(L,\bQ)=\bQ$ by connectedness. Furthermore we have $\dim_\bQ H_1(L,\bQ)=b+2\cdot\sum_ig_i$ by~\cite[p.~10]{Mu61}, which implies the claim since $\dim_\bQ H_1(L,\bQ)=\dim_\bQ H_2(L,\bQ)$ by Poincar\'e duality.
\end{proof}

\begin{comment}
\begin{proof} It is easy to see that the integer $b+2\cdot\sum_ig_i$ does not depend on the strong resolution so that it suffices to show the lemma for one special resolution. This will be done in the sequel.

The proof of Lemma~\ref{lem-int-coho-surfaces} uses the following notation.

\begin{notation}
Let $IC_X$ be the intersection cohomology complex with rational coefficients, normalized such that $R^{-2+i}\Gamma(X,IC_X)\cong IH^i(X,\bQ)$.
\end{notation}

Recall that the singularity $p\in X$ is locally algebraic by Artin's Approximation Theorem \cite[Thm.~3.8]{Art69}. Let $\phi:\tilde{V}\to V$ be a strong resolution of an algebraic surface that admits an open embedding $U\subset V^\an$ of a neighborhood $U\subset X$ of $p$. Let $\pi:\tilde{U}\to U$ be the restriction of $\phi^\an$ to $U$. By~\cite[Ex.~1.8.1]{dCM09} there exists a decomposition
\[R\pi_*\bQ_{\tilde{V}^\an}[2]\cong IC_{V^\an}\oplus T,\]
where $T$ is a skyscraper sheaf with stalk $H^2(E,\bQ)$. This implies that
\[H^i(\tilde{U},\bQ)\cong \IH^i(U,\bQ),\quad i\neq 2 \]
and
\[H^2(\tilde{U},\bQ)\cong \IH^2(U,\bQ)\oplus H^2(E,\bQ).\]
The claim follows now from Fact~\ref{fact-top-reso}, since $H^1(E,\bC)\cong\bQ^{b+2\cdot\sum_ig_i}$.
\end{proof}
\end{comment}

\begin{lem}\label{lem-quasihom-surface}
Let $p\in X$ be a quasihomogeneous normal surface singularity. Then 
\[ \sO_{X,p}\to\Omega^{[1]}_{X,p}\to\Omega^{[2]}_{X,p}\quad \text{is exact}\quad \Leftrightarrow\quad \IH^1_{\text{loc}}(p\in X,\bQ)=0.\]
\end{lem}

\begin{proof}[Proof of Lemma~\ref{lem-quasihom-surface} - Setup of notation]
Fact~\ref{fact-quasihom-algebraizity} provides a complex algebraic variety $p\in V$ together with an algebraic $\bC^*$-action leaving $p$ fixed such that there exists an equivariant isomorphism $\sg{V}{p}^\an\cong \sg{X}{p}$ of complex space germs. After shrinking $X$ if necessary, we thus may assume that there exists an equivariant open embedding $X\subset V^\an$. Let $\overline{V}$ be the projective variety with algebraic $\bC^*$-action obtained as follows: We embed $V\subset\bC^N$ such that the torus action is induced by an action by diagonal matrices on $\bC^N$. Then let $\overline{V}$ be the closure of $V$ in $\bP^N\supset\bC^N$.

Let $\pi:\tilde{V}\to \overline{V}$ be a functorial resolution so that there is an algebraic action of $\bC^*$ on $\tilde{V}$ and $\pi$ is an equivariant morphism. The reduced fiber $\pi^{-1}(\{p\})_\red$ is an snc divisor, which we denote by $E=\sum_iE_i$.
\renewcommand{\qedsymbol}{}
\end{proof}

\begin{proof}[Proof of Lemma~\ref{lem-quasihom-surface} - Technical preparation]
Since $\tilde{V}$ is complete,~\cite[Sect.~4]{Bia73} implies the existence of smooth connected components $F^+,F^-\subset \tilde{V}^{\bC^*}$ of the set of fixed points on $\tilde{V}$ that are uniquely determined by the existence of a dominant rational morphisms $\psi^\pm:\tilde{V}\dashedrightarrow F^\pm$ such that
\[\psi^+(v)=\lim_{\bC^*\ni t\to 0}t\cdot v \in F^+ \quad \text{and}\quad \psi^-(v)=\lim_{\bC^*\ni t\to \infty}t\cdot v \in F^- \]
for any general point $v\in \tilde{V}$. Moreover we have
\begin{enumerate}
 \item\label{BB1} $F^+\subset E, F^-\not\subset E$ by quasihomogeneity of $p\in X$.
 \item\label{BB2} If $E_i\not\cong\bP^1$, then $E_i=F^+$. Indeed, if $E_i\not\cong\bP^1$, then $E_i$ is point-wise fixed by $\bC^*$ so that~\cite[Sect.~4, Cor.~1]{Bia73} implies that $E_i=F^+$ or $E_i=F^-$. The latter case does not occur by $(\ref{BB1})$. 
 \item\label{BB3} The dual graph $\Gamma$ of $E$ is a star, see e.g.~\cite[Sect.~2]{P77}. In the terminology of Lemma~\ref{lem-int-coho-surfaces} this means that $b=0$.
\end{enumerate}
\renewcommand{\qedsymbol}{}
\end{proof}

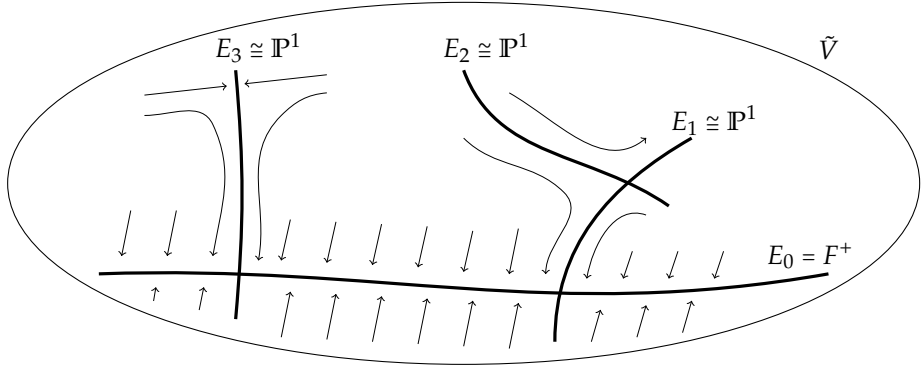
\begin{figure}[ht]\centering
\begin{tikzpicture}
  [scale=.6,auto=left]
\tikzset{curve/.style={circle,draw = black,fill=blue!20, minimum size=0.5cm, inner sep = 0cm}}
\tikzset{contrcurve/.style={very thick,circle,draw = black,fill=blue!20, minimum size=0.5cm, inner sep = 0cm}}
  \draw (0,0) ellipse (10cm and 4cm); 
  \node (L1) at (8,3) {$\tilde{V}$};

  \draw[very thick] (-8,-2) to [out=2,in=190] (8,-2);
  \node at (7.6,-1.6) {$E_0=F^+$};

  \draw[very thick] (2,-3.5) to [out=90,in=210] (5,1);
  \node at (5.5,1.4) {$E_1\cong\bP^1$};

  \draw[very thick] (4.5,-0.5) to [out=145,in=-70] (0,2.5);
  \node at (0.5,3) {$E_2\cong\bP^1$};

  \draw[very thick] (-5,-3) to [out=85,in=275] (-5,2.5);
  \node at (-4.5,3) {$E_3\cong\bP^1$};

  \draw[->,very thin] (0,1) to [out=-45,in=135] (2.2,-0.2) to [out=-45,in=90] (1.8,-2);

  \draw[->,very thin] (0.2,-1.05) to  (0,-1.95);
  \draw[->,very thin] (1.2,-1) to  (1,-2);
  \draw[->,very thin] (-0.8,-0.95) to (-1,-1.85);
  \draw[->,very thin] (-1.8,-0.9) to (-2,-1.8);
  \draw[->,very thin] (-2.8,-0.85) to  (-3,-1.75);
  \draw[->,very thin] (-3.8,-0.8) to  (-4,-1.7);

  \draw[->,very thin] (0,-3.6) to  (0.2,-2.6);
  \draw[->,very thin] (1,-3.65) to (1.2,-2.65);
  \draw[->,very thin] (-1,-3.55) to (-0.8,-2.55);
  \draw[->,very thin] (-2,-3.5) to (-1.8,-2.5);
  \draw[->,very thin] (-3,-3.45) to (-2.8,-2.45);
  \draw[->,very thin] (-4,-3.4) to (-3.8,-2.4);

  \draw[->,very thin] (-3,2) to [out=185,in=45] (-4,1.5) to [out=225,in=80] (-4.5,-1.7);
  \draw[->,very thin] (-3,2.4) to (-4.8,2.2);
  \draw[->,very thin] (-7,1.95) to (-5.2,2.15);
  \draw[->,very thin] (-7,1.5) to [out=0,in=115] (-5.5,1.3) to [out=-65,in=80] (-5.5,-1.6);

  \draw[->,very thin] (-6.3,-0.6) to (-6.5,-1.6);
  \draw[->,very thin] (-7.3,-0.6) to (-7.5,-1.6);

  \draw[->,very thin] (-5.8,-2.8) to (-5.7,-2.3);
  \draw[->,very thin] (-6.8,-2.6) to (-6.75,-2.3);

  \draw[->,very thin] (1,2) to [out=-30,in=220] (4,1);

  \draw[->,very thin] (4,-0.7) to [out=160,in=80] (2.7,-2.1);

  \draw[->,very thin] (3.7,-1.5) to (3.5,-2.1);
  \draw[->,very thin] (4.7,-1.5) to (4.5,-2.1);
  \draw[->,very thin] (5.7,-1.5) to (5.5,-2.1);

  \draw[->,very thin] (2.8,-3.5) to (3,-2.8);
  \draw[->,very thin] (3.8,-3.4) to (4,-2.7);
  \draw[->,very thin] (4.8,-3.3) to (5,-2.6);

\end{tikzpicture}
\setlength{\captionmargin}{2pt}
\caption{Possible situation in the proof of Lemma~\ref{lem-quasihom-surface}. The exceptional divisor $E=E_0+E_1+E_2+E_4$ is drawn thick. The thin drawn arrows indicate in which direction $t\cdot v$ moves when $t\in(0;1]\subset\bC^*$ runs from $1$ towards $0$ and $v\in\tilde{V}$.}
\label{fig:figure1}
\end{figure}

\begin{proof}[Proof of Lemma~\ref{lem-quasihom-surface}]
By Theorem~\ref{thm-top-poincare} we only need to show the 'only if' part. To this end we assume that $IH^1_{\text{loc}}(p\in X,\bQ)\neq 0$. We need to find a closed reflexive $1$-form on $X$ that is not exact on arbitrarily small neighborhoods of $p\in X$.

Lemma~\ref{lem-int-coho-surfaces} and Item~(\ref{BB3}) together imply that $2\cdot\sum_ig_i=b+2\cdot\sum_ig_i\neq 0$. Then Item~(\ref{BB2}) shows that $F^+\subset E$ is an irreducible component of positive genus. In particular, there exists a non-zero differential form $\alpha\in\Gamma(F^+,\Omega^1_{F^+})\backslash\{0\}$. Since $F^+$ is complete, the rational map $\psi^+:\tilde{V}\to F^+$ is a morphism on an open subset $V'\subset\tilde{V}$ whose complement has codimension $\geq 2$. In particular, the pull-back $\psi|_{V'}^*(\alpha)\in\Gamma(V',\Omega^1_{V'})$ extends to a closed non-zero differential form $\beta\in\Gamma(\tilde{V},\Omega^1_{\tilde{V}})$. Since $\beta|_{F^+}=\alpha$ is not exact, the induced holomorphic reflexive differential form on $X$ is not exact on arbitrarily small neighborhoods of $p\in X$.
\end{proof}

\begin{proof}[Proof of Proposition~\ref{intro-prop-gorenstein-surface}]
Because of the Gorenstein assumption we may apply the main result in~\cite{DY10}. It claims that the cohomology groups of the complex $0\to\sO_{X,p}\to\Omega^{[1]}_{X,p}\to\Omega^{[2]}_{X,p}\to 0$ in degrees one and two have the same dimension \emph{if and only if} the singularity $p\in X$ is quasihomogeneous. Together with Theorem~\ref{thm-top-poincare} and Lemma~\ref{lem-quasihom-surface} this implies the claim.
\end{proof}

\subsection{Dependancy on the holomorphic structure}\label{ssec-dependancy-hol-structure}
During the proof of Proposition~\ref{intro-ex-top-counterexample} we are concerned with normal surface singularities whose minimal resolution has an exceptional divisor of the following type.

\begin{defn}\label{cond-S}
Let $S\subset\bP^1$ be a subset of cardinality four. We say that a normal surface singularity $p\in X$ is \emph{of type $(S)$} if the reduced exceptional set of the minimal resolution is a snc divisor $E=C+C_1+\cdots+C_4$ such that
\begin{itemize}
 \item the curves $C,C_1,\cdots,C_4$ are rational curves,
 \item $C_i\cap C_j=\emptyset$ for $i\neq j$,
 \item $(C\cdot C)=-2$, $(C\cdot C_i)=1$ and $(C_i\cdot C_i)=-3$, and
 \item there exists an isomorphism $C\cong \bP^1$ that maps $C\cap(C_1+\cdots+ C_4)\subset C$ onto $S\subset\bP^1$.
\end{itemize}
\end{defn}

\begin{lem}\label{lem-type-S-quasi}
Let $S\subset\bP^1$ be a subset of cardinality $4$. Then, up to isomorphism of complex space germs, there exists exactly one quasihomogeneous normal surface singularity of type $(S)$. 
\end{lem}

\begin{rem}
Although the quasihomogeneous normal surface singularity in Lemma~\ref{lem-type-S-quasi} is unique, the $\bC^*$-action with positive weights is never unique.
\end{rem}

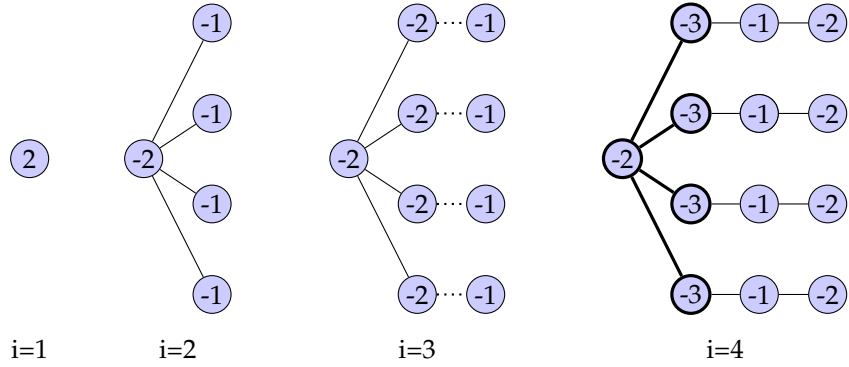
\begin{figure}\centering
\begin{tikzpicture}
  [scale=.6,auto=left]

\tikzset{curve/.style={circle,draw = black,fill=blue!20, minimum size=0.5cm, inner sep = 0cm}}
\tikzset{contrcurve/.style={very thick,circle,draw = black,fill=blue!20, minimum size=0.5cm, inner sep = 0cm}}

  \node (i0) at (0,0.8) {i=1};
  \node (i0) at (3.25,0.8) {i=2};
  \node (i0) at (8.5,0.8) {i=3};
  \node (i0) at (15.25,0.8) {i=4};

  \node [curve] (n00) at (0,5)  {2};

  \node [curve] (n10) at (2.5,5)  {-2};
  \node [curve] (n11) at (4,8)  {-1};
  \node [curve] (n12) at (4,6)  {-1};
  \node [curve] (n13) at (4,4)  {-1};
  \node [curve] (n14) at (4,2)  {-1};

  \node [curve] (n20) at (7,5)  {-2};
  \node [curve] (n21) at (8.5,8)  {-2};
  \node [curve] (n22) at (8.5,6)  {-2};
  \node [curve] (n23) at (8.5,4)  {-2};
  \node [curve] (n24) at (8.5,2)  {-2}; 
  \node [curve] (n25) at (10,8)  {-1};
  \node [curve] (n26) at (10,6)  {-1};
  \node [curve] (n27) at (10,4)  {-1};
  \node [curve] (n28) at (10,2)  {-1};

  \node [contrcurve] (n30) at (13,5)  {-2};
  \node [contrcurve] (n31) at (14.5,8)  {-3};
  \node [contrcurve] (n32) at (14.5,6)  {-3};
  \node [contrcurve] (n33) at (14.5,4)  {-3};
  \node [contrcurve] (n34) at (14.5,2)  {-3}; 
  \node [curve] (n35) at (16,8)  {-1};
  \node [curve] (n36) at (16,6)  {-1};
  \node [curve] (n36) at (16,6)  {-1};
  \node [curve] (n37) at (16,4)  {-1};
  \node [curve] (n38) at (16,2)  {-1};
  \node [curve] (n39) at (17.5,8)  {-2};
  \node [curve] (n310) at (17.5,6)  {-2};
  \node [curve] (n311) at (17.5,4)  {-2};
  \node [curve] (n312) at (17.5,2)  {-2};

  \foreach \from/\to in {n10/n11,n10/n12,n10/n13,n10/n14,n20/n21,n20/n22,n20/n23,n20/n24,n31/n35,n32/n36,n33/n37,n34/n38,n30/n31,n35/n39,n36/n310,n37/n311,n38/n312}
    \draw (\from) -- (\to);

  \foreach \from/\to in {n21/n25,n22/n26,n23/n27,n24/n28}
    \draw [dotted, thick] (\from) -- (\to);

  \foreach \from/\to in {n30/n31,n30/n32,n30/n33,n30/n34}
    \draw [very thick] (\from) -- (\to);

\end{tikzpicture}
\setlength{\captionmargin}{7pt}
\caption{Dual graphs of the reduced inverse image of $\sigma_1$ in $X_i$ in the proof of Lemma~\ref{lem-type-S-quasi}. Vertices are labeled by the self-intersection of the corresponding rational curve. Dotted edges correspond to intersection points blown up in the next step. The divisor corresponding to the thick subgraph on the right hand side gets contracted by $\pi$.}
\end{figure}

\begin{proof}[Proof of uniqueness]
Uniqueness follows directly from Fact~\ref{fact-quasihom-algebraizity} and the algebraic uniqueness result in~\cite[Thm.~2.1]{P77}. To apply Pinkham's result we only need to observe that for a quasihomogeneous normal surface singularity the exceptional curve $C$ in the minimal resolution is automatically the central curve in the sense of~\cite[Sect.~2]{P77}.
\end{proof}

\begin{proof}[Proof of existence]
We prove existence in several steps: 
\begin{enumerate}
 \item Let $X_1:=|\sO_{\bP^1}(2)|$ be the total space of the line bundle $\sO_{\bP^1}(2)$ with zero section $\sigma_1\subset X_1$. There is a natural $\bC^*$-action on $X_1$ with $X_1^{\bC^*}=\sigma_1$. Observe that $(\sigma_1\cdot\sigma_1)=2$.
 \item Let $p_2:X_2\to X_1$ be the blowing up of $X_1$ in the set $S\subset\bP^1\cong\sigma_1\subset X_1$. The strict transform of $\sigma_1$ is denoted by $\sigma_2\subset X_2$ and the $p_2$-exceptional curves are denoted by $C''_1,\cdots,C''_4$. Observe that $\bC^*$ still acts on $X_2$, $(\sigma_2\cdot\sigma_2)=-2$ and $(C''_i\cdot C''_i)=-1$ for $1\leq i\leq 4$.
 \item For any $1\leq i\leq 4$ the set $C''_i\backslash \sigma_2$ contains exactly one $\bC^*$-fixed point $c_i$. Let $p_3:X_3\to X_2$ be the blowing-up in $\{c_1,\cdots,c_4\}$ and let $D_i$ be the exceptional curve over $c_i$. Let further $\sigma_3=p_3^{-1}(\sigma_2)$ and let $C'_i$ be the strict transform of $C''_i$. The $\bC^*$-action extends to $X_3$, $(\sigma_3\cdot\sigma_3)=-2$ and $(C'_i\cdot C'_i)=-2$.
 \item Finally, let $p_4:\tilde{X}=X_4\to X_3$ be the blowing-up of the $\bC^*$-fixed points $C'_i\cap D_i$. We write $C:=p_4^{-1}(\sigma_3)$ and $C_i$ for the strict transform of $C'_i$. Then $(C\cdot C)=-2$, $(C_i\cdot C_i)=-3$ so that the snc divisor $C+C_1+\cdots + C_4$ satisfies the requirements of Definition~\ref{cond-S}.
 \item With $C_0:=C$, a short calculation shows that the intersection matrix $(C_i\cdot C_j)_{0\leq i,j\leq 4}$ is negative definite. Thus~\cite{Grau62} implies that the subset $E=C+C_0+\cdots + C_4\subset\tilde{X}$ is the exceptional set of a resolution. More precisely, there exists a point $p$ on a reduced complex space $X$ together with a holomorphic map $\pi:\tilde{X}\to X$ such that $\pi^{-1}(p)_\red=E$ and $\pi:\tilde{X}\backslash E\to X\backslash\{p\}$ is biholomorphic. We may replace $X$ by its normalization so that it becomes normal.
\end{enumerate}
It follows from the construction that $p\in X$ is a normal surface singularity of type $(S)$. The map $\pi$ is its minimal resolution. Let us show that $p\in X$ is quasihomogeneous.

The proper morphism $\pi:\tilde{X}\times\bC^*\to X\times\bC^*$ is a topological quotient. In particular, the $\bC^*$-action on $\tilde{X}$ descends to a topological group action $\bC^*\times X\to X$ which is holomorphic outside $p\in X$. Thus the map $\bC^*\times X\to X$ is holomorphic. It remains to show that the weights on $T_pX$ all have the same sign. This follows immediately from the fact that the curve $C\subset\tilde{X}$ is point-wise fixed by $\bC^*$.
\end{proof}

\begin{lem}\label{lem-type-S-several}
Let $S\subset\bP^1$ be a subset of cardinality four. Then, up to isomorphism of complex space germs, there exist at least two normal surface singularities of type $(S)$.
\end{lem}

\begin{proof}
Observe that by Lemma~\ref{lem-type-S-quasi} there exists a normal surface singularity of type $(S')$ for any subset $S'\subset\bP^1$ of cardinality four. Since $Aut(\bP^1)=PGL(2)$ does not act transitively on the set of subsets $S'\subset\bP^1$ of cardinality four, a singularity of type $(S)$ is not taut in the sense of~\cite[Def.~1.1]{L73}. Thus~\cite[Eqn.~(4.5)]{L73} implies the claim.
\end{proof}

\begin{lem}\label{lem-type-S-min-ell}
Let $S\subset\bP^1$ be a subset of cardinality four. Then any normal surface singularity of type $(S)$ is minimally elliptic.
\end{lem}

\begin{proof}
A straightforward calculation shows that the fundamental cycle is \[{\textstyle Z=2\cdot C+\sum_iC_i.}\] Moreover it is easy to calculate that $\chi(Z)=0$ and $\chi(Z')>0$ for any $0<Z'<Z$. This implies that $p\in X$ is minimally elliptic by~\cite[Thm.~3.4, Def.~3.2]{Lau77}.
\end{proof}

\begin{proof}[Proof of Proposition~\ref{intro-ex-top-counterexample}]
Let $S=\{0,1,2,3\}$. Let $p_1\in X_1$ be a quasihomogeneous normal surface singularity of type $(S)$ given by Lemma~\ref{lem-type-S-quasi}. By Lemma~\ref{lem-type-S-several} there exists a normal surface singularity $p_2\in X_2$ of type $(S)$ that is not isomorphic to $p_1\in X_1$. By Lemma~\ref{lem-type-S-quasi} the singularity $p_2\in X_2$ is \emph{not} quasihomogeneous.

The singularities $p_i\in X_i$ are minimally elliptic by Lemma~\ref{lem-type-S-min-ell}.

By shrinking $X_i$ further we may assume that it is homeomorphic to the open cone over the link of $p_i\in X_i$. This proves Item~(\ref{it-top-counterexample-homeo}) since the link only depends on the resolution graph. Indeed, the link can be obtained from the resolution graph by the plumbing construction in~\cite[Sect.~1]{Mu61}. 

Item~(\ref{it-top-counterexample-int-coho}) is a direct consequence of Lemma~\ref{lem-int-coho-surfaces}.

To prove Item~(\ref{it-top-counterexample-poincare}) recall that $p_i\in X_i$ is Gorenstein by~\cite[Thm.~3.9]{Lau77}. Then we may apply Proposition~\ref{intro-prop-gorenstein-surface} proved in Section~\ref{ssec-gorenstein-surface}.
\end{proof}

%% file: contractions.tex
\section{Poincar\'e Lemma on holomorphically contractible spaces}\label{sec-contr}

\subsection{Notions of holomorphic contractibility}
The following definition of holomorphic contractibility is slightly less general than Reiffen's original version in~\cite[Def.~3]{Rei67}. In fact, it coincides with Reiffen's notion of ``$p$-fixed contraction'' in~\cite[Def.~4(b)]{Rei67}. We favor the more restrictive notion since it captures all examples we can think of and proofs become technically less involved in this way.

The definition of a holomorphic deformation retract coincides with Reiffen's definition in~\cite[Def.~4(a)]{Rei67}.

\begin{defn}[Holomorphic contractions]\label{defn-contraction}
Let $\sg{X}{p}$ be a reduced complex space germ and let $\sg{Y}{p}\subset\sg{X}{p}$ be a reduced subgerm. 
\begin{enumerate}
\item We say that the space germ \sg{X}{p} is \emph{holomorphically contractible to \sg{Y}{p}} if there exist representatives $X\supset Y\ni p$ of the germs together with an open neighborhood $U\subset X$ of $p$, a domain $T\subset \bC$ containing both $0$ and $1$ and a holomorphic map
\[ \phi:U\times T\to X,\]
called \emph{contraction map}, such that $\phi(p,t)=p$ for any $t\in T$, $\phi(x,1)=x$ and $\phi(x,0)\in Y$ for any $x\in U$.
\item The contraction map $\phi$ is called a \emph{holomorphic deformation retraction} if it satisfies $\phi(y,0)=y$ for any $y\in Y\cap U$. In this case, the subgerm \sg{Y}{p} is called a \emph{holomorphic deformation retract} of \sg{X}{p}.
\end{enumerate}
\end{defn}

\begin{example}\label{ex-contraction-to-subspace}
If $p$ is a smooth point of $X$ then $\sg{X}{p}$ is holomorphically contractible to the germ of any closed subspace $Y\subset X$ passing through $p$. Indeed, we may assume that $X\subset\bC^n$ is an open set and $p=0$ is the origin. Then the scalar multiplication $(x,t)\mapsto \phi(x,t):=t\cdot x$ gives a contraction map when restricted to an appropriate open subset of $X\times\bC$.
\end{example}

\begin{rem}[Holomorphic deformation retracts of normal spaces]\label{rem-contr-normal}
If $p\in X$ is a normal point and $Y\subset X$ is a holomorphic retract of $X$ at $p$, then $p$ is a normal point of $Y$. Indeed, let $\phi:U\times T\to X$ denote a holomorphic deformation retraction as in Definition~\ref{defn-contraction}. By shrinking $U$ and $X$ further we may assume that $X$ is a normal complex space. Then the identity map $\id_{Y\cap U}:Y\cap U\to Y\cap U$ factors through the map $\phi(\bullet,0):U\to Y$. After shrinking $U$ further we may certainly assume that $\phi(\bullet,0)^{-1}(\{\text{non-normal points of }Y\})\subset U $ is nowhere dense. Then \cite[Prop.~71.15]{KK83} implies that the map $\phi(\bullet,0):U\to Y$ factors through the normalization $Y^{\text{norm}}\to Y$. This implies that the identity map $\id_{Y}$ factors through the normalization of $Y$, i.e., $Y$ is normal.
\end{rem}

\begin{figure}[ht]
\begin{minipage}[t]{0.45\linewidth}
\centering
\begin{tikzpicture}
  [scale=.3,auto=left]
\tikzset{curve/.style={circle,draw = black,fill=blue!20, minimum size=0.5cm, inner sep = 0cm}}
\tikzset{contrcurve/.style={very thick,circle,draw = black,fill=blue!20, minimum size=0.5cm, inner sep = 0cm}}
  \draw (0,0) ellipse (10cm and 6cm); 
  \node (L1) at (8,-5) {$X$};
 
  \draw[->] (2,0) -- ($(2,0)!1cm!(0,0)$);
  \draw[->] (-2,0) -- ($(-2,0)!1cm!(0,0)$);
  \draw[->] (0,2) -- ($(0,2)!1cm!(0,0)$);
  \draw[->] (0,-2) -- ($(0,-2)!1cm!(0,0)$);
  \draw[->] (1.4,1.4) -- ($(1.4,1.4)!1cm!(0,0)$);
  \draw[->] (1.4,-1.4) -- ($(1.4,-1.4)!1cm!(0,0)$);
  \draw[->] (-1.4,1.4) -- ($(-1.4,1.4)!1cm!(0,0)$);
  \draw[->] (-1.4,-1.4) -- ($(-1.4,-1.4)!1cm!(0,0)$);

  \draw[->] (5,0) -- ($(5,0)!2cm!(0,0)$);
  \draw[->] (-5,0) -- ($(-5,0)!2cm!(0,0)$);
  \draw[->] (0,5) -- ($(0,5)!2cm!(0,0)$);
  \draw[->] (0,-5) -- ($(0,-5)!2cm!(0,0)$);
  \draw[->] (3.5,3.5) -- ($(3.5,3.5)!2cm!(0,0)$);
  \draw[->] (3.5,-3.5) -- ($(3.5,-3.5)!2cm!(0,0)$);
  \draw[->] (-3.5,3.5) -- ($(-3.5,3.5)!2cm!(0,0)$);
  \draw[->] (-3.5,-3.5) -- ($(-3.5,-3.5)!2cm!(0,0)$);

  \draw[->] (4.6,1.9) -- ($(4.6,1.9)!2cm!(0,0)$);
  \draw[->] (4.6,-1.9) -- ($(4.6,-1.9)!2cm!(0,0)$);
  \draw[->] (-4.6,1.9) -- ($(-4.6,1.9)!2cm!(0,0)$);
  \draw[->] (-4.6,-1.9) -- ($(-4.6,-1.9)!2cm!(0,0)$);

  \draw[->] (1.9,4.6) -- ($(1.9,4.6)!2cm!(0,0)$);
  \draw[->] (-1.9,4.6) -- ($(-1.9,4.6)!2cm!(0,0)$);
  \draw[->] (1.9,-4.6) -- ($(1.9,-4.6)!2cm!(0,0)$);
  \draw[->] (-1.9,-4.6) -- ($(-1.9,-4.6)!2cm!(0,0)$);

  \draw[very thick] (0,0) to [out=90,in=195] (5,4);
  \draw[very thick] (0,0) to [out=90,in=345] (-5,4);
  \node (L2) at (5,3.2) {$C$};
  \node (L3) at (0.2,-0.5) {$p$};
\end{tikzpicture}
\setlength{\captionmargin}{2pt}
\caption{A smooth space $X$ can be contracted to any singular curve passing through $p$.}
\label{fig:figure1}
\end{minipage}
\hspace{0.5cm}
\begin{minipage}[t]{0.45\linewidth}
\centering
\begin{tikzpicture}
  [scale=.3,auto=left]

\tikzset{curve/.style={circle,draw = black,fill=blue!20, minimum size=0.5cm, inner sep = 0cm}}
\tikzset{contrcurve/.style={very thick,circle,draw = black,fill=blue!20, minimum size=0.5cm, inner sep = 0cm}}
  \draw (0,0) ellipse (10cm and 6cm); 
  \node (L1) at (8,-5) {$X$};
  \node (L2) at (0,-0.8) {$p$};
  \draw[very thick] (-8,0) to [out=10,in=170] (0,0) to [out=-10,in=190] (8,0);
  \node (L3) at (8.7,0) {$C$};

  \draw[->] (0,2) -- (0,1);
  \draw[->] (1,1.9) -- (1,0.9);
  \draw[->] (2,1.7) -- (2,0.7);
  \draw[->] (3,1.6) -- (3,0.6);
  \draw[->] (4,1.5) -- (4,0.5);
  \draw[->] (5,1.6) -- (5,0.6);
  \draw[->] (6,1.7) -- (6,0.7);
  \draw[->] (7,1.9) -- (7,0.9);
  \draw[->] (8,2) -- (8,1);
  \draw[->] (-1,2.1) -- (-1,1.1);
  \draw[->] (-2,2.3) -- (-2,1.3);
  \draw[->] (-3,2.4) -- (-3,1.4);
  \draw[->] (-4,2.5) -- (-4,1.5);
  \draw[->] (-5,2.4) -- (-5,1.4);
  \draw[->] (-6,2.3) -- (-6,1.3);
  \draw[->] (-7,2.1) -- (-7,1.1);
  \draw[->] (-8,2) -- (-8,1);

  \draw[->] (0,-2) -- (0,-1);
  \draw[->] (1,-2.1) -- (1,-1.1);
  \draw[->] (2,-2.3) -- (2,-1.3);
  \draw[->] (3,-2.4) -- (3,-1.4);
  \draw[->] (4,-2.5) -- (4,-1.5);
  \draw[->] (5,-2.4) -- (5,-1.4);
  \draw[->] (6,-2.3) -- (6,-1.3);
  \draw[->] (7,-2.1) -- (7,-1.1);
  \draw[->] (8,-2) -- (8,-1);
  \draw[->] (-1,-1.9) -- (-1,-0.9);
  \draw[->] (-2,-1.7) -- (-2,-0.7);
  \draw[->] (-3,-1.6) -- (-3,-0.6);
  \draw[->] (-4,-1.5) -- (-4,-0.5);
  \draw[->] (-5,-1.6) -- (-5,-0.6);
  \draw[->] (-6,-1.7) -- (-6,-0.7);
  \draw[->] (-7,-1.9) -- (-7,-0.9);
  \draw[->] (-8,-2) -- (-8,-1);

  \draw[->] (0,5) -- (0,3);
  \draw[->] (1,4.9) -- (1,2.9);
  \draw[->] (2,4.7) -- (2,2.7);
  \draw[->] (3,4.6) -- (3,2.6);
  \draw[->] (4,4.5) -- (4,2.5);
  \draw[->] (5,4.6) -- (5,2.6);
  \draw[->] (-1,5.1) -- (-1,3.1);
  \draw[->] (-2,5.3) -- (-2,3.3);
  \draw[->] (-3,5.4) -- (-3,3.4);

 \draw[->] (0,-5) -- (0,-3);
  \draw[->] (-1,-4.9) -- (-1,-2.9);
  \draw[->] (-2,-4.7) -- (-2,-2.7);
  \draw[->] (-3,-4.6) -- (-3,-2.6);
  \draw[->] (-4,-4.5) -- (-4,-2.5);
  \draw[->] (-5,-4.6) -- (-5,-2.6);
  \draw[->] (1,-5.1) -- (1,-3.1);
  \draw[->] (2,-5.3) -- (2,-3.3);
  \draw[->] (3,-5.4) -- (3,-3.4);

\end{tikzpicture}
\setlength{\captionmargin}{2pt}
\caption{A one-dimensional deformation retract $C$ of a normal space $X$ is automatically smooth.}
\label{fig:figure2}
\end{minipage}
\end{figure}

\begin{example}\label{ex-action-contraction}
Suppose that $\bC^*$ acts holomorphically on a space germ $\sg{X}{p}$ with only non-negative weights, see Section~\ref{ssec-local-c-star-action}. Then the subgerm $\sg{X}{p}^{\bC^*}$ of fixed points is a holomorphic deformation retract of \sg{X}{p}. A holomorphic deformation retraction is given by $T=\{t\in\bC:\, |t|< 2\}$, a sufficiently small open neighborhood $U\subset X$ of $p$ and
\[\phi(x,t)=\begin{cases} t\cdot x &\mbox{if } t\in T\backslash\{0\} \\ {\displaystyle \lim_{s\to 0}{s\cdot x}} & \mbox{if } t=0. \end{cases}\]
These claims follow immediately from Lemma~\ref{lem-c-star-linearization}.
\end{example}

\subsection{Poincar\'e Lemmas on holomorphically contractible complex spaces}\label{ssec-poincare-contractible}
In order to formulate the main result of this section in an optimal way, we need to introduce the following notation for the cohomology groups of the cochain complex of stalks of sheaves of differential forms.

\begin{notation}\label{not-poincare-failure} Let $p\in X$ be a point on reduced complex space. We write
\[ P^i_\h(\sg{X}{p}) := h^i(\Omega^{\bullet}_\h|_{X,p},\text{d})=\frac{ker\, \text{d}:\Omega^i_\h|_{X,p}\to \Omega^{i+1}_\h|_{X,p}}{im\, \text{d}:\Omega^{i-1}_\h|_{X,p}\to \Omega^{i}_\h|_{X,p}},\quad i>0.\]
In a similar way we denote by $P^i_\text{K\"a}(\sg{X}{p})$ and $P^i_\text{K\"a/tor}(\sg{X}{p})$ the cohomology groups of the complexes of K\"ahler differential forms and K\"ahler differential forms modulo torsion, respectively. If $p\in X$ is a normal point, we define $P^i_\refl(\sg{X}{p})$ in a similar manner.

Observe that if $\bullet\in\{\h,\text{K\"a},\text{K\"a}/\tor\}$, then any holomorphic map $f:\sg{X}{p}\to \sg{Y}{q}$ between reduced complex space germs induces a pull-back map $f^*:P^i_\bullet(\sg{Y}{q})\to P^i_\bullet(\sg{X}{p})$. Moreover there exist functorial complex-linear maps
\[P^i_\text{K\"a}(\sg{X}{p})\to P^i_\text{K\"a/tor}(\sg{X}{p})\to P^i_\h(\sg{X}{p}) .\]
\end{notation}

\subsubsection*{Formulation of the main results} The following theorem covers Theorem~\ref{thm-intro-contractions-h-forms} of the introduction and likewise constitutes a considerable strengthening of the main results in~\cite{Rei67} and~\cite{Ferr70}. 

\begin{thm}[Poincar\'e Lemma for $\h$-differential forms on contractible spaces, Section~\ref{ssec-proof-poincare-contractions}]\label{thm-h-poincare-contraction}
Let $p\in X$ be a point on a reduced complex space and let $j:Y\subset X$ be the inclusion of a reduced complex subspace containing $p\in X$, $i>0$. Then the following hold:
\begin{enumerate}
 \item\label{it-thm-contr-h-1} If $\sg{X}{p}$ is holomorphically contractible to $\sg{Y}{p}$, then there exists a surjective complex-linear map
 \[P^i_\h(\sg{Y}{p})\twoheadrightarrow P^i_\h(\sg{X}{p}).\]
 \item\label{it-thm-contr-h-2} If $\sg{Y}{p}$ is a holomorphic deformation retract of $\sg{X}{p}$, then the pull-back map
 \[j^*:P^i_\h(\sg{X}{p})\cong P^i_\h(\sg{Y}{p})\]
is bijective.
\end{enumerate}
The same statements hold for K\"ahler differential forms and K\"ahler differential forms modulo torsion.
\end{thm}

\begin{rem}
In Theorem~\ref{thm-h-poincare-contraction}(\ref{it-thm-contr-h-2}) it is not sufficient to require that $X_p$ is holomorphically contractible to $Y_p$. Indeed, together with Example~\ref{ex-contraction-to-subspace} this would imply that the Poincar\'e Lemma holds for $\h$-differential forms on all reduced complex spaces.
\end{rem}

\begin{cor}[Poincar\'e Lemma for reflexive differential forms on contractible spaces I, Section~\ref{ssec-proof-cor-contr-klt-base-space}]\label{cor-contr-klt-base-space}
Let $p\in X$ be a point on a locally algebraic klt base space $X$ or let $p\in X$ be an isolated rational singularity. Let further $Y\subset X$ be a reduced complex subspace containing $p$, $i>0$. Then the following hold:
\begin{enumerate}
 \item\label{it-cor-contr-klt-base-space-1} If $\sg{X}{p}$ is holomorphically contractible to $\sg{Y}{p}$, then there exists a surjective map
 \[P^i_\h(\sg{Y}{p})\twoheadrightarrow P^i_{\text{refl}}(\sg{X}{p}).\]
 \item\label{it-cor-contr-klt-base-space-2} If $\sg{X}{p}$ is holomorphically contractible to $\sg{Y}{p}$ and $\dim_pY\leq 2$, then the sequence
\[0\to\bC\to\sO_{X,p}\xrightarrow{\text{d}}\Omega^{[1]}_{X,p}\xrightarrow{\text{d}}\cdots\xrightarrow{\text{d}}\Omega^{[n]}_{X,p}\to 0 \]
is exact.
 \item\label{it-cor-contr-klt-base-space-3} If $\sg{Y}{p}$ is a holomorphic deformation retract of $\sg{X}{p}$, then there exists a bijective map
 \[{\textstyle P^i_\h(\sg{Y}{p})\cong P^i_{\text{refl}}(\sg{X}{p})}.\]
\end{enumerate}
\end{cor}

\begin{thm}[Poincar\'e Lemma for reflexive differential forms on contractible spaces II, Section~\ref{ssec-proof-poincare-contractions}]\label{thm-refl-poincare-contraction}
Let $X$  be a normal complex space equipped with a local holomorphic $\bC^*$-action at $p\in X$. Suppose that the weights at $p\in X$ are non-negative. Then the following hold:
\begin{enumerate}
 \item\label{it-1-thm-poincare-reflexive} If the germ at $p$ of the fixed locus is not contained in the singular locus of $X$, i.e., if $\sg{X}{p}^{\bC^*}\not\subset X_{sing}$, then
\[{\textstyle P^i_{\refl}(\sg{X}{p})=0\quad\forall i> dim_p(X^{\bC^*}).}\]
 \item\label{it-2-thm-poincare-reflexive} If in addition to~(\ref{it-1-thm-poincare-reflexive}) the fixed locus $X^{\bC^*}$ is of dimension one, then the sequence
\[0\to\bC\to\sO_{X,p}\xrightarrow{\text{d}}\Omega^{[1]}_{X,p}\xrightarrow{\text{d}}\cdots\xrightarrow{\text{d}}\Omega^{[dim_p(X)]}_{X,p}\to 0 \]
is exact.
\end{enumerate}
\end{thm}

\subsubsection*{Strategy of proof of these results} We will give a unified proof of Theorem~\ref{thm-h-poincare-contraction} which works for all classes of differential forms specified there. This increases slightly the necessary technical preparations and the reader will observe that if one fixes a class of differential forms the line of arguments admits shortcuts.

After some preparations the proofs of the above results will be given in Subsections~\ref{ssec-proof-poincare-contractions} and~\ref{ssec-proof-cor-contr-klt-base-space}.

\subsection{Preparation: Differentiation and line integrals of functions with values in coherent sheaves}
For lack of an adequate reference we give a formulation of the following fact. We denote the standard coordinate on $\bC$ by $t$.

\begin{prop}[Differentiation and line integrals of functions with values in coherent sheaves]\label{fact-int-diff-coh-sheaves}
There exist unique $\sO_X(X)$-linear maps
\[
\begin{array}{lllrcl}
 \Gamma(X\times T,\pr_X^*(\sF))&\to&\Gamma(X\times T,\pr_X^*(\sF)), & s & \mapsto & \frac{d}{dt}s, \vspace{5pt}\\
 \Gamma(X\times T,\pr_X^*(\sF))&\to&\Gamma(X,\sF), & s&\mapsto&\int_\gamma s\text{dt},
\end{array}
\]
where $T\subset\bC$ is any open subset, $X$ is any complex space, $\sF$ is any coherent sheaf on $X$ and $\gamma:[a,b]\to T$ is any continuous path such that the following properties are satisfied:
\begin{enumerate}
\item\label{it-int-coh-sheaves-T} (Functoriality I). Let $T'\subset T$ be an open subset such that $\text{im}(\gamma)\subset T'$. Then 
\[\begin{array}{rcl}\bigl(\frac{d}{dt}s\bigr)|_{X\times T'}  & = & \frac{d}{dt}(s|_{X\times T'}),  \vspace{5pt} \\ \int_\gamma s\text{dt} & = & \int_\gamma s|_{X\times T'}\text{dt}. \end{array}\]
\item\label{it-int-coh-sheaves-func} (Functoriality II). If $f:Y\to X$ is a holomorphic map between complex spaces equipped with coherent sheaves $\sG$ and $\sF$, respectively, and if $\phi:f^*\sF\to\sG$ is a morphism of coherent sheaves, then
\[\begin{array}{rcll}\phi\bigl(\int_\gamma s\text{dt}\bigr) & = & \int_\gamma(\phi(s))\text{dt} & \in\Gamma(Y,\sG), \vspace{5pt} \\ \frac{d}{dt}\phi_T(s) & = & \phi_T(\frac{d}{dt}s) & \in \Gamma(Y\times T,\pr_{Y}^*(\sG)) \end{array}\]
for any section $s\in\Gamma(X\times T,\pr_X^*(\sF))$. Here, by $\phi_T$ we denote the induced map $\phi_T:(f\times\id_T)^*\pr_{X}^*(\sF)\to\pr_{Y}^*(\sG)$.
\item\label{it-int-coh-sheaves-add} (Additivity). If $a\leq c\leq b$, then $\int_\gamma s\text{dt}=\int_{\gamma|_{[a,c]}}s\text{dt} + \int_{\gamma|_{[c,b]}}s\text{dt}.$
\item\label{it-int-coh-sheaves-norm} (Normalization). If $\sF=\sO_X$ and $s\in\Gamma(X\times T,\pr_X^*\sO_X)=\Gamma(X\times T,\sO_{X\times T})$, then the following hold.
 \begin{itemize}
\item  The holomorphic function $\frac{d}{dt}s:X\times T\to\bC$ is the partial derivative of  $s:X\times T\to\bC$ w.r.t. the standard coordinate $t$ on $T\subset\bC$. In other words,
\[{\textstyle (\frac{d}{dt}s)(p,t)=\frac{d}{dt'}|_{t'=t}s(p,t').}\]
\item  The value of the holomorphic function $\int_\gamma s\text{dt}:X\to\bC$ at some point $p\in X$ is the line integral of the holomorphic function $s(p,-):T\to \bC$ along the path $\gamma$. In other words,
\[{\textstyle \bigl(\int_\gamma s\text{dt}\bigr)(p)=\int_\gamma s(p,t)\text{dt}\quad\forall p\in X.} \] 
\end{itemize}
\end{enumerate}
\end{prop}

\subsubsection{Strategy of the proof of Proposition~\ref{fact-int-diff-coh-sheaves}} The proof of Proposition~\ref{fact-int-diff-coh-sheaves} consists of four parts: For both collections of maps $\int\,\text{dt}$ and $\frac{d}{dt}$ we need to show existence and uniqueness. Since the proofs for $\int\,\text{dt}$ and $\frac{d}{dt}$ proceed along the same lines, we will only outline the arguments in the more involved case of the line integral $\int\,\text{dt}$, leaving the analogous proof for $\frac{d}{dt}$ to the reader.

During the proof we will denote the usual line integral of a holomorphic function $X\times T\to\bC^s$ along a continuous path $\gamma:[a,b]\to T$ by
\[{\textstyle \int^{\text{usual}}_\gamma f\text{dt}:X\to\bC^s.}\]
In other words, $\bigl(\int_\gamma^\text{usual}f\text{dt}\bigr)(p)=\int_\gamma f(p,t)\text{dt}$ for any $p\in X$.

\subsubsection{Uniqueness of the line integral}
Let $T$, $X$, $\sF$, $s\in\Gamma(X\times T,\pr_X^*(\sF))$ and $\gamma:[a,b]\to T$ be as in Proposition~\ref{fact-int-diff-coh-sheaves}. We need to show that $\int_\gamma s\text{dt}\in\Gamma(X,\sF)$ is already determined by the properties listed in Proposition~\ref{fact-int-diff-coh-sheaves}.

To this end let us choose an integer $m\geq 1$, real numbers $a=a_0\leq a_1\leq \cdots\leq a_m=b$, Stein open subsets $T_1,\cdots,T_m\subset T$ such that $\gamma_u:=\gamma|_{[a_{u-1},a_u]}:[a_{u-1},a_u]\to T_u$ and a covering $X=\bigcup_{i\in I}X_i$ by Stein open subsets $X_i\subset X$ together with non-negative integers $r_i$ and $\sO_{X_i}$-linear surjections $\alpha_i:\sO^{\oplus r_i}_{X_i}\to \sF|_{X_i}$.

The Stein property implies that the induced maps
\[{\textstyle \alpha_{i,u}:\Gamma(X_i\times T_u,\pr_X^*(\sO^{\oplus r_i}_{X_i}))\to\Gamma(X_i\times T_u,\pr_X^*(\sF)) }\]
are surjective for $i\in I$ and $1\leq u\leq m$. In particular, we may choose in addition preimages $s_{i,u}\in\Gamma(X_i\times T_u,\pr_X^*(\sO^{\oplus r_i}_{X_i}))$ of $s|_{X_i\times T_u}$, i.e., $\alpha_{i,u}(s_{i,u})=s|_{X_i\times T_u}$.

Using the properties required in the proposition we can now calculate
\begin{align*}
{\textstyle \int_\gamma s\text{dt}|_{X_i}}&\,= \,{\textstyle \int_\gamma s|_{X_i\times T}\text{dt}}   &&  \text{by Property } (\ref{it-int-coh-sheaves-func})\\
& \,= \,{\textstyle \sum_u\int_{\gamma_u}s|_{X_i\times T}\text{dt} } &&\text{by Property } (\ref{it-int-coh-sheaves-add})\\
& \,= \, {\textstyle \sum_u\int_{\gamma_u}s|_{X_i\times T_u}\text{dt}} &&\text{by Property } (\ref{it-int-coh-sheaves-T})\\
& \,= \,{\textstyle  \sum_u\alpha_{i}\bigl(\int_{\gamma_u}s_{i,u}\text{dt}\bigr)} &&\text{by Property } (\ref{it-int-coh-sheaves-func})\\
& \,= \,{\textstyle  \sum_u\alpha_{i}\bigl(\int_{\gamma_u}^\text{usual}s_{i,u}\text{dt}\bigr)} && \text{by Properties } (\ref{it-int-coh-sheaves-func}),(\ref{it-int-coh-sheaves-norm}).
\end{align*}
This finishes the proof of the uniqueness part.\qed

\subsubsection{Existence of the line integral} The basic idea of the proof of the existence part is to use the formula in the proof of the uniqueness part as a definition. Let $T$, $X$, $\sF$, $s\in\Gamma(X,\pr_X^*(\sF))$ and $\gamma:[a,b]\to T$ be as in Proposition~\ref{fact-int-diff-coh-sheaves}.

Fix some Stein open subset $X'\subset X$ together with an $\sO_{X'}$-linear surjection $\alpha:\sO_{X'}^{\oplus r}\to\sF|_{X'}$.

\begin{notation}\label{not-integration-defn}
Let $\xi=((a_u)_{0\leq u\leq m},(T_u)_{1\leq u\leq m},(s_u)_{1\leq u\leq m})$ be a triple consisting of a sequence $a=a_0\leq\cdots\leq a_m=b$ of real numbers, Stein open subset $T_u\subset T$ such that $\gamma([a_{u-1},a_u])\subset T_u$ and preimages $s_u\in\Gamma(X'\times T_u,\pr_X^*(\sO_{X'}^{\oplus r}))$ of $s|_{X'\times T_u}$ under the map induced by $\alpha$. Then we define
\begin{equation}\label{eqn-integration-defn}
{\textstyle \int_{\gamma,\xi} s\text{dt}_{X',\,\alpha}:=\alpha\bigl(\sum_{u=1}^m\int^\text{usual}_{\gamma|_{[a_{u-1},a_u]}}s_u\text{dt}\bigr)\,\,\,\in\,\,\,\Gamma(X',\sF). }
\end{equation}
\end{notation}

Observe that such choices of $s_u$ as in Notation~\ref{not-integration-defn} always exist by the Stein property.

\begin{plainclaim}\label{claim-integration-existence-1}
The value of $\int_{\gamma,\xi}s\text{dt}_{X',\,\alpha}$ does not depend on the choice of $\xi$.
\end{plainclaim} 

\begin{proof}[Proof of Claim~\ref{claim-integration-existence-1}]
Let us first prove that it does not depend on the choice of $s_u$. Fix some $u$ and let $s_u'\in\Gamma(X'\times T_u,\pr_{X'}^*(\sO_{X'}^{\oplus r}))$ be another preimage of $s|_{X'\times T_u}$. By the Stein property of $X'$ the kernel $\sI=\ker(\sO_{X'}^{\oplus r}\xrightarrow{\alpha}\sF)$ is generated by finitely many sections $\eta_1,\cdots,\eta_w\in\Gamma(X',\sI)$. Then any section of $\ker(\pr_{X'}^*(\sO_{X'}^{\oplus r})\to\pr_{X'}^*(\sF))=\pr_X^*(\sI)$ lies in the span of the sections $\pr_{X'}^*(\eta_1),\cdots,\pr_{X'}^*(\eta_w)\in\Gamma(X'\times T_U,\pr_{X'}^*(\sI))$, since $X'\times T_u$ is likewise Stein. In particular, there exist holomorphic functions $f_1,\cdots,f_w:X'\times T_u\to\bC$ such that 
\[{\textstyle s_u-s_u'=f_1\cdot\pr_{X'}^*(\eta_1)+\cdots+f_w\cdot\pr_{X'}^*(\eta_w).}\]
Writing $\gamma_u:=\gamma|_{[a_{u-1},a_u]}$ we can now calculate
\[{\textstyle \int^\text{usual}_{\gamma_u}(s_u-s_u')\text{dt}=\sum_j\int^\text{usual}_{\gamma_u}f_j\cdot\pr_{X'}^*(\eta_j)\text{dt}=\sum_j\Bigl(\int^\text{usual}_{\gamma_u}f_j\text{dt}\Bigr)\cdot\pr_{X'}^*(\eta_j)\in\Gamma(X',\sI)}\]
which in fact shows independence from the choice of $s_u$.

It is obvious that the value of~(\ref{eqn-integration-defn}) does not change if one replaces the sequence $(a_u)_u$ by a finer subdivision of $[a,b]$ or if one shrinks $T_u$. This proves the independence from the choice of $\xi$ since one can pass from one $\xi$ to another by a finite sequence of such steps and their inverses.
\end{proof}

Thus we can simply write $\int_\gamma s\text{dt}_{X',\,\alpha}\in\Gamma(X',\sF)$. This section obviously does not depend on the open set $T\supset\text{im}(\gamma)$, but it may still depend on the choice of $\alpha$.

\begin{plainclaim}\label{claim-integration-existence-2}
The value of $\int_\gamma s\text{dt}_{X',\,\alpha}$ does not depend on the choice of $\alpha$.
\end{plainclaim}

\begin{proof}[Proof of Claim~\ref{claim-integration-existence-2}]
Suppose first that $\alpha':\sO_{X'}^{\oplus r'}\to\sF$ factors as $\alpha'=\alpha\circ\delta$ for some $\delta:\sO_{X'}^{\oplus r'}\to\sO_{X'}^{\oplus r}$. Then we can certainly choose the lifts $s_u$ in Notation~\ref{not-integration-defn} such that they are compatible with $\delta$ and this easily implies that $\int_\gamma s\text{dt}_{X',\,\alpha}=\int_\gamma s\text{dt}_{X',\,\alpha'}$.

We can pass from $\alpha$ to an arbitrary $\alpha'$ by such a step and its inverse. This proves Claim~\ref{claim-integration-existence-2}.
\end{proof}

Using Claims~\ref{claim-integration-existence-1} and~\ref{claim-integration-existence-2} we write $\int_\gamma s\text{dt}_{X'}\in\Gamma(X',\sF)$ for the unique section satisfying $\int_\gamma s\text{dt}_{X'}=\int_{\gamma,\xi} s\text{dt}_{X',\,\alpha}$ for any choice of $\xi$ and $\alpha$.

Let $X''\subset X'$ be an open Stein subset. Then $\int_\gamma s\text{dt}_{X'}|_{X''}=\int_\gamma s\text{dt}_{X''}$ by construction. In particular these sections glue and we may introduce the following notation.

\begin{notation}
Let $\int_\gamma s\text{dt}\in\Gamma(X,\sF)$ be the unique section satisfying
\[{\textstyle \int_\gamma s\text{dt}|_{X'}=\int_\gamma s\text{dt}_{X'}}\]
for any open Stein subset $X'\subset X$.
\end{notation}

\begin{plainclaim}\label{claim-integration-existence-3}
The sections $\int_\gamma s\text{dt}\in\Gamma(X,\sF)$ constructed above satisfy Properties~(\ref{it-int-coh-sheaves-T}),~(\ref{it-int-coh-sheaves-func}),~(\ref{it-int-coh-sheaves-add}) and~(\ref{it-int-coh-sheaves-norm}) of Proposition~\ref{fact-int-diff-coh-sheaves}.
\end{plainclaim}

\begin{proof}[Proof of Claim~\ref{claim-integration-existence-3}]
Property~(\ref{it-int-coh-sheaves-T}) is obvious. Properties~(\ref{it-int-coh-sheaves-add}) and~(\ref{it-int-coh-sheaves-norm}) can be checked locally on open Stein subsets $X'\subset X$ and thus follow from Notation~\ref{not-integration-defn} and Claim~\ref{claim-integration-existence-1}.

Property~(\ref{it-int-coh-sheaves-func}) can be checked locally on an open Stein subsets $Y'\subset Y$ such that there exists an open Stein subset $X'\subset X$ containing $f(Y')$ and quotient maps $\alpha_{X'}:\sO_{X'}^{\oplus r_X}\to \sF$ and $\alpha_{Y'}:\sO_{Y'}^{\oplus r_Y}\to\sG$ together with a map $\tau$ fitting into a commutative diagram
\[ {\textstyle \begin{array}{l}\xymatrix{f^*(\sF)|_{Y'} \ar[r]^(.6){\phi|_{Y'}} & \sG|_{Y'} \\ f^*(\sO_X^{\oplus r_X})|_{Y'} \ar[u]^{\alpha_{X'}} \ar[r]_(.6)\tau & \sO_{Y'}^{\oplus r_Y} \ar[u]_{\alpha_{Y'}}. }\end{array}}\]
Given a section $s\in\Gamma(X\times T,\pr_X^*(\sF))$ we can certainly choose the two sets of data required in Notation~\ref{not-integration-defn} in order to calculate $\int_\gamma s\text{dt}|_{X'}$ and $\int_\gamma \phi(s)\text{dt}|_{Y'}$ such that they are compatible with respect to $\tau$. Then Property~(\ref{it-int-coh-sheaves-func}) follows immediately.
\end{proof}

The assignment $s\mapsto \int_\gamma s\text{dt}$ can be easily verified to be $\sO_X(X)$-linear using the construction. Thus Claim~\ref{claim-integration-existence-3} finishes the proof of the existence of the line integral.\qed

\subsubsection{Further remarks} We gather some remarks needed in the sequel.

\begin{notation}\label{not-restr-to-slices}
For any section $s\in\Gamma(X\times T,\pr_X^*\sF)$ and any $t\in T$ we denote the restriction of $s$ to $\sF|_{X\times\{t\}}$ by $s_t\in\Gamma(X,\sF)$.
\end{notation}

\begin{rem}\label{rem-fundamental-2}
Using Notation~\ref{not-restr-to-slices} observe the formula
\[\int_\gamma\frac{d}{dt}s\,\text{dt}=s_{\gamma(b)}-s_{\gamma(a)}\]
for any continuous path $\gamma:[a,b]\to X $. It can be verified by reduction to the case $\sF=\sO_X$, following the same line of arguments as in the proof of Proposition~\ref{fact-int-diff-coh-sheaves}.
\end{rem}

\begin{rem}\label{rem-q-vs-deriv-d}
Let $X$ be a reduced complex space and $T\subset\bC$ an open subset. We denote by $q^i:\Omega^i_\h|_{X\times T}\to \pr_X^*\Omega^i_\h|_X$ the projection map of the product decomposition in Proposition~\ref{prop-h-product} and 
\[E:=\frac{\partial}{\partial t}\in \Gamma(X\times T,\pr_X^*(T_X))\subset \Gamma(X\times T,T_{X\times T})\]
is the vector field associated with the standard coordinate $t$ on $T$. Then the formulas 
\begin{enumerate}[label={(\alph*)}] 
 \item\label{it-rem-q-vs-deriv} $q^i\LieDer_E(\alpha)=\frac{d}{dt}q^i(\alpha)$, and
 \item\label{it-rem-q-vs-int} $\int_\gamma q^{i+1}d(\alpha)\,\text{dt}=d\int_\gamma q^i\alpha\,\text{dt}$
\end{enumerate}
hold for any $\alpha\in\Gamma(X\times T,\Omega^i_\h|_{X\times T})$ and any continuous path $\gamma$ in $T$. Similar statements hold for K\"ahler differential forms, K\"ahler differential forms modulo torsion and, if $X$ is normal, for reflexive differential forms.

To see these claims we first observe that they can be checked easily by a computation in local coordinates in the case when $X$ is smooth. The singular case can be reduced to the smooth case in the following way:

For K\"ahler differential forms we use the properties listed in Proposition~\ref{fact-int-diff-coh-sheaves} to reduce to the case when $X$ and $T$ are both Stein and $X$ admits a closed embedding $X\subset M$ into a complex manifold $M$ so that $\Omega^i_M\to\Omega^i_X$ is surjective. Then Formulas~\ref{it-rem-q-vs-deriv}and~\ref{it-rem-q-vs-int} for $M$ easily imply the same statements for $X$.

For all other classes we use the inclusion $\Omega^i_X/\tor\subset\Omega^i_\h|_X\subset j_*\Omega^i_{X_\sm}$ to reduce to the smooth case.
\end{rem}

\subsection{Application to closed differential forms}

\begin{prop}\label{prop-contractions-key}
Let $X$ be a reduced complex space, $T\subset\bC$ a domain containing $0$ and $1$, and let $\alpha\in\Gamma(X\times T,\Omega^i_\h|_{X\times T})$ be a \emph{closed} $\h$-differential form of positive degree $i>0$ on the product. Then there exists an $\h$-differential form $\beta\in\Gamma(X,\Omega^{i-1}_\h|_X)$ of degree $i-1$ on $X$ such that
\begin{enumerate}
\item\label{it-prop-key-contractions-1} $j_1^*\alpha - j_0^*\alpha = d\beta$, where $j_t:X\cong X\times\{t\}\subset X\times T$ is the inclusion, and
\item\label{it-prop-key-contractions-2} if $m_p\subset\sO_X$ is the vanishing ideal of some point $p\in X$ and $r\geq 0$ is a non-negative integer, then
\[\alpha\in\Gamma\bigl(X\times T,\pr_X^{-1}(m^r_p)\cdot\Omega^i_\h|_{X\times T}\bigr)\quad\implies\quad\beta\in\Gamma\bigl(X,m^r_p\cdot\Omega^{i-1}_\h|_X\bigr). \]
\end{enumerate}
The same statements hold for K\"ahler differential forms, K\"ahler differential forms modulo torsion and, if $X$ is normal, for reflexive differential forms.
\end{prop}

\begin{proof}[Proof in the setup of $\h$-differential forms]
Since $T$ is connected by assumption we may choose a path
\[\gamma:[a,b]\to T\]
in $T$ running from $\gamma(a)=0$ to $\gamma(b)=1$. We take up the notation of Remark~\ref{rem-q-vs-deriv-d}: The projection map associated with the decomposition in Proposition~\ref{prop-h-product} is denoted by $q^i: \Omega^i_\h|_{X\times T}\to\pr_X^*\Omega^i_\h|_X$ and the vector field on $X\times T$ associated with the standard coordinate $t$ on $T\subset\bC$ is denoted by $E:=\frac{\partial}{\partial t}\in\Gamma(X\times T,T_{X\times T})$. Using the notation introduced in Proposition~\ref{fact-int-diff-coh-sheaves} we can now formulate the following claim.

\begin{claim}
Conditions~(\ref{it-prop-key-contractions-1}) and~(\ref{it-prop-key-contractions-2}) are satisfied if $\beta$ is given as \[\beta:=\int_\gamma q^{i-1}\iota_E\alpha \text{dt}\,\in\,\Gamma(X,\Omega^{i-1}_\h|_X).\]
\end{claim}

\begin{proof}[Proof of Condition~(\ref{it-prop-key-contractions-1})]
We can simply calculate that
\begin{align*}
j_1^*(\alpha)-j_0^*(\alpha) &\,= \, q^i(\alpha)_{\gamma(b)}-q^i(\alpha)_{\gamma(a)}  &&  \text{using Notation } \ref{not-restr-to-slices}\\
& \,= \, \int_\gamma\frac{d}{dt}q^i(\alpha)\,\text{dt}  &&\text{by Remark } \ref{rem-fundamental-2}\\
& \,= \, \int_\gamma q^i\LieDer_E\alpha\,\text{dt}  &&\text{by Remark }\ref{rem-q-vs-deriv-d}\ref{it-rem-q-vs-deriv} \\
& \,= \, \int_\gamma q^id\iota_E\alpha\,\text{dt}  &&\text{since } d\alpha=0\\
& \,= \, d\int_\gamma q^{i-1}\iota_E\alpha\,\text{dt}  &&\text{by Remark } \ref{rem-q-vs-deriv-d}\ref{it-rem-q-vs-int} \\
& \,= \, d\beta. &&
\end{align*}
which shows the claim. In the fourth step of the preceding calculation we make use of Cartan's formula $\LieDer_E=d\circ\iota_E+\iota_E\circ d$. To see this formula for $\h$-differential forms, recall that the sheaf of $\h$-differential forms are torsion-free by Proposition~\ref{prop-h-properties}(\ref{it-prop-coh}) so that is suffices to prove the formula on the smooth locus $X_{\sm}$. On the smooth locus it follows immediately from Items~(\ref{it-prop-smooth}),~(\ref{it-prop-contraction}) and~(\ref{it-prop-lie-der}) of Propostion~\ref{prop-h-properties}.
\end{proof}

\begin{proof}[Proof of Condition~(\ref{it-prop-key-contractions-2})]
Condition~\ref{it-prop-key-contractions-2} follows immediately from Property~(\ref{it-int-coh-sheaves-func}) in Proposition~\ref{fact-int-diff-coh-sheaves} applied to the inclusion $m_p^r\cdot\Omega^{i-1}_\h|_X\subset\Omega^{i-1}_\h|_X$ of coherent sheaves on $X$, and the fact that $\iota_E$ and $q^i$ are $\sO_X$-linear, see Proposition~\ref{prop-h-properties}.
\end{proof}
\renewcommand{\qedsymbol}{}
\end{proof}
\vspace{-\baselineskip}
\begin{proof}[Proof for other classes of differential forms]
The line of arguments given above  applies verbatim to the other classes of differential forms specified in Proposition~\ref{prop-contractions-key}. Observe that Remark~\ref{rem-q-vs-deriv-d} remains valid.
\end{proof}

\subsection{Proof of theorems in Section~ \ref{ssec-poincare-contractible}}\label{ssec-proof-poincare-contractions}

\subsubsection{Proof of Theorem~\ref{thm-h-poincare-contraction} in the setup of $\h$-differential forms}
We maintain the notation of Definition \ref{defn-contraction}. In particular, a contraction map is given by
\[\phi:U\times T\to X \]
where $U\subset X$ is an open neighborhood of $p\in X$ and $T\subset \bC$ is a domain containing both $0$ and $1$. Recall that $\phi_1:=\phi(\bullet,1)=\text{id}_U:U\to X$ is the identity map and that $\phi_0:=\phi(\bullet,0):U\to U\cap Y$. We fix some index $i>0$.

Items~(\ref{it-thm-contr-h-1}) and~(\ref{it-thm-contr-h-2}) of Theorem~\ref{thm-h-poincare-contraction} are proved in the following two claims.

\begin{plainclaim}\label{plainclaim-surjectivity}
The pull-back map $\phi_0^*:P^i_\h(Y_p)\to P^i_\h(X_p)$ is surjective.
\end{plainclaim}

\begin{proof} Let $[\alpha]\in P^i_\h(X_p)$ be an element. We may choose
\begin{itemize}
 \item $X'\subset X$ an open  neighborhood of $p$ such that $[\alpha]$ is represented by a \emph{closed} section $\alpha\in\Omega^i_\h(X')$,
 \item $T'\subset T$ a relatively compact subdomain containing both $0$ and $1$, and
 \item $U'\subset X'$ an open neighborhood of $p$ such that $\phi(U'\times T')\subset X'$.
\end{itemize}
Proposition~\ref{prop-contractions-key}(\ref{it-prop-key-contractions-1}) applied to the closed differential form $\phi|_{U'\times T'}^*(\alpha)\in\Gamma(U'\times T',\Omega^i_\h|_{X\times T})$ shows that
\[\phi_1|_{U'}^*(\alpha) - \phi_0|_{U'}^*(\alpha)\in d\,\Gamma(U',\Omega^{i-1}_\h|_X),\]
which immediately implies that $[\alpha]=[\phi_1|_{U'}^*(\alpha)]=[\phi_0|_{U'}^*(\alpha)]\in\phi_0^*P^i_\h(Y_p)\subset P^i_\h(X_p)$.
\end{proof}

\begin{plainclaim}\label{plainclaim-defo-retract}
If $\sg{Y}{p}$ is a deformation retract of $\sg{X}{p}$, then the pull-back map \[{\textstyle j^*:P^i_\h(\sg{X}{p})\to P^i_\h(\sg{Y}{p})}\] is bijective.
\end{plainclaim}

\begin{proof}
By assumption the inclusion map $j:Y\to X$ satisfies $\phi_0\circ j|_{Y\cap U}=id_{Y\cap U}$. It follows that 
\[j^*\circ \phi_0^*=\id:P^i_\h(Y_p)\to P^i_\h(Y_p)\]
is the identity map. Thus Claim~\ref{plainclaim-defo-retract} is a consequence of Claim~\ref{plainclaim-surjectivity}.
\end{proof}

This finishes the proof of Theorem~\ref{thm-h-poincare-contraction}.\qed

\subsubsection{Proof of Theorem~\ref{thm-h-poincare-contraction} for other classes of differential forms}
The proof can be applied verbatim to the other cases. Recall from Section~\ref{ssec-class-diff-forms} that there exists still a pull-back of differential forms by $\phi$.\qed

\subsubsection{Proof of Theorem~\ref{thm-refl-poincare-contraction}}
Let $i\geq dim_pX^{\bC^*}$ be arbitrary and fix some element $[\alpha]\in P^i_\refl(\sg{X}{p})$. We choose an open neighborhood $X'\subset X$ of $p\in X'$ such that $[\alpha]$ is represented by a closed reflexive differential form $\alpha\in\Gamma(X',\Omega^{[i]}_X)$.

Let $B:=B_2(0)\subset\bC$ and choose a sufficiently small neighborhood $U\subset X'$ of $p$ such that the contraction map given by Example \ref{ex-action-contraction} restricts to a map
\[\phi:U\times B\to X'.\]
Restricting to $U\cong U\times \{0\}$ yields a holomorphic map $\phi_0:U\to X'^{\bC^*}$, which is a left inverse for the inclusion $U^{\bC^*}\subset U$.

\begin{plainclaim}\label{thm-refl-poincare-contraction-claim-1}
Let $W:=U_\reg\times B\cap\phi^{-1}(X'_\reg)\subset U\times B$. Then the complementary subset  $W^c:=(U\times B)\backslash W\subset U\times B$ is a closed analytic subset of codimension $\geq 2$.
\end{plainclaim}

\begin{proof}[Proof of Claim~\ref{thm-refl-poincare-contraction-claim-1}]
We write $B^*=B\backslash\{0\}$. For $t\in B^*$ the map $\phi(\cdot,t):U\to X'$ is an open immersion. In particular, the set $W$ satisfies
\[W\cap U\times B^*=U_\reg\times B^*\]
and by normality of $X$ any component of $W^c$ meeting $U\times B^*$ is of codimension at least $2$.

It remains to show that $U\times \{0\}\cap W\neq\emptyset$. Indeed, this follows from the assumption that $X^{\bC^*}\not\subset X_\sing$ at $p\in X$ and $\phi_0|_{U^{\bC^*}}=id_{U^{\bC^*}}$.
\end{proof}

By Claim~\ref{thm-refl-poincare-contraction-claim-1} the closed differential form $\phi|_W^*(\alpha)\in\Gamma(W,\Omega^{[i]}_{X\times B})$ extends to a unique closed section $\beta\in\Gamma(U\times B,\Omega^{[i]}_{X\times B})$, i.e., $\beta|_W=\phi|_W^*(\alpha)$. Proposition~\ref{prop-contractions-key} shows that
\begin{equation}\label{eqn-refl-contr}
\alpha|_U - j_0^*\beta=j_1^*\beta-j_0^*\beta\in \text{d}\,\Gamma(U,\Omega^{[i-1]}_X),
\end{equation}
where $j_t:U\cong U\times\{t\}\subset U\times B$ is the inclusion map. Item~(\ref{it-1-thm-poincare-reflexive}) of Theorem~\ref{thm-refl-poincare-contraction} follows immediately since $j^*_0\beta=0$ for $i>dim_p(X^{\bC^*})$.

\begin{plainclaim}\label{thm-refl-poincare-contraction-claim-2}
If $C:=X^{\bC^*}$ is a curve, then $C$ is smooth at $p\in C$ and the differential form $\alpha|_{C\cap X'_\reg}\in\Gamma(C\cap X'_\reg,\Omega^i_C)$ extends to a closed differential form $\gamma\in\Gamma(C\cap X'_\reg\cup\{p\},\Omega^i_C)$.
\end{plainclaim}

\begin{figure}\centering
\begin{tikzpicture}
  [scale=.7,auto=center]

\tikzset{curve/.style={circle,draw = black,fill=blue!20, minimum size=0.5cm, inner sep = 0cm}}
\tikzset{contrcurve/.style={very thick,circle,draw = black,fill=blue!20, minimum size=0.5cm, inner sep = 0cm}}

\draw (0,0) ellipse (2cm and 3cm);
\draw[->] (3,0) -- node[above] {$\phi_0$} (5,0);
\draw (0,3) to [out=260,in=80] (0,-3);
\draw (6,3) to [out=260,in=80] (6,-3);
\draw[thick] (1,1.5) to [out=280,in=70] (1,-1.5);

\draw[thick] (-2,0) to (2,0);
\draw (-1.89,1) to (1.89,1);
\draw (-1.49,2) to (1.49,2);
\draw (-1.89,-1) to (1.89,-1);
\draw (-1.49,-2) to (1.49,-2);

\node at (1.3,1.4) {$D$};
\node at (1.5,-0.3) {$v$};
\node at (0.2,-0.3) {$p$};
\node at (6.2,-0.3) {$p$};
\node at (0,-0.05) {\textbullet};
\node at (6,-0.05) {\textbullet};
\node at (1.25,-0.05) {\textbullet};
\node at (1.7,-2.5) {$U$};
\node at (0.35,-2.5) {$C$};
\node at (6.35,-2.5) {$C$};
\node at (-1.2,0.25) {\tiny $\phi_0^{-1}(p)$};

\end{tikzpicture}
\setlength{\captionmargin}{7pt}
\caption{Situation in the proof of Claim~\ref{thm-refl-poincare-contraction-claim-2}.}
\end{figure}
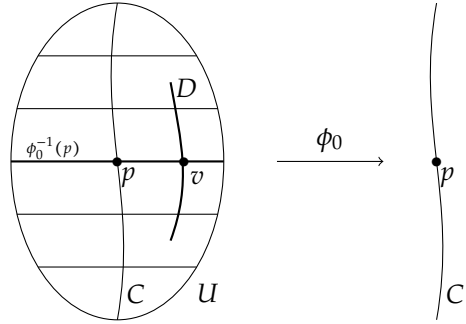

\begin{proof}[Proof of Claim~\ref{thm-refl-poincare-contraction-claim-2}]
Smoothness of $C$ follows from Remark~\ref{rem-contr-normal} and normality of $X$.

By Krull's principal ideal theorem the set $\phi_0^{-1}(\{p\})\subset U$ is of pure codimension one. Since $U$ is normal there exists a point $v\in U_\sm\cap\phi_0^{-1}(\{p\})$. Let $D\subset V$ be locally closed smooth curve passing through $v\in U$ such that $D\cap \phi_0^{-1}(\{p\})=\{v\}$. Then the restricted map
\[f:=\phi_0|_D:D\to C \]
is a non-constant holomorphic map between smooth curve germs and by Equation~(\ref{eqn-refl-contr}) the differential form
\[f|_{D\backslash\{v\}}^*(\alpha|_{C\backslash\{p\}})\in \alpha|_{D\backslash\{v\}} + \Gamma(U,\Omega^i_U)|_{D\backslash\{v\}}\subset \Gamma(D,\Omega^i_D)\subset\Gamma(D\backslash\{v\},\Omega^i_D)\]
has no pole at $v$. A computation in local coordinates shows that this already implies that $\alpha|_{C\cap X_\reg}$ has no pole at $p$ and thus extends to a differential form $\gamma$ on $C\cap X'_\reg\cup\{p\}$ as claimed.
\end{proof}

If $i=1$ and $X^{\bC^*}$ is a curve not contained in $X_\singu$, then Claim~\ref{thm-refl-poincare-contraction-claim-2} shows that in a suitable neighborhood of $p$ the reflexive differential form $j_0^*\beta$ coincides with the K\"ahler differential form $\phi_0^*\gamma$, which is exact since $\gamma$ is so by smoothness of $X^{\bC^*}$ at $p$. Equation~(\ref{eqn-refl-contr}) implies that in this case $P^1_\refl(\sg{X}{p})=0$. Together with Item~(\ref{it-1-thm-poincare-reflexive}) this implies Item~(\ref{it-2-thm-poincare-reflexive}).\qed

\begin{rem}\label{rem-refl-poincare-contraction}
In the situation of Item~(\ref{it-2-thm-poincare-reflexive}) of Theorem~\ref{thm-refl-poincare-contraction}, the proof presented above even shows that for any closed germ $\alpha\in m_p\cdot\Omega^i_\h|_{X,p}$ of degree $i\geq 1$ there exists a germ of a differential form $\beta\in m_p\cdot\Omega^{i-1}_\h|_{X,p}$ such that $\alpha=\text{d}\beta$. In fact, if $i=1$, this is true since $\text{d}m_p=\text{d}\sO_{X,p}$. If $i>1$, then in Equation~(\ref{eqn-refl-contr}) we have $j_1^*\beta-j_0^*\beta\in \text{d}\Gamma(U,m_p\cdot\Omega^{[i-1]}_{X})$ by Item~(\ref{it-prop-key-contractions-2}) of Proposition~\ref{prop-contractions-key}.
\end{rem}

\subsection{Proof of Corollary~\ref{cor-contr-klt-base-space}}\label{ssec-proof-cor-contr-klt-base-space}
The proof of Corollary~\ref{cor-contr-klt-base-space} relies on the following strengthening of~\cite[Thm.~5.4]{GKP12}.

\begin{lem}\label{lem-poincare-low-degrees-first-kind}
Let $p\in X$ be a point on a locally algebraic klt base space. Then the sequence
\[(\star)\quad 0\to\bC\to\sO_{X,p}\xrightarrow{\text{d}}\Omega^{[1]}_{X,p}\xrightarrow{\text{d}}\Omega^{[2]}_{X,p}\xrightarrow{\text{d}}\Omega^{[3]}_{X,p} \]
is exact.
\end{lem}

\begin{proof}[Proof of Lemma~\ref{lem-poincare-low-degrees-first-kind}]
Let $\pi:\tilde{X}\to X$ be a strong resolution such that the reduced fiber $F=\pi^{-1}(\{p\})_\red\subset\tilde{X}$ is an snc divisor. First, we prove the following claim.

\begin{claim}\label{claim-extension-theorem-version}
$\pi_*\bigl(\sI_F\cdot\Omega^i_{\tilde{X}}(log\, F)\bigr)=\Omega^{[i]}_X$ if $i\geq 1$.
\end{claim}

\begin{proof}[Proof of Claim~\ref{claim-extension-theorem-version}]
Pushing forward the short exact sequence $0\to \sI_F\cdot\Omega^i_{\tilde{X}}(log\, F)\to\Omega^i_{\tilde{X}}\to\Omega^i_F/\tor\to 0$ yields a left exact sequence
\[0\to\pi_*\bigl(\sI_F\cdot\Omega^i_{\tilde{X}}(log\, F)\bigr)\to\underbrace{\pi_*\Omega^i_{\tilde{X}}}_{=\,\Omega^{[i]}_X \text{ by} \atop \text{\cite[2.12]{GKP12}}}\to\underbrace{\Gamma(F,\Omega^i_F/\tor)}_{=0 \text{ by~\cite[5.22]{KM98}} \atop \text{and~\cite[1.2]{Nam01}}}, \]
which finishes the proof of the claim.
\end{proof}

In the sequel we denote the inclusion of the complement of $F$ by $j:\tilde{X}\backslash F\to \tilde{X}$. The sheaf $j_!\bC_{\tilde{X}\backslash E}$ is obtained from the constant sheaf on $\tilde{X}\backslash F$ by extension by zero. We equip the resolution
\[0\to j_!\bC_{\tilde{X}\backslash F}\to\sI_F\xrightarrow{\text{d}}\sI_F\cdot\Omega^1_{\tilde{X}}(log\, F)\xrightarrow{\text{d}}\cdots\xrightarrow{\text{d}}\sI_F\cdot \Omega^n_{\tilde{X}}(log\, F)\to 0 \]
with the \emph{filtration b\^{e}te} and consider the resulting spectral sequence
\[E^{i,j}_1=R^j\pi_*\bigl(\sI_F\cdot\Omega^i_{\tilde{X}}(log\, F)\bigr)\implies R^{i+j}\pi_*j_!\bC_{\tilde{X}\backslash F}=E^{i+j}_\infty.\]
Using Claim~\ref{claim-extension-theorem-version} the cohomology groups $h^i(\star)$ of the sequence $(\star)$ appear in the five-term exact sequence as follows:
\begin{equation}\label{eqn-five-term-klt-poincare} 0\to \underbrace{h^1(\star)}_{=E^{1,0}_2}\to E^1_\infty\to E^{0,1}_2\to \underbrace{h^2(\star)}_{=E^{2,0}_2}\to E^2_\infty\end{equation}
Observe that
\begin{itemize}
 \item $E^{i+j}_\infty=0$ for $i+j>0$ by the proof of~\cite[Lem.~14.4]{GKKP11}, and
 \item $E^{0,1}_1=0$, since $X$ has rational singularities by~\cite[Thm.~5.22]{KM98}.  
\end{itemize}
These facts together with Sequence~(\ref{eqn-five-term-klt-poincare}) imply Lemma~\ref{lem-poincare-low-degrees-first-kind}. 
\end{proof}

\begin{proof}[Proof of Items~(\ref{it-cor-contr-klt-base-space-1}) and~(\ref{it-cor-contr-klt-base-space-3}) of Corollary~\ref{cor-contr-klt-base-space}]
These claims are immediate consequences of Theorem~\ref{thm-h-poincare-contraction} and Propositions~\ref{prop-h-refl-klt} and~\ref{prop-h-refl-isolated-rational}.
\end{proof}

\begin{proof}[Proof of Item~(\ref{it-cor-contr-klt-base-space-2}) of Corollary~\ref{cor-contr-klt-base-space}]
The exactness in degrees $>2$ follows immediately from Item~(\ref{it-cor-contr-klt-base-space-1}). The exactness in degrees $\leq 2$ is covered by~\cite[Prop.~2.5]{CF02} and Lemma~\ref{lem-poincare-low-degrees-first-kind}.
\end{proof}

%% file: degeneration.tex
\section{On the degeneration of the reflexive Hodge-de Rham spectral sequence}\label{sec-degeneration}
We first prove Item~\ref{it-intro-ex-degeneration-surface} of Proposition~\ref{intro-ex-degeneration}. Recall that we claimed the following.

\begin{prop}\label{prop-degeneration-rat-surface}
Let $X$ be a normal projective surface with rational singularities. Then the spectral sequence $E^{i,j}_1=H^j(X,\Omega^{[i]}_X)\implies H^{i+j}(X,\bC)$ degenerates at $E_1$.
\end{prop}

\begin{proof}[Proof of Proposition~\ref{prop-degeneration-rat-surface}]
Let $\pi:\tilde{X}\to X$ be a strong resolution such that $\tilde{X}$ is a projective complex manifold. First, we prove two preparatory claims.

\begin{claim}\label{claim-degeneration-reflexive-hodge-numbers}
For any $(i,j)\neq (1,1)$ we have $\dim_\bC H^j(X,\Omega^{[i]}_X)=\dim_\bC H^j(\tilde{X},\Omega^i_{\tilde{X}})$.
\end{claim}

\begin{proof}[Proof of Claim~\ref{claim-degeneration-reflexive-hodge-numbers}]
Observe that $\Omega^{[i]}_X=\pi_*\Omega^i_{\tilde{X}}$ by Proposition~\ref{prop-h-refl-isolated-rational} and its proof. This already settles the case $j=0$. The other cases are proved in the sequel.
\[
\begin{array}{ll}
  i=0: & \text{by rationality of the singularities of } X,\\
  i=2: & \text{by Kodaira vanishing~\cite[Cor.~2.68]{KM98},}\\
  (i,j)=(1,2): & \text{by Serre duality } \dim_\bC H^2(X,\Omega^{[1]}_X)=\dim_\bC H^0(X,\Omega^{[1]}_X).
\end{array}
\]
This finishes the proof of Claim~\ref{claim-degeneration-reflexive-hodge-numbers}.
\end{proof}

\begin{claim}\label{claim-degeneration-betti-numbers}
For any $k\neq 2$ we have $\dim_\bC H^k(X,\bC)=\dim_\bC H^k(\tilde{X},\bC)$.
\end{claim}

\begin{proof}[Proof of Claim~\ref{claim-degeneration-betti-numbers}]
Let $F=\sum_iF_i=\pi^{-1}(X_\singu)_\red$ be the reduced exceptional set. The Leray spectral sequence for the sheaf cohomology of the constant sheaf $\bC_{\tilde{X}}$ is 
\[\tilde{E}_2^{i,j}=H^i(X,R^j\pi_*\bC_{\tilde{X}})\implies H^{i+j}(\tilde{X},\bC).\]
Since $\pi$ is a homeomorphism over $X_\sm$ we have $\tilde{E}_2^{i,j}=$ for $i>0$ and $j>0$. Moreover we know that $\tilde{E}_2^{0,j}=H^j(F,\bC)$ by Fact~\ref{fact-top-reso}. Thus the spectral sequence machinery establishes the following long exact sequence
\[\begin{tikzpicture}[>=angle 90]
  \matrix (m) [matrix of math nodes,row sep=1.5em, column sep=3em]
  {
|[name=00]|0 &  |[name=01]| H^1(X,\bC) &  |[name=02]| H^1(\tilde{X},\bC) &  |[name=03]| H^1(F,\bC) & \\
 &  |[name=11]| H^2(X,\bC) &  |[name=12]| H^2(\tilde{X},\bC) & |[name=13]| H^2(F,\bC) & |[name=14]|  \\
 &  |[name=21]| H^3(X,\bC) &  |[name=22]| H^3(\tilde{X},\bC) & |[name=23]| 0 & |[name=24]|  \\
 &  |[name=31]| H^4(X,\bC) &  |[name=32]| H^4(\tilde{X},\bC) & |[name=33]| 0. & |[name=34]|  \\
};
 \draw[->,font=\scriptsize]
 (00) edge (01)
 (01) edge (02)
 (02) edge (03)
 (11) edge (12)
 (12) edge (13)
 (21) edge (22)
 (22) edge (23)
 (31) edge (32)
 (32) edge (33);
 \path[->, font=\scriptsize]
 (03) edge[out=-5, in=175] node[pos=0.6, above] {$\tilde{d}_2$} (11)
 (13) edge[out=-5, in=175] node[pos=0.6, above] {$\tilde{d}_3$} (21)
 (23) edge[out=-5, in=175] node[pos=0.6, above] {$\tilde{d}_4$} (31);
\end{tikzpicture}\]
In this sequence, for any $r\in\{2,3,4\}$, the map $\tilde{d}_r$ can be identified with the boundary map $\tilde{E}_r^{0,r-1}\to\tilde{E}_r^{r,0}$.

Observe that $H^1(F,\bC)=0$ since the exceptional set over a rational singularity is a tree of rational curves. Thus Claim~\ref{claim-degeneration-betti-numbers} is equivalent to the map $H^2(\tilde{X},\bC)\to H^2(F,\bC)$ being surjective. This is true since by negative definiteness of the intersection matrix $(F_i\cdot F_j)_{i,j}$ the vector space $H^2(F,\bC)$ is spanned by the images of the cohomology classes $[F_i]\in H^2(\tilde{X},\bC)$ of the irreducible components. 
\end{proof}

In order to prove that the spectral sequence of Proposition~\ref{prop-degeneration-rat-surface} degenerates at $E_1$, we assume to the contrary that it does not degenerate and subsequently establish a contradiction. So let us assume that some boundary map $d^{i,j}_r:E^{i,j}_r\to E^{i+r,j-r+1}_r$ of the $E_r$-page is non-zero for some $r\geq 1$. Then we know that $\dim_\bC E_{r+1}^{i,j}<\dim_\bC E_r^{i,j}$ and $\dim_\bC E_{r+1}^{i+r,j-r+1}<\dim_\bC E_r^{i+r,j-r+1}$. In particular, we see that
\begin{equation}\textstyle \label{eqn-i-j}dim_\bC E^{i+j}_\infty = \sum_{a+b=i+j} \dim_\bC E^{a,b}_\infty < \sum_{a+b=i+j} \dim_\bC E^{a,b}_1\end{equation}
and 
\begin{equation}\textstyle \label{eqn-i-j-1}\dim_\bC E^{i+j+1}_\infty = \sum_{a+b=i+j+1} \dim_\bC E^{a,b}_\infty < \sum_{a+b=i+j+1} \dim_\bC E^{a,b}_1.\end{equation}
Let now $k\in\{i+j,i+j+1\}\backslash\{2\}$. Then the inequality
\[\dim_\bC \underbrace{H^k(X,\bC)}_{=E^k_\infty}\,\,\, \stackbin[(\ref{eqn-i-j-1})]{(\ref{eqn-i-j})}{<}\,\,\, \sum_{a+b=k}\dim_\bC \underbrace{H^b(X,\Omega^{[a]}_X)}_{=E^{a,b}_2} \,\xlongequal[\text{and } k\neq 2]{\text{Claim }\ref{claim-degeneration-reflexive-hodge-numbers}} \,\dim_\bC H^k(\tilde{X},\bC) \]
contradicts Claim~\ref{claim-degeneration-betti-numbers}. This yields the desired degeneration at $E_1$.
\end{proof}

The counterexample in Item~\ref{it-intro-ex-degeneration-threefold} of Proposition~\ref{intro-ex-degeneration} is constructed as follows.

\begin{const}\label{const-counterexample-degeneration}
For $k\geq 2$ the hypersurface singularity
\[ p=(0,0,0,0)\in X^0=\{\underbrace{x^2+y^2+z^2+w^{2k}}_{=:f(x,y,z,w)}=0\}\subset \bC^4\]
is terminal since it admits a small resolution by \cite[Ex.~2.2]{L81}. The torus $\bC^*$ acts on $X^0$ by $t\cdot (x,y,z,w)=(t^kx,t^ky,t^kz,tw)$. Let $X'\subset\bP^4$ be the Zariski closure of $X^0\subset\bC^4\subset\bP^4$. The torus action can be extended to an algebraic action on $X'$. Finally let $X\to X'$ be obtained by applying the functorial resolution in~\cite[Thm.~3.36]{Koll07} to the set of singular points $X'_\singu\backslash\{p\}$ away from $p\in X'$. In particular, the torus action lifts to an algebraic action on $X$.

Despite the algebraicity of the construction we regard $X$ as a complex space rather than a variety.
\end{const}

\begin{rem}
Construction~\ref{const-counterexample-degeneration} can also be performed if $k=1$. In this case the projective variety $X$ is toric so that the spectral sequence degenerates at $E_1$ by~\cite[Thm.~12.5]{Dan78}.
\end{rem}

Observe that by construction the singularity $p\in X$ satisfies all properties specified in Proposition~\ref{intro-ex-degeneration}.

\begin{prop}\label{prop-degeneration-counterexample}
Let $X$ be the projective complex space of Construction~\ref{const-counterexample-degeneration}, $k\geq 2$ arbitrary. Then the spectral sequence
\[E^{i,j}_1=H^j(X,\Omega^{[i]}_X)\implies H^{i+j}(X,\bC)=E^{i+j}_\infty \]
does not degenerate at $E_1$.
\end{prop}

\begin{proof}[Proof of Proposition~\ref{prop-degeneration-counterexample}] The torus $\bC^*$ acts both on the pointed space $p\in X$ and the space germ $\sg{X}{p}$. The proof relies on a careful study of the induced action of $\bC^*$ on various vector spaces naturally associated with $p\in X$ and $\sg{X}{p}$. All naturally defined maps between these vector spaces are compatible with the torus actions.

For example the torus acts by the identity on the discrete groups $H^k(X,\bZ)$. By naturality of the universal coefficient theorem for cohomology it likewise acts by the identity on $H^k(X,\bC)$ for any $k$. This together with the functoriality of the spectral sequence immediately yields the following observation.

\begin{obs}\label{obs-invariants}
If the spectral sequence $E^{i,j}_1=H^j(X,\Omega^{[i]}_X)\implies H^{i+j}(X,\bC)$ degenerates at $E_1$, then the torus $\bC^*$ acts trivially on $H^j(X,\Omega^{[i]}_X)$ for any $(i,j)$.
\end{obs}

From now on, we assume that the spectral sequence $E^{i,j}_r$ degenerates at $E_1$. Recall the local-to-global Ext spectral sequence
\[H^i(X,\sExt_{\sO_X}^j(\Omega^{[1]}_X,\Omega^{[3]}_X))\implies\text{Ext}^{i+j}_{\sO_X}(\Omega^{[1]}_X,\Omega^{[3]}_X),\]
which is a special case of the Grothendieck spectral sequence. Since $\sHom_{\sO_X}(\Omega^{[1]}_X,\Omega^{[3]}_X)\cong\Omega^{[2]}_X$, its five-term exact sequence is given by
\[0\to H^1(X,\Omega^{[2]}_X)\to \overbrace{\underbrace{\text{Ext}^1_{\sO_X}(\Omega^{[1]}_X,\Omega^{[3]}_X)}_{\bC^* \text{ acts trivially,}\atop\text{by Obs. }\ref{obs-invariants}}}^{\text{X terminal }\implies\text{ Cohen-Macaulay}\atop\cong\, H^2(X,\Omega^{[1]}_X)'\text{ by Serre duality}} \to\underbrace{\text{Ext}^1_{\sO_{X,p}}(\Omega^{[1]}_{X,p},\Omega^{[3]}_{X,p})}_{\bC^*\text{ acts non-trivially},\atop\text{by Lemma }\ref{claim-ext-weights}}\to \underbrace{H^2(X,\Omega^{[2]}_X)}_{\bC^* \text{ acts trivially,}\atop\text{by Obs. }\ref{obs-invariants}}.\]
Since the sequence is $\bC^*$-invariant, this leads to the desired contradiction by the following Lemma~\ref{claim-ext-weights}.
\end{proof}

\begin{lem}\label{claim-ext-weights}
The torus $\bC^*$ acts on $\text{Ext}^1_{\sO_{X,p}}(\Omega^{[1]}_{X,p},\Omega^{[3]}_{X,p})$ with weights $-k+1,-k+2,\cdots,k-2,k-1$. Each of these weights has multiplicity one.
\end{lem}

\begin{proof}
Observe that $\bigl(\Omega^1_{X,p}\bigr)^\tor\cong H^0_{\{p\}}(X,\Omega^1_X)=0$ by~\cite[Sect.~2.3]{G80}. Moreover the short exact sequence $0\to\Omega^1_X\to\Omega^{[1]}_X\to\bigl(\Omega^1_X\bigr)^\text{cotor}\to 0$ and $H^1_{\{p\}}(X,\Omega^1_X)=0$ by~\cite[Sect.~2.3]{G80} immediately imply that $\bigl(\Omega^1_{X,p}\bigr)^\text{cotor}=0$ so that $\Omega^1_X=\Omega^{[1]}_X$.

In particular there exists a short exact sequence
\begin{equation}\label{seq-degeneration}
0\to\sO_{X,p}\text{d}f\to \underbrace{\sO_{X,p}\text{d}x\oplus\sO_{X,p}\text{d}y\oplus\sO_{X,p}\text{d}z\oplus\sO_{X,p}\text{d}w}_{=:\sF}\xrightarrow{\phi}\Omega^{[1]}_{X,p}\to 0
\end{equation}
of $\sO_{X,p}$-modules with torus action. More precisely, we denote by $\sO_{X,p}\text{d}x\subset\Omega^{[1]}_{X,p}$ the submodule spanned by $\text{d}x$, which is closed by pull-back by the torus action on the space germ $\sg{X}{p}$. A similar definition applies to $\text{d}y,\text{d}z$ and $\text{d}w$. The module $\sO_{X,p}\text{d}f$ is the kernel of the resulting map $\phi$. Again there exist natural isomorphisms $(t\cdot)^*\sO_{X,p}\text{d}f\cong\sO_{X,p}\text{d}f$ of $\sO_{X,p}$-modules so that functoriality of spectral sequences shows that the long exact Ext-sequence
\[\hspace{-10em}\begin{tikzpicture}[>=angle 90]
  \matrix (m) [matrix of math nodes,row sep=1em, column sep=1em]
  {
|[name=00]|0 &  |[name=01]|\Homom_{\sO_{X,p}}(\Omega^{[1]}_{X,p},\Omega^{[3]}_{X,p}) &  |[name=02]|\Homom_{\sO_{X,p}}(\sF,\Omega^{[3]}_{X,p}) &  |[name=03]|\Homom_{\sO_{X,p}}(\sO_{X,p}\text{d}f,\Omega^{[3]}_{X,p}) \\
 &  |[name=11]|\text{Ext}^1_{\sO_{X,p}}(\Omega^{[1]}_{X,p},\Omega^{[3]}_{X,p}) &  |[name=12]|\text{Ext}^1_{\sO_{X,p}}(\sF,\Omega^{[3]}_{X,p}) & |[name=13]| \\
};
 \draw[->,font=\scriptsize]
 (00) edge (01)
 (01) edge (02)
 (02) edge node[auto] {$\alpha$}(03)
 (11) edge (12);
 \path[->, font=scriptsize]
 (03) edge[out=-5, in=175] (11);
\end{tikzpicture}\]
is compatible with the torus action. The last term is zero since $\sF$ is a projective module. This exhibits the module in question as the cokernel of the map $\alpha$.

A homogeneous generator $\lambda$ of the module $\Homom_{\sO_{X,p}}(\sO_{X,p}\text{d}f,\Omega^{[3]}_{X,p})$ satisfies $\lambda(\text{d}f)=\frac{dx\wedge dy\wedge dz}{w^{2k-1}}$ and is thus of degree $-2k+3k-(2k-1)=-k+1$. Sequence~(\ref{seq-degeneration}) shows that the image of $\alpha$ is spanned by $x\cdot\lambda, y\cdot\lambda,z\cdot\lambda$ and $w^{2k-1}\cdot \lambda$. In particular, a basis of the cokernel of $\alpha$ over $\bC$ is given by the residue classes of $\lambda,w\cdot\lambda,\cdots,w^{2k-2}\cdot\lambda$. This shows the lemma.
\end{proof}

%% file: lipman-zariski.tex
\section{On the Lipman-Zariski conjecture}\label{sec-lz}
In this section we prove Corollary~\ref{cor-lz}. We assume that $T_X$ is locally free and seek to prove that $X$ is smooth. Let us consider the evaluation map 
\begin{equation}\label{eqn-lz-varphi}
\varphi:T_{X,p}\otimes_{\sO_{X,p}}\bC_p\to T_pX
\end{equation}
and its dual
\[\varphi^*:m_p/m_p^2\to\Hom_{\sO_{X,p}}(T_{X,p},\bC_p).\]
It satisfies $\varphi^*([f])(V)=V(f)(p)$ for any $f\in m_p$ and $V\in T_{X,p}$.

\subsection{Proof of Corollary~\ref{cor-lz} if $\varphi$ is the zero map}
Under this additional assumption, the exterior derivative satisfies $d(m_p\cdot\Omega^{[i]}_{X,p})\subset m_p\cdot\Omega^{[i+1]}_{X,p}$. By Theorem~\ref{thm-refl-poincare-contraction} and Remark~\ref{rem-refl-poincare-contraction} the complexes
\[{\textstyle 
\begin{array}{c}\xymatrix@C=25pt@R=0pt{
0 \ar[r] &  m_p\ar[r]^{\text{d}} &   \Omega^{[1]}_{X,p} \ar[r]^(.56){\text{d}} &  \cdots \ar[r]^(.45){\text{d}} &  \Omega^{[n]}_{X,p} \ar[r] &  0 \\
0 \ar[r] &  m_p \ar[r]^(.4){\text{d}} &  m_p\cdot \Omega^{[1]}_{X,p} \ar[r]^(.65){\text{d}} & \cdots \ar[r]^(.35){\text{d}} &  m_p\cdot\Omega^{[n]}_{X,p} \ar[r] &  0 }
\end{array}}
\]
are acyclic. Then so is the quotient complex
\[{\textstyle 
\begin{array}{c}\xymatrix@C=25pt@R=0pt{
0\ar[r] & 0 \ar[r] &  \Omega^{[1]}_{X,p}\otimes\bC_p \ar[r]^(.65){\text{d}} &  \cdots \ar[r]^(.35){\text{d}} &  \Omega^{[n]}_{X,p}\otimes\bC_p \ar[r] &  0.  }
\end{array}}\]
Since $\Omega^{[1]}_{X,p}$ is free, we have $\Omega^{[i]}_{X,p}\otimes\bC_p=\bigwedge^i\bigl(\Omega^{[1]}_{X,p}\otimes\bC_p\bigr)$ and calculate
\[0 = \sum_{i=1}^{n}(-1)^i\cdot dim_\bC(\Omega^{[i]}_{X,p}\otimes\bC_p)=\sum_{i=1}^n(-1)^i\cdot{n \choose i} = -1.\]
This contradiction finishes the proof of Corollary~\ref{cor-lz} if $\varphi$ is the zero map.\qed

\subsection{Proof of Corollary~\ref{cor-lz} in the general case}

\begin{figure}[ht]\centering
\includegraphics[width=110mm]{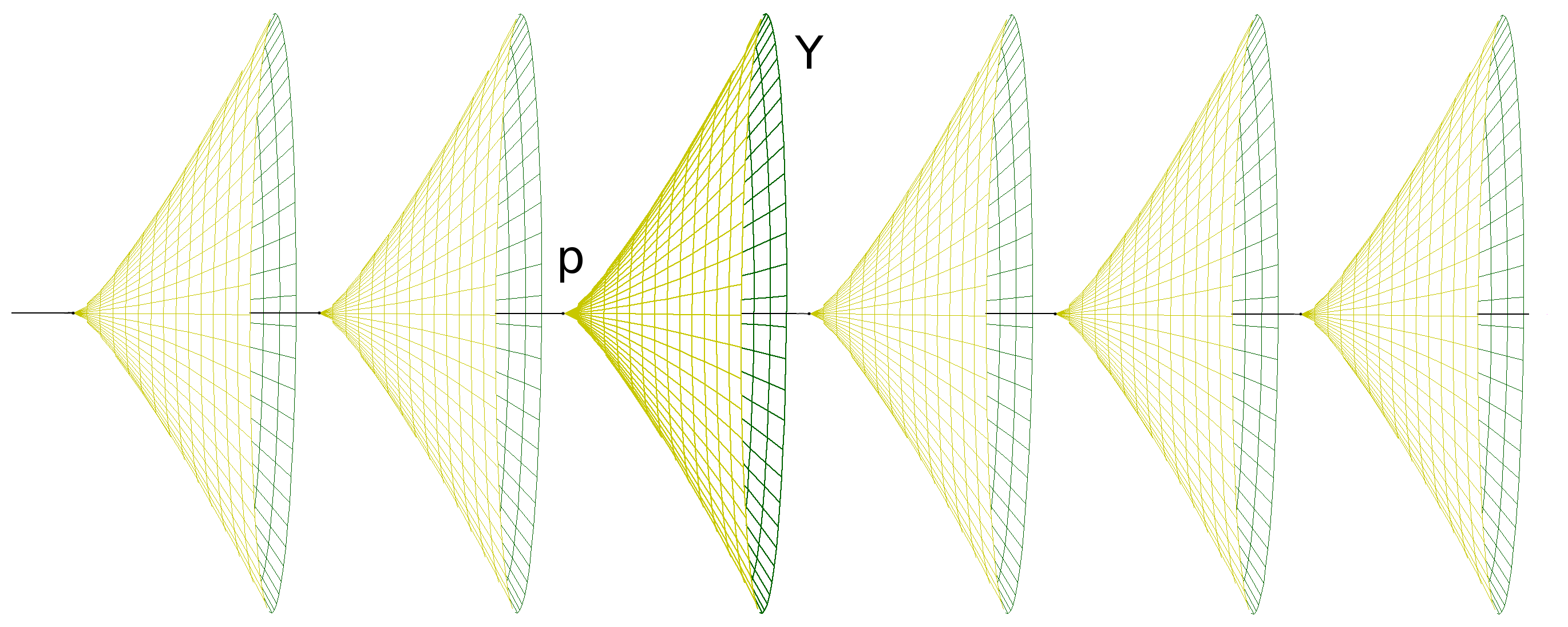}
\setlength{\captionmargin}{2pt}
\caption{Hypothetical situation in the proof of Corollary~\ref{cor-lz}, if $s=1$. The figure shows the thick drawn subspace $Y=\{f_1=0\}_\red$ and further thin drawn level sets of the holomorphic function $f_1:X\to\bC$. The flow map associated with the vector field $V_1$ interchanges these subspaces of $X$.}
\label{fig:figure1}
\end{figure}

By Lemma~\ref{lem-c-star-linearization} we may choose homogeneous functions $f_1,\cdots,f_s\in m_p$ such that $\varphi^*([f_1]),\cdots,\varphi^*([f_s])$ is a basis of $\text{im}(\varphi^*)$. By construction there exist vector fields $V_i\in T_{X,p}$, $1\leq i\leq s$, such that $V_i(f_j)(p)=\delta_{i,j}$.

Let $Y\subset X$ be a representative of the reduced space germ $\{f_1=\cdots=f_s=0\}_\red\subset \sg{X}{p}$ that is sufficiently small so that the vector fields $V_i$ and the functions $f_i$ are defined on an open neighbourhood of $Y$. Observe that the space germ $\sg{Y}{p}\subset\sg{X}{p}$ is $\bC^*$-stable.

For any $1\leq i\leq s$, let $\bC\times X\ni(t,x)\mapsto e^{t\cdot V_i}x\in X$ denote the flow associated with $V_i$. After shrinking $Y$ if necessary, there exist suitable disks $B_i=\{z:|z|<\epsilon_i\}\subset\bC$ with $\epsilon_i>0$ such that the map
\[
{\textstyle \Psi:\prod_iB_i\times Y\to X,\quad (t_1,\cdots,t_s,y)\mapsto e^{t_1\cdot V_1}\cdots e^{t_s\cdot V_s}y\in X}
\]
is well-defined.

\begin{plainclaim}\label{claim-immersion-lz}
The map $\Psi$ induces an open immersion of a neighbourhood of $\xi:=(0,\cdots,0,p)\in\prod_iB_i\times Y$ into $X$.
\end{plainclaim} 

\begin{proof}[Proof of Claim~\ref{claim-immersion-lz}]
The space germ $\sg{Y}{p}$ is of dimension at least $\dim_pX-s$ by Krull's principal ideal theorem and the space germ $\sg{X}{p}$ is irreducible by normality. Thus any embedding $\prod_iB_i\times\sg{Y}{p}\hookrightarrow\sg{X}{p}$ of space germs is an isomorphism of space germs so that~\cite[Prop.~45.9]{KK83} implies that we only need to show that the derivative
\[{\textstyle d_\xi\Psi:\bC^s\times T_pY\to T_pX, \quad d_\xi\Psi(a_1,\cdots,a_s,A)=\sum_ia_i\cdot\varphi(V_i\otimes 1)+A}\]
at $\xi$ is injective.

To prove the claimed injectivity assume that $\sum_ia_i\cdot\varphi(V_i\otimes 1)\in T_pY\subset T_pX$ for coefficients $a_i\in\bC$. For any $1\leq j\leq s$, pairing with $[f_j]\in m_p/m_p^2=T_p^*X$ gives
\[{\textstyle a_j=\sum_ia_i\delta_{i,j}=\langle [f_j],\sum_ia_i\cdot\varphi(V_i\otimes 1)\rangle\in \langle [f_j], T_pY\rangle =\{0\}. }\]
This shows that $a_j=0$ for any $j$ and thus finishes the proof of Claim~\ref{claim-immersion-lz}.
\end{proof}

Using Claim~\ref{claim-immersion-lz} the assumption of Corollary~\ref{cor-lz} implies that $\sg{Y}{p}$ is normal and  $T_Y$ is likewise locally free in a neighbourhood of $p$. If $\bC^*$ acts on $\sg{Y}{p}$ with only positive weights, then $\sg{Y}{p}$ is locally algebraic by Fact~\ref{fact-quasihom-algebraizity}, so that~\cite{Hoch77} implies smoothness of $Y$. Otherwise $Y^{\bC^*}\subset Y$ is again a curve not contained in the singular locus $Y_{\singu}$ so that the proof in the special case above applies to $Y$ and thus shows smoothness of $Y$. In any case Claim~\ref{claim-immersion-lz} proves Corollary~\ref{cor-lz}.\qed

\begin{rem}
If $p\in X$ is a surface singularity satisfying the assumptions of Corollary~\ref{cor-lz}, then \cite[Satz~(3.1)]{SW80} implies that $p\in X$ is a quasihomogeneous singularity. For quasihomogeneous singularities Hochster proved the conjecture in \cite{Hoch77}.
\end{rem}

%% file: kan.tex
\section{On KAN type non-vanishing}\label{sec-kan}
Theorem~\ref{thm-poincare-vs-kan} is an immediate consequence of the following more general result. Recall our Notation~\ref{not-poincare-failure} from Section~\ref{ssec-poincare-contractible}.

\begin{thm}[Subsection~\ref{ssec-kan-proof}]\label{thm-kan-poincare}
Let $\LB$ be an ample line bundle on a projective normal complex space $X$. Assume that $X$ has
\begin{itemize}
 \item isolated rational singularities of dimension $n\geq 4$, or
 \item isolated Cohen-Macaulay Du Bois singularities of dimension $n\geq 5$.
\end{itemize}
 Then, if $\text{WDiv}_\bQ(\sg{X}{p})/\hspace{-0.2em}\sim_\bQ$ denotes the group of local analytic Weil divisors with rational coefficients on arbitrarily small neighborhoods of $p\in X$ modulo $\bQ$-linear equivalence, we have
\[ \dim_\bC\,H^2(X,\Omega^{[1]}_X\otimes\LB^{-1})\geq\sum_{p\in X_\singu}\Bigl(\text{dim}_\bC\,P^3_\refl(\sg{X}{p}) + \dim_\bQ\text{WDiv}_\bQ(\sg{X}{p})/\hspace{-0.2em}\sim_\bQ\Bigr).\]
Moreover, the natural map $P^3_\h(\sg{X}{p})\xrightarrow{\sim} P^3_\refl(\sg{X}{p})$ is isomorphic for any $p\in X$.
\end{thm}

In Theorem~\ref{thm-kan-poincare} we can not mitigate the assumption by only assuming that $p\in X$ is an isolated Du Bois singularity of dimension $n\geq 4$, which is shown by the following proposition.

\begin{prop}[Subsection~\ref{ssec-prop-kan-optimality}]\label{prop-kan-optimality}
The singularity $p\in X$ of Construction~\ref{ex-kan-poincare} is an isolated Du Bois cone singularity of dimension $n$ such that
\begin{enumerate}
 \item\label{kan-counterex-infty} $\dim_\bQ\text{WDiv}_\bQ(\sg{X}{p})/\hspace{-0.2em}\sim_\bQ=\infty$ if $n\geq 4$, and
 \item\label{kan-counterex-poincare} $ \dim_\bC\,H^2(X,\Omega^{[1]}_X\otimes\LB^{-1}) < \dim_\bC P^3_\refl(\sg{X}{p})$ if $n\geq 11$ and $\LB$ is an ample line bundle on $X$.
\end{enumerate}
\end{prop}

\begin{const}[Isolated Du Bois singularities of dimension $n\geq 4$ violating Theorem~\ref{thm-kan-poincare}]\label{ex-kan-poincare}
We choose arbitrary elliptic curves $E_1,\cdots,E_{n-1}$ and let $E=E_1\times\cdots\times E_{n-1}$. Let further $\sM$ be a very ample line bundle on $E$ such that
\begin{enumerate}
 \item $H^i(E,\sM^{k})=0$ for $i>0$ and $k>0$, and
 \item the linear system $|\sM|$ induces a projectively normal embedding $E\to\bP^N$ for some $N>0$.
\end{enumerate} 
Finally we define $X\subset\bP^{N+1}$ to be the projective cone over the image of $E$ in $\bP^N$ with vertex $p\in X$.
\end{const}

\subsection{A result on certain isolated Du Bois singularities}\label{ssec-key-kan}
For the ease of formulation in the arguments below, we will repeatedly consider skyscraper sheaves at some point $p\in X$ of a complex space as complex vector spaces. Moreover, in order to shorten statements, we will use the following assumption.

\begin{plainass}\label{ass-du-bois}
The isolated Du Bois singularity $p\in X$ is normal, of dimension $n$ and satisfies the following:
\begin{itemize}
 \item $p\in X$ is Cohen-Macaulay and $n\geq 5$; or
 \item $p\in X$ is a rational singularity and $n\geq 4$.
\end{itemize}
\end{plainass}

The following Proposition~\ref{prop-key-kan} is the key to the proof of Theorem~\ref{thm-kan-poincare}.

\begin{prop}\label{prop-key-kan}
Let $p\in X$ be an isolated Du Bois singularity of dimensions $n\geq 4$ satisfying Assumption~\ref{ass-du-bois}. Let $\pi:\tilde{X}\to X$ be a strong resolution and let $E=\sum_rE_r=\pi^{-1}(\{x\})_\red$ be the exceptional divisor with its reduced scheme structure. Then there exists a natural inclusion
\[{\textstyle P^3_\refl(\sg{X}{p})\oplus H^2(E,\bC)/\sum_r\langle E_r|_E\rangle\subset \text{ker}\Bigl(R^1\pi_*\Omega^1_{\tilde{X}}(log\, E)\to R^2\pi_*\Omega^1_{\tilde{X}}(log\, E)\Bigr),}\]
where $\langle E_r|_E\rangle\subset H^2(E,\bC)$ is the $\bC$-span of the image of the cohomology class of the divisor $E_r$ in $H^2(E,\bC)$. 
\end{prop}

Proposition~\ref{prop-key-kan} is an immediate consequence of Lemma~\ref{lem-key-kan}, which is proved after some technical preliminaries.

\begin{lem}\label{lem-key-du-bois-vanishing}
Let $p\in X$ be an isolated singularity of dimension $n\geq 4$ satisfying Assumption~\ref{ass-du-bois} and let $\pi:\tilde{X}\to X$ be a strong resolution and let $E=\pi^{-1}(\{x\})_\red$ be the exceptional divisor with its reduced scheme structure.

Then $R^i\pi_*\sO_{\tilde{X}}=0$, $H^i(E,\sO_E)=0$ and $H^0(E,\Omega^i_E/\tor)=0$ for $1\leq i\leq 3$.
\end{lem}

\begin{proof}
The vanishing of $R^i\pi_*\sO_{\tilde{X}}$ is a consequence of \cite[Lem.~3.3]{Kov99}. Fact~\ref{fact-isolated-du-bois} implies that $H^i(E,\sO_E)=0$ for $1\leq i\leq 3$. Now a closer look at the proof of \cite[Lem.~1.2]{Nam01} shows that $H^0(E,\Omega^i_E/\tor)=0$ for $1\leq i\leq 3$.
\end{proof}

\begin{lem}\label{lem-dubois-comparison}
Let $p\in X$ be an isolated Du Bois singularity of dimension $n\geq 4$ and let $\pi:\tilde{X}\to X$ be a strong resolution. Then
\begin{enumerate}
 \item\label{it-fk-to-refl} the natural map 
 \[ \frac{\text{ker}\,d:\pi_*\Omega^3_{\tilde{X}}\to\pi_*\Omega^4_{\tilde{X}}}{\text{im}\, d:\pi_*\Omega^2_{\tilde{X}}\to\pi_*\Omega^3_{\tilde{X}}}\xrightarrow{\sim} P^3_\refl(\sg{X}{p})\]
is bijective, and
 \item\label{it-h-to-fk} if in addition $p\in X$ satisfies Assumption~\ref{ass-du-bois}, then the natural map
\[P^3_\h(\sg{X}{p})\xrightarrow{\sim}\frac{\text{ker}\,d:\pi_*\Omega^3_{\tilde{X}}\to\pi_*\Omega^4_{\tilde{X}}}{\text{im}\, d:\pi_*\Omega^2_{\tilde{X}}\to\pi_*\Omega^3_{\tilde{X}}}\]
is also bijective.
\end{enumerate}
\end{lem}

\begin{proof}
To prove Item~(\ref{it-fk-to-refl}) recall that \cite[Thm.~(1.3)]{SvS85} states that $\Omega^{[2]}_X=\pi_*\Omega^2_{\tilde{X}}$ and, if $n>4$,  $\Omega^{[3]}_X=\pi_*\Omega^3_{\tilde{X}}$. If $n=4$, then \cite[Cor.~(1.4)]{SvS85} implies that the map
\[d:\Omega^{[3]}_X/\pi_*\Omega^3_{\tilde{X}}\to\Omega^{[4]}_X/\pi_*\Omega^4_{\tilde{X}} \]
induced by the exterior derivative is injective. This completes the proof of Item~(\ref{it-fk-to-refl}).

To prove Item~(\ref{it-h-to-fk}) write $E=\pi^{-1}(\{p\})_\red$ so that Proposition~\ref{prop-h-isolated} yields a left exact sequence
\[0\to\Omega^i_\h|_X\to\pi_*\Omega^i_{\tilde{X}}\to H^0(E,\Omega^i_E/\tor) \]
for any $i>0$. Lemma~\ref{lem-key-du-bois-vanishing} immediately implies that the first map is bijective for $1\leq i\leq 3$, which proves Item~(\ref{it-h-to-fk}).
\end{proof}

\begin{lem}\label{lem-key-kan}
Let $p\in X$ be an isolated Du Bois singularity of dimension $n\geq 4$ satisfying Assumption~\ref{ass-du-bois}. Let $\pi:\tilde{X}\to X$ be a strong resolution and let $E=\pi^{-1}(\{x\})_\red=\sum_rE_r$ be the exceptional divisor with its reduced scheme structure. Then
\begin{enumerate}
\item\label{it-lem-key-kan-psi} there exists a natural isomorphism
\[\psi:\text{ker}\,\Bigl(d:R^1\pi_*\Omega^1_{\tilde{X}}\to R^1\pi_*\Omega^2_{\tilde{X}}\Bigr) \xrightarrow{\sim}  P^3_\refl(\sg{X}{p})\oplus H^2(E,\bC),\]
\item\label{it-lem-key-kan-comm} the connecting morphism $\delta:\bigoplus_rH^0(E_r,\sO_{E_r})\to R^1\pi_*\Omega^1_{\tilde{X}}$ associated with the short exact sequence $0\to\Omega^1_{\tilde{X}}\to\Omega^1_{\tilde{X}}(log\, E)\to\bigoplus_i\sO_{E_i}\to 0$ fits into a commutative diagram
\[\begin{array}{l}\xymatrix{
\bigoplus_rH^0(E_r,\sO_{E_r}) \ar[r]^(.55)\alpha \ar[d]_\delta & H^2(E,\bC) \ar[d]^{(0,\,\id)}\\
R^1\pi_*\Omega^1_{\tilde{X}} & P^3_\refl(\sg{X}{p})\oplus H^2(E,\bC)  \ar[l]_(.6){\psi^{-1}}  
}\end{array}\]
where $\alpha(H^0(E_r,\sO_{E_r}))=\langle E_r|_E\rangle \subset H^2(E,\bC)$ is the subspace spanned by the cohomology class in $H^2(E,\bC)$ of the line bundle $E_r|_E$.
\end{enumerate}
\end{lem}

\begin{proof}[Proof of Item~(\ref{it-lem-key-kan-psi}) in Lemma~\ref{lem-key-kan}] We equip the de Rham complexes $\Omega^\bullet_{\tilde{X}}$ and $\Omega^\bullet_E/\tor$ both with the \emph{filtration b\^{e}te}. In this way the pull-back $\res:\Omega^\bullet_{\tilde{X}}\to \Omega^\bullet_E/\tor$ induces a morphism $\res^{i,j}_k:{_{\tilde{X}}E}^{i,j}_k\to {_EE}^{i,j}_k$  between spectral sequences converging to an isomorphism
\[\res^{i+j}_\infty :R^{i+j}\pi_*\bC_{\tilde{X}}={_{\tilde{X}}E}^{i+j}_\infty\xrightarrow{\sim}{_EE}^{i+j}_\infty=H^{i+j}(E,\bC) \]
for $i+j>0$ by Fact~\ref{fact-top-reso}. On the $E_1$-page $\res$ is just the obvious morphism
\[R^j\pi_*\Omega^i_{\tilde{X}}={_{\tilde{X}}E}^{i,j}_1\xrightarrow{\res^{i,j}_1}{_EE}^{i,j}_1=H^j(E,\Omega^i_E/\tor).\]
Observe that
\begin{enumerate}
  \item ${_{\tilde{X}}E}^{0,j}_1=0$ for $1\leq j\leq 3$ by Lemma~\ref{lem-key-du-bois-vanishing},
  \item ${_EE}^{i,0}_1=0$ and ${_EE}^{0,j}_1=0$ for $1\leq i,j\leq 3$ by Lemma~\ref{lem-key-du-bois-vanishing}, and
  \item ${_{\tilde{X}}E}^{3,0}_2\cong P^3_\refl(\sg{X}{p})$ by Lemma~\ref{lem-dubois-comparison}.
\end{enumerate}
The observations (1)-(3) together imply that the boundary map ${_{\tilde{X}}d}_2$ on ${_{\tilde{X}}E}_2$ and $\res$ yield an isomorphism
\[({_{\tilde{X}}d}^{1,1}_2,\res^{1,1}_2):{_{\tilde{X}}E}^{1,1}_2\xrightarrow{\sim} {_{\tilde{X}}E}^{3,0}_2\oplus {_EE}^{1,1}_2, \]
which can be identified with the isomorphism in Item~(\ref{it-lem-key-kan-psi}) of the lemma.
\end{proof}

\begin{proof}[Proof of Item~(\ref{it-lem-key-kan-comm}) in Lemma~\ref{lem-key-kan}]
We pick up the notation of the proof of Item~\ref{it-lem-key-kan-psi}. Fix some $r_0$ and abbreviate $G=E_{r_0}$. We denote by ${_GE}^{i,j}_1=H^j(G,\Omega^i_G)\implies H^{i+j}(G,\bC)$ the spectral sequence associated with the \emph{filtration b\^{e}te} on $\Omega^\bullet_G$. Then the morphisms
\[\begin{array}{l}\xymatrix@R-1.2pc{
0 \ar[r] & \Omega^1_{\tilde{X}} \ar[r] & \Omega^1_{\tilde{X}}(log\, E) \ar[r] & \oplus_i\sO_{E_i} \ar[r] & 0 & \\
0 \ar[r] & \Omega^1_{\tilde{X}} \ar[r]\ar@{=}[u]\ar[d]^d & \Omega^1_{\tilde{X}}(log\, G) \ar[r]\ar[u]\ar[d]^d & \sO_{G} \ar[r]\ar[u]\ar[d]^d & 0 & (\star)_1\ar[d]^d\\
0 \ar[r] & \Omega^2_{\tilde{X}} \ar[r] & \Omega^2_{\tilde{X}}(log\, G) \ar[r] & \Omega^1_G \ar[r] & 0 & (\star)_2
}\end{array}\]
between short exact residue sequences yield a commutative diagram
\[\begin{array}{l}\xymatrix@R-1.2pc{
\bigoplus_rH^0(E_r,\sO_{E_r}) \ar[d]_\delta & H^0(G,\sO_G) \ar[d]\ar[l]\ar[r]^{d=0} & H^0(G,\Omega^1_G) \ar[d]  \\
R^1\pi_*\Omega^1_{\tilde{X}} & R^1\pi_*\Omega^1_{\tilde{X}} \ar@{=}[l]\ar[r]^{d} &  R^1\pi_*\Omega^2_{\tilde{X}} 
}\end{array}\]
and this shows that 
\[im\,\delta\subset\text{ker}\,\Bigl(d:R^1\pi_*\Omega^1_{\tilde{X}}\to R^1\pi_*\Omega^2_{\tilde{X}}\Bigr). \]
We equip the complex $\Omega^\bullet_{\tilde{X}}(log\, G)$ with two decreasing filtrations:  the \emph{filtration b\^{e}te} $F^\bullet$ and the filtration $W^\bullet$ determined by
\[W^p\Omega^\bullet_{\tilde{X}}(log\, G) = \begin{cases} \Omega^\bullet_{\tilde{X}}(log\, G) & \text{if }p\leq 0 \\ \Omega^\bullet_{\tilde{X}} & \text{if } p=1  \\ 0 & \text{if }p\geq 2  \end{cases} \]
\cite[(1.4.9)+(1.4.10)]{Del71} shows that there exists a morphism of spectral sequences
\[{_{W,F}d}^{i,j}_k:{_GE}^{i,j}_k\to {_{\tilde{X}}E}^{i+1,j+1}_k \]
of degree $(+1,+1)$ such that ${_{W,F}d}_1$ is given by the connecting morphisms of the long exact sequences associated with the residue sequences $(\star)_\bullet$ and $\pi_*$. In particular, we have
\[\delta|_{H^0(G,\sO_G)}={_{W,F}d}^{0,0}_2:H^0(G,\sO_G)\to \text{ker}\,\Bigl(d:R^1\pi_*\Omega^1_{\tilde{X}}\to R^1\pi_*\Omega^2_{\tilde{X}}\Bigr)\]
and a commutative diagram
\[\begin{array}{l}\xymatrix{
H^0(G,\sO_G) \ar[d]_{{_Gd}^{0,0}_2}\ar[r]^(.3)\delta & \text{ker}\,\Bigl(d:R^1\pi_*\Omega^1_{\tilde{X}}\to R^1\pi_*\Omega^2_{\tilde{X}}\Bigr) \ar[d]^{{_{\tilde{X}}d}^{1,1}_2} \\
0 \ar[r]^{{_{W,F}d}^{2,-1}_2} & P^3_\refl(\sg{X}{p})
}\end{array}\]
so that $\delta=\psi^{-1}\circ(0,\id)\circ\alpha$ indeed factors as claimed. It is a standard fact that $\alpha$ satisfies the property stated in Item~(\ref{it-lem-key-kan-comm}).
\end{proof}

\subsection{Proof of Theorem~\ref{thm-kan-poincare}}\label{ssec-kan-proof}

The following lemma has already been observed for log canonical singularities in \cite[Lem.~4.11]{GKP12}.

\begin{lem}\label{lem-kan-local-interpretation}
Let $V$ be a normal projective complex variety of dimension $n\geq 4$ with only isolated singularities and let $\LB$ be an ample line bundle.

Then there exists a natural isomorphism
\[H^2(V,\Omega^{[1]}_V\otimes \LB^{-1})\xrightarrow{\sim} H^0(V,R^1\pi_*\Omega^1_{\tilde{V}}(log\, E)\otimes\LB^{-1}) \]
for any strong resolution $\pi:\tilde{V}\to V$, where $E\subset\tilde{V}$ is the exceptional divisor with its reduced scheme structure.
\end{lem}

\begin{proof}
The Leray spectral sequence for the sheaf $\Omega^1_{\tilde{V}}(log\, E)\otimes\pi^*\LB^{-1}$ on $\tilde{V}$ and the morphism $\pi$ starts off with
\[E^{i,j}_2=H^j(V,R^i\pi_*(\Omega^1_{\tilde{V}}(log\, E))\otimes\LB^{-1})\implies E^{i+j}_\infty=H^{i+j}(\tilde{V},\Omega^1_{\tilde{V}}(log\, E)\otimes\pi^*\LB^{-1}).\]
Steenbrink's vanishing theorem \cite[Thm.~2a')]{Steen85} states that $E^{m}_\infty=0$ for $m<n$. Since $n\geq 4$ the five-term exact sequence associated with the Leray spectral sequence shows that the boundary map of the $E_2$-page gives an isomorphism
\[H^0(V,R^1\pi_*(\Omega^1_{\tilde{V}}(log\, E))\otimes\LB^{-1})=E^{0,1}_0 \xrightarrow{\sim} E^{2,0}_2 =  H^2(V,\pi_*\Omega^1_{\tilde{V}}(log\, E)\otimes\LB^{-1}).\]
This proves the lemma, since $\Omega^{[1]}_V=\pi_*\Omega^1_{\tilde{V}}(log\, E)$ by \cite[Thm.~(1.3)]{SvS85}.
\end{proof}

\begin{lem}\label{lem-kan-weil-divisors}
Let $X$ be a reduced complex space with an isolated Du Bois singularity $p\in X$ of dimension $n\geq 4$ satisfying Assumption~\ref{ass-du-bois}. Let $\pi:\tilde{X}\to X$ be a strong resolution and let $E=\sum_rE_r$ be the exceptional divisor with its reduced scheme structure. Then
\[\dim_\bC H^2(E,\bC)/{\textstyle \sum_r}\langle E_r|_E\rangle = \dim_\bQ\text{WDiv}_\bQ(\sg{X}{p})/\hspace{-0.2em}\sim_\bQ.\]
\end{lem}

\begin{proof}
By Fact~\ref{fact-top-reso} there exist arbitrarily small neighborhoods $U$ of $p\in X$ so that $E\to\tilde{U}:=\pi^{-1}(U)$ is a homotopy equivalence. The exponential sequences on $\tilde{U}$ and $E$ give rise to a commutative diagram
\[\begin{array}{l}\xymatrix{
H^1(\tilde{U},\sO_{\tilde{U}}) \ar[d]\ar[r] & \text{Pic}(\tilde{U}) \ar[r]\ar[d] & H^2(\tilde{U},\bZ) \ar[d]^{\cong}\ar[r] & H^2(\tilde{U},\sO_{\tilde{U}}) \ar[d] \\
H^1(E,\sO_E) \ar[r] & \text{Pic(E)} \ar[r] & H^2(E,\bZ) \ar[r] & H^2(E,\sO_E).
}\end{array}\]
By Lemma~\ref{lem-key-du-bois-vanishing}, the outer terms vanish if we take the limit over all such $U$. This establishes the isomorphism $\varinjlim_U\text{Pic}(\tilde{U})\cong H^2(E,\bZ)$. The lemma now follows from the exact push forward sequence
\[{\textstyle \sum_r\bQ\cdot[E_r]\to\varinjlim_U\text{Pic}(\tilde{U})\otimes\bQ\to\text{WDiv}_\bQ(\sg{X}{p})/\hspace{-0.2em}\sim_\bQ\to 0}\]
for analytic divisors with coefficients in $\bQ$.
\end{proof}

\begin{proof}[Proof of Theorem~\ref{thm-kan-poincare}]
Let $\pi:\tilde{X}\to X$ be a strong resolution of $X$ where $\tilde{X}$ is a projective manifold. For any $p\in X_\singu$ let $E^p=\sum_{r\in R_p}E^p_r\subset \tilde{X}$ denote the reduced fiber over $p\in X$. We can calculate using relative GAGA in the first step
\begin{align*}
  \dim_\bC\,H^2(X,\Omega^{[1]}_X\otimes\LB^{-1}) &\,\stackbin[]{ \ref{lem-kan-local-interpretation}}{=} \, \sum_{p\in X_\singu}\dim_\bC\bigl(R^1\pi_*\Omega^1_{\tilde{X}}(\text{log}\,E^p)\bigr)_p \\
& \,\stackbin[]{ \ref{prop-key-kan}}{\geq} \, \sum_{p\in X_\singu} {\textstyle \Bigl(\dim_\bC P^3_\refl(\sg{X}{p}) + \dim_\bC H^2(E^p,\bC)/\sum_{r\in R_p}\langle E^p_r|_{E^p}\rangle\Bigr)}  \\
& \,\stackbin[]{ \ref{lem-kan-weil-divisors}}{=} \,  \sum_{p\in X_\singu}{\textstyle \Bigl(\dim_\bC P^3_\refl(\sg{X}{p}) + \dim_\bQ\text{WDiv}_\bQ(\sg{X}{p})/\hspace{-0.2em}\sim_\bQ\Bigr),}  
\end{align*}
which proves the first claim of Theorem~\ref{thm-kan-poincare}. The second claim follows from Lemma~\ref{lem-dubois-comparison}.
\end{proof}

\subsection{Proof of Proposition~\ref{prop-kan-optimality}}\label{ssec-prop-kan-optimality}
We maintain the notation and assumptions of Example~\ref{ex-kan-poincare}. We denote the blowing-up in the vertex by $\pi:\tilde{X}\to X$. The projective complex manifold $E$ will be considered as the exceptional divisor of $\pi$ and $\sI_E\subset\sO_{\tilde{X}}$ is its ideal sheaf.

We split the proof into several claims proven separately.

\begin{plainclaim}\label{claim-du-bois-optimality-db}
The singularity $p\in X$ is Du Bois.
\end{plainclaim}

\begin{proof}[Proof of Claim~\ref{claim-du-bois-optimality-db}]
Using Fact~\ref{fact-isolated-du-bois} it suffices to show that $R^i\pi_*\sO_{\tilde{X}}\to H^i(E,\sO_E)$ is bijective for $i>0$. This is equivalent to $R^i\pi_*\sI_E=0$ for $i>0$. The last assertion is a consequence of the formal function theorem, Condition~\ref{ex-kan-poincare}(1) and $\sI_E^k/\sI_E^{k+1}\cong\sM^{k}$.
\end{proof}

\begin{plainclaim}\label{claim-du-bois-wdiv-infty}
For any $n\geq 4$, we have $\dim_\bQ\text{WDiv}_\bQ(\sg{X}{p})/\hspace{-0.2em}\sim_\bQ=\infty$.
\end{plainclaim}

\begin{proof}[Proof of Claim~\ref{claim-du-bois-wdiv-infty}]
As in the proof of Lemma~\ref{lem-kan-weil-divisors} there exist arbitrarily small neighborhoods $U\subset\tilde{X}$ of $p$ together with commutative diagrams
\[\begin{array}{l}\xymatrix{
H^1(\tilde{U},\bZ) \ar[d]^{\cong}\ar[r] & H^1(\tilde{U},\sO_{\tilde{U}}) \ar[d]^{\text{res}^1_U}\ar[r] & \text{Pic}(\tilde{U}) \ar[r]\ar[d] & H^2(\tilde{U},\bZ) \ar[d]^{\cong}\ar[r] & H^2(\tilde{U},\sO_{\tilde{U}}) \ar[d]^{\text{res}^2_U} \\
H^1(E,\bZ) \ar[r] & H^1(E,\sO_E) \ar[r] & \text{Pic(E)} \ar[r] & H^2(E,\bZ) \ar[r] & H^2(E,\sO_E).
}\end{array}\]
Claim~\ref{claim-du-bois-optimality-db} and Fact~\ref{fact-isolated-du-bois} together imply that $\varinjlim_U\text{res}^i_U:R^i\pi_*\sO_{\tilde{X}}\xrightarrow[]{\sim} H^i(E,\sO_E)$ is an isomorphism for $i>0$ so that $\varinjlim_U\text{Pic}(\tilde{U})\cong \text{Pic}(E)$ by the five lemma. The second row in the diagram then immediately shows that $\varinjlim_U\text{Pic}(\tilde{U})\otimes\bQ$ is of infinite dimension over $\bQ$. Now the claim follows from the push forward sequence
\[{\textstyle \bQ\xrightarrow[]{\cdot [E]}\varinjlim_U\text{Pic}(\tilde{U})\otimes\bQ\to\text{WDiv}_\bQ(\sg{X}{p})/\hspace{-0.2em}\sim_\bQ\to 0.}\]
for analytic divisors with coefficients in $\bQ$.
\end{proof}

\begin{plainclaim}\label{claim-du-bois-optimality-kan}
For any ample line bundle $\LB$ on $X$ we have 
\[\dim_\bC H^2(X,\Omega^{[1]}_X\otimes\LB^{-1}) \,\leq \, n\cdot(n-1).\]
\end{plainclaim}

\begin{proof}[Proof of Claim~\ref{claim-du-bois-optimality-kan}]
The residue sequence for $1$-forms with logarithmic poles along $E$ is
\[0\to \Omega^1_{E}\to\Omega^1_{\tilde{X}}(log\, E)\otimes\sO_E\to \sO_E\to 0. \]
Tensoring with $\sI_E^k/\sI_E^{k+1}\cong\sM^{k}$ for $k>0$ and $\Omega^1_E\cong \bigoplus_{i=1}^{n-1}\sO_E$ yield
\[ {\textstyle 0\to \bigoplus_{i=1}^{n-1}\sM^{k}\to\Omega^1_{\tilde{X}}(log\, E)\otimes\sI_E^k/\sI_E^{k+1}\to \sM^{k}\to 0.}\]
The higher cohomology groups on the right and left hand side vanish by Condition~\ref{ex-kan-poincare}(1) if $k>0$. This implies that
\[H^i(E,\Omega^1_{\tilde{X}}(log\, E)\otimes\sI_E^k/\sI_E^{k+1})=0\]
for $i>0$ and $k>0$. In particular, using the formal function theorem, we calculate
\[R^i\pi_*\Omega^1_{\tilde{X}}(log\, E) \cong \stackbin[\leftarrow]{}{\text{lim}} \,H^1\bigl(E,\Omega^1_{\tilde{X}}(log\, E)\otimes \sO_{\tilde{X}}/\sI_E^k\bigr) \cong  H^1\bigl(E,\Omega^1_{\tilde{X}}(log\, E)|_E\bigr).\]
Then the short sequence $0\to\Omega^1_E\to\Omega^1_{\tilde{X}}(\log E)|_E\to\sO_E\to 0$ implies that
\[\dim_\bC \Bigl(R^i\pi_*\Omega^1_{\tilde{X}}(log\, E)\Bigr)_p\leq h^{1,1}(E)+ h^{0,1}(E)=n\cdot(n-1).\]
The claim follows from Lemma~\ref{lem-kan-local-interpretation}.
\end{proof}

\begin{plainclaim}\label{claim-du-bois-optimality-poincare}
$\text{dim}_\bC P^3_\refl(\sg{X}{p})={n-1 \choose 3}.$
\end{plainclaim}

\begin{proof}[Proof of Claim~\ref{claim-du-bois-optimality-poincare}]
The groups $P^i_\h(\sg{X}{p})$ vanish for $i>0$ by Theorem~\ref{thm-h-poincare-contraction}. Moreover, for $i>0$, because $\pi$ is the blowing-up of the vertex of a cone singularity, there exists a short exact sequence 
\[0\to \Omega^i_\h|_{X,p}\to\bigl(\pi_*\Omega^i_{\tilde{X}}\bigr)_p\to H^i(E,\Omega^i_E)\to 0\]
of stalks of sheaves at $p\in X$, see also Proposition~\ref{prop-h-isolated}. Now Lemma~\ref{lem-dubois-comparison}(\ref{it-fk-to-refl}) implies that
\[P^3_\refl(\sg{X}{p})\cong H^3(\bigl(\pi_*\Omega^\bullet_{\tilde{X}}\bigr)_p) \cong H^0(E,\Omega^3_E)\]
is of dimension ${n-1\choose 3}$, since $\Omega^1_E\cong\bigoplus_{i=1}^{n-1}\sO_E$.
\end{proof}

\begin{proof}[Proof of Proposition~\ref{prop-kan-optimality}]
The proposition follows from the preparatory claims proven above. We only need to observe that $n\cdot (n-1)<{n-1 \choose 3}$ for $n\geq 11$.
\end{proof}

%% file: kopf.bbl
\def\cprime{$'$}
\providecommand{\bysame}{\leavevmode\hbox to3em{\hrulefill}\thinspace}
\providecommand{\MR}{\relax\ifhmode\unskip\space\fi MR}
% \MRhref is called by the amsart/book/proc definition of \MR.
\providecommand{\MRhref}[2]{%
  \href{http://www.ams.org/mathscinet-getitem?mr=#1}{#2}
}
\providecommand{\href}[2]{#2}
\begin{thebibliography}{GKKP11}

\bibitem[SGA71]{SGA1}
\emph{Rev\^etements \'etales et groupe fondamental}, Springer-Verlag, Berlin,
  1971, S{\'e}minaire de G{\'e}om{\'e}trie Alg{\'e}brique du Bois Marie
  1960--1961 (SGA 1), Dirig{\'e} par Alexandre Grothendieck. Augment{\'e} de
  deux expos{\'e}s de M. Raynaud, Lecture Notes in Mathematics, Vol. 224.
  {\sf\scriptsize 0354651 (50 \#7129)}

\bibitem[Bor84]{Bor84}
\emph{Intersection cohomology}, Progress in Mathematics, vol.~50, Birkh\"auser
  Boston Inc., Boston, MA, 1984, Notes on the seminar held at the University of
  Bern, Bern, 1983, Swiss Seminars. {\sf\scriptsize 788171 (88d:32024)}

\bibitem[BBD82]{BBD81}
{\sc A.~A. Be{\u\i}linson, J.~Bernstein, and P.~Deligne}: \emph{Faisceaux
  pervers}, Analysis and topology on singular spaces, {I} ({L}uminy, 1981),
  Ast\'erisque, vol. 100, Soc. Math. France, Paris, 1982, pp.~5--171.
  {\sf\scriptsize 751966 (86g:32015)}

\bibitem[BB73]{Bia73}
{\sc A.~Bia{\l}ynicki-Birula}: \emph{Some theorems on actions of algebraic
  groups}, Ann. of Math. (2) \textbf{98} (1973), 480--497. {\sf\scriptsize
  0366940 (51 \#3186)}

\bibitem[CF02]{CF02}
{\sc F.~Campana and H.~Flenner}: \emph{Contact singularities}, Manuscripta
  Math. \textbf{108} (2002), no.~4, 529--541. {\sf\scriptsize 1923538
  (2003k:32041)}

\bibitem[CLS11]{CLS11}
{\sc D.~A. Cox, J.~B. Little, and H.~K. Schenck}: \emph{Toric varieties},
  Graduate Studies in Mathematics, vol. 124, American Mathematical Society,
  Providence, RI, 2011. {\sf\scriptsize 2810322 (2012g:14094)}

\bibitem[Dan78]{Dan78}
{\sc V.~I. Danilov}: \emph{The geometry of toric varieties}, Uspekhi Mat. Nauk
  \textbf{33} (1978), no.~2(200), 85--134, 247. {\sf\scriptsize 495499
  (80g:14001)}

\bibitem[Del71]{Del71}
{\sc P.~Deligne}: \emph{Th\'eorie de {H}odge. {II}}, Inst. Hautes \'Etudes Sci.
  Publ. Math. (1971), no.~40, 5--57. {\sf\scriptsize 0498551 (58 \#16653a)}

\bibitem[Dim04]{Dim04}
{\sc A.~Dimca}: \emph{Sheaves in topology}, Universitext, Springer-Verlag,
  Berlin, 2004. {\sf\scriptsize 2050072 (2005j:55002)}

\bibitem[DY10]{DY10}
{\sc R.~Du and S.~Yau}: \emph{Local holomorphic de {R}ham cohomology}, Comm.
  Anal. Geom. \textbf{18} (2010), no.~2, 365--374. {\sf\scriptsize 2672237
  (2011g:32039)}

\bibitem[Dur95]{D95}
{\sc A.~H. Durfee}: \emph{Intersection homology {B}etti numbers}, Proc. Amer.
  Math. Soc. \textbf{123} (1995), no.~4, 989--993. {\sf\scriptsize 1233968
  (95e:14014)}

\bibitem[Fer70]{Ferr70}
{\sc A.~Ferrari}: \emph{Cohomology and holomorphic differential forms on
  complex analytic spaces}, Ann. Scuola Norm. Sup. Pisa (3) \textbf{24} (1970),
  65--77. {\sf\scriptsize 0274810 (43 \#570)}

\bibitem[Gil64]{Gil64}
{\sc M.~C. Gilmartin}: \emph{Every analytic variety is locally contractible},
  ProQuest LLC, Ann Arbor, MI, 1964, Thesis (Ph.D.)--Princeton University.
  {\sf\scriptsize 2614525}

\bibitem[GR71]{GR71}
{\sc H.~Grauert and R.~Remmert}: \emph{Analytische {S}tellenalgebren},
  Springer-Verlag, Berlin, 1971, Unter Mitarbeit von O. Riemenschneider, Die
  Grundlehren der mathematischen Wissenschaften, Band 176. {\sf\scriptsize
  0316742 (47 \#5290)}

\bibitem[Gra62]{Grau62}
{\sc H.~Grauert}: \emph{\"{U}ber {M}odifikationen und exzeptionelle analytische
  {M}engen}, Math. Ann. \textbf{146} (1962), 331--368. {\sf\scriptsize 0137127
  (25 \#583)}

\bibitem[GKP13]{GKP12}
{\sc D.~{Greb}, S.~{Kebekus}, and T.~{Peternell}}: \emph{{Reflexive
  differential forms on singular spaces -- Geometry and Cohomology}}, Journal
  f\"ur die Reine und Angewandte Mathematik (Crelle's Journal), published
  electronically (2013).

\bibitem[GKKP11]{GKKP11}
{\sc D.~Greb, S.~Kebekus, S.~J. Kov{\'a}cs, and T.~Peternell}:
  \emph{Differential forms on log canonical spaces}, Publ. Math. Inst. Hautes
  \'Etudes Sci. (2011), no.~114, 87--169. {\sf\scriptsize 2854859}

\bibitem[Gre75]{G75}
{\sc G.-M. Greuel}: \emph{Der {G}auss-{M}anin-{Z}usammenhang isolierter
  {S}ingularit\"aten von vollst\"andigen {D}urchschnitten}, Math. Ann.
  \textbf{214} (1975), 235--266. {\sf\scriptsize 0396554 (53 \#417)}

\bibitem[Gre80]{G80}
{\sc G.-M. Greuel}: \emph{Dualit\"at in der lokalen {K}ohomologie isolierter
  {S}ingularit\"aten}, Math. Ann. \textbf{250} (1980), no.~2, 157--173.
  {\sf\scriptsize 582515 (82e:32009)}

\bibitem[Ham71]{H71}
{\sc H.~Hamm}: \emph{Lokale topologische {E}igenschaften komplexer {R}\"aume},
  Math. Ann. \textbf{191} (1971), 235--252. {\sf\scriptsize 0286143 (44
  \#3357)}

\bibitem[Hoc77]{Hoch77}
{\sc M.~Hochster}: \emph{The {Z}ariski-{L}ipman conjecture in the graded case},
  J. Algebra \textbf{47} (1977), no.~2, 411--424. {\sf\scriptsize 0469917 (57
  \#9697)}

\bibitem[HJ13]{HJ13}
{\sc A.~{Huber} and C.~{J{\"o}rder}}: \emph{{Differential forms in the
  h-topology}}, ArXiv e-prints (2013).

\bibitem[KK83]{KK83}
{\sc L.~Kaup and B.~Kaup}: \emph{Holomorphic functions of several variables},
  de Gruyter Studies in Mathematics, vol.~3, Walter de Gruyter \& Co., Berlin,
  1983, An introduction to the fundamental theory, With the assistance of
  Gottfried Barthel, Translated from the German by Michael Bridgland.
  {\sf\scriptsize 716497 (85k:32001)}

\bibitem[Keb13]{Keb12}
{\sc S.~Kebekus}: \emph{Pull-back morphisms for reflexive differential forms},
  Adv. Math. \textbf{245} (2013), 78--112. {\sf\scriptsize 3084424}

\bibitem[Kol07]{Koll07}
{\sc J.~Koll{\'a}r}: \emph{Lectures on resolution of singularities}, Annals of
  Mathematics Studies, vol. 166, Princeton University Press, Princeton, NJ,
  2007. {\sf\scriptsize 2289519 (2008f:14026)}

\bibitem[KM98]{KM98}
{\sc J.~Koll{\'a}r and S.~Mori}: \emph{Birational geometry of algebraic
  varieties}, Cambridge Tracts in Mathematics, vol. 134, Cambridge University
  Press, Cambridge, 1998, With the collaboration of C. H. Clemens and A. Corti,
  Translated from the 1998 Japanese original. {\sf\scriptsize 1658959
  (2000b:14018)}

\bibitem[Kov99]{Kov99}
{\sc S.~J. Kov{\'a}cs}: \emph{Rational, log canonical, {D}u {B}ois
  singularities: on the conjectures of {K}oll\'ar and {S}teenbrink}, Compositio
  Math. \textbf{118} (1999), no.~2, 123--133. {\sf\scriptsize 1713307
  (2001g:14022)}

\bibitem[Lau73]{L73}
{\sc H.~B. Laufer}: \emph{Taut two-dimensional singularities}, Math. Ann.
  \textbf{205} (1973), 131--164. {\sf\scriptsize 0333238 (48 \#11563)}

\bibitem[Lau77]{Lau77}
{\sc H.~B. Laufer}: \emph{On minimally elliptic singularities}, Amer. J. Math.
  \textbf{99} (1977), no.~6, 1257--1295. {\sf\scriptsize 0568898 (58 \#27961)}

\bibitem[Lau81]{L81}
{\sc H.~B. Laufer}: \emph{On {${\bf C}P^{1}$} as an exceptional set}, Recent
  developments in several complex variables ({P}roc. {C}onf., {P}rinceton
  {U}niv., {P}rinceton, {N}. {J}., 1979), Ann. of Math. Stud., vol. 100,
  Princeton Univ. Press, Princeton, N.J., 1981, pp.~261--275. {\sf\scriptsize
  627762 (82m:32012)}

\bibitem[Lee09]{L09}
{\sc B.~Lee}: \emph{Local acyclic fibrations and the de {R}ham complex},
  Homology, Homotopy Appl. \textbf{11} (2009), no.~1, 115--140. {\sf\scriptsize
  2506129 (2010j:14034)}

\bibitem[Lip65]{Lip65}
{\sc J.~Lipman}: \emph{Free derivation modules on algebraic varieties}, Amer.
  J. Math. \textbf{87} (1965), 874--898. {\sf\scriptsize 0186672 (32 \#4130)}

\bibitem[Loj64]{Loj64}
{\sc S.~Lojasiewicz}: \emph{Triangulation of semi-analytic sets}, Ann. Scuola
  Norm. Sup. Pisa (3) \textbf{18} (1964), 449--474. {\sf\scriptsize 0173265 (30
  \#3478)}

\bibitem[Loo84]{Loo84}
{\sc E.~J.~N. Looijenga}: \emph{Isolated singular points on complete
  intersections}, London Mathematical Society Lecture Note Series, vol.~77,
  Cambridge University Press, Cambridge, 1984. {\sf\scriptsize 747303
  (86a:32021)}

\bibitem[Mum61]{Mu61}
{\sc D.~Mumford}: \emph{The topology of normal singularities of an algebraic
  surface and a criterion for simplicity}, Inst. Hautes \'Etudes Sci. Publ.
  Math. (1961), no.~9, 5--22. {\sf\scriptsize 0153682 (27 \#3643)}

\bibitem[Nam01]{Nam01}
{\sc Y.~Namikawa}: \emph{Deformation theory of singular symplectic
  {$n$}-folds}, Math. Ann. \textbf{319} (2001), no.~3, 597--623.

\bibitem[Pin77]{P77}
{\sc H.~Pinkham}: \emph{Normal surface singularities with {$C\sp*$} action},
  Math. Ann. \textbf{227} (1977), no.~2, 183--193. {\sf\scriptsize 0432636 (55
  \#5623)}

\bibitem[Rei67]{Rei67}
{\sc H.-J. Reiffen}: \emph{Das {L}emma von {P}oincar\'e f\"ur holomorphe
  {D}ifferential-formen auf komplexen {R}\"aumen}, Math. Z. \textbf{101}
  (1967), 269--284. {\sf\scriptsize 0223599 (36 \#6647)}

\bibitem[SW81]{SW80}
{\sc G.~Scheja and H.~Wiebe}: \emph{Zur {C}hevalley-{Z}erlegung von
  {D}erivationen}, Manuscripta Math. \textbf{33} (1980/81), no.~2, 159--176.
  {\sf\scriptsize 597817 (82d:13008)}

\bibitem[SvS85]{SvS85}
{\sc J.~Steenbrink and D.~van Straten}: \emph{Extendability of holomorphic
  differential forms near isolated hypersurface singularities}, Abh. Math. Sem.
  Univ. Hamburg \textbf{55} (1985), 97--110. {\sf\scriptsize 831521
  (87j:32025)}

\bibitem[Ste85]{Steen85}
{\sc J.~H.~M. Steenbrink}: \emph{Vanishing theorems on singular spaces},
  Ast\'erisque (1985), no.~130, 330--341, Differential systems and
  singularities (Luminy, 1983). {\sf\scriptsize 804061 (87j:14026)}

\bibitem[Voe96]{Voe96}
{\sc V.~Voevodsky}: \emph{Homology of schemes}, Selecta Math. (N.S.) \textbf{2}
  (1996), no.~1, 111--153. {\sf\scriptsize 1403354 (98c:14016)}

\end{thebibliography}
